\documentclass[12pt]{amsart}
 \newtheorem{thm}{Theorem}[section]
\newtheorem{prop}[thm]{Proposition}
\newtheorem{lem}[thm]{Lemma}
\newtheorem{cor}[thm]{Corollary}
\theoremstyle{definition}

\theoremstyle{remark}
\newtheorem{rmk}[thm]{Remark}

\usepackage{amsfonts}
\usepackage{amsmath}
\usepackage{calligra}
\usepackage{amssymb}
\usepackage{verbatim}
\usepackage[colorlinks]{hyperref}
\usepackage[T1]{fontenc}
\usepackage{mathrsfs}
\usepackage{color}
\usepackage[utf8]{inputenc}
\usepackage[makeroom]{cancel}
\usepackage{tikz-cd}
\input xy
\xyoption{all}

  \usepackage{mathtools}

\numberwithin{equation}{section}
\usepackage{etoolbox}

\begin{document}

\title{ Pseudo-dual pairs and branching of Discrete Series}
\author{  Bent {\O}rsted,   Jorge A.  Vargas}
\thanks{Partially supported by Aarhus University  (Denmark),  CONICET
(Argentina)  }
\date{\today }
\address{ Mathematics Department, Aarhus University, Denmark; FAMAF-CIEM, Ciudad Universitaria, 5000 C\'ordoba, Argentine}
\email{orsted@imf.au.dk,   vargas@famaf.unc.edu.ar}
\maketitle

\markboth{{\O}rsted- Vargas}{Duality}

\maketitle
{\centering\footnotesize  We are delighted to make this paper part of a tribute to Toshiyuki Kobayashi,\\ for his
indefatigable dedication to representation theory\\ and Lie theory, and his many new ideas here.\par}


 \texttt{Abstract}{ For a semisimple Lie group $G$, we study Discrete Series representations with admissible branching
to a symmetric subgroup $H$. This is done using a canonical associated symmetric subgroup $H_0$,
forming a pseudo-dual pair with $H$, and a corresponding branching law for this group with respect to
its maximal compact subgroup. This is in analogy with either Blattner's or  Kostant-Heckman multiplicity formulas, and
has some resemblance to Frobenius reciprocity. We give several explicit examples and links to
Kobayashi-Pevzner theory of symmetry breaking and holographic operators. Our method is
well adapted to computer algorithms, such as for example the Atlas program.}

\tableofcontents

\section{Introduction} For a semisimple Lie group $G$, an irreducible representation $(\pi, V)$ of $G$ and closed reductive subgroup $H\subset G$ the problem of decomposing the restriction of $\pi$ to $H$ has received attention ever since number theory or physics and other branches of mathematics required a solution. In this paper, we are concerned with the important particular case of branching representations of the Discrete Series, i.e. those $\pi$ arising as closed irreducible subspaces of the left regular representation in $L^2(G)$, and breaking the symmetry by a reductive subgroup $H$. Here much work has been done. Notable is the paper of Gross-Wallach, \cite{GW}, and the work of Toshiyuki Kobayashi and his school. For further references on the subject, we refer to the overview work of Toshiyuki Kobayashi and references therein. The aim of this note is to compute the decomposition of the restriction of an $H$-admissible representation  $\pi$ to a symmetric subgroup $H$ (see~\ref{sec:gwmultip}), in \cite{GW} it is derive a duality Theorem for Discrete Series representation.   Their duality is based on the dual subgroup $G^d$ (this is
the dual subgroup which enters the duality introduced by Flensted-Jensen in his
study of discrete series for affine symmetric spaces \cite{FG}) and, roughly speaking, their formula looks like \begin{equation*} \dim Hom_H(\sigma, \pi_{\vert_H})=\dim Hom_{\widetilde K}(F_\sigma, \tilde\pi). \end{equation*} Here, $\pi$ is a irreducible square integrable representation of $G$, $\sigma $ is a irreducible representation of $H$, $F_\sigma$ is a irreducible representation of a maximal compact subgroup $\widetilde{K}$ of $G^d$, and $\tilde\pi$ is a finite sum of fundamental representations of $G^d$ attached to $\pi$. In \cite{os}, B. Speh and the first author considered  a different duality Theorem for restriction to a symmetric subgroup, let $H_0$ the associated subgroup to $H$ (see 3.1), hence, $L:=H_0\cap H$ is a maximal compact subgroup of both $H, H_0$. Then,  \begin{equation*} \dim Hom_H(\sigma, \pi_{\vert_H})=\dim  Hom_{L}(\sigma_0, \widetilde\Pi)\tag{$\ddag$}. \end{equation*} Here, $\pi$ is certain irreducible  representation of $G$, $\sigma $ is a irreducible representation of $H$, $\sigma_0$ is the lowest $L$-type of $\sigma$ and $\widetilde\Pi$ is a finite sum of irreducible representation of $H_0$ attached to $\pi$. The purpose of this paper is, for a $H$-admissible Discrete Series $\pi$ for $G$,  to show a formula as the above  and to provide an explicit isomorphism between the two vector spaces involved in the equality. This is embodied in Theorem~\ref{prop:Rwelldefgc}.

Theorem~\ref{prop:Rwelldefgc} reduces the branching law  in
two steps (1) For the maximal compact subgroup $K$ of $G$ and the lowest $K$-type
of $\pi$, branching this under $L$ (maximal compact in $H$ and also in $H_0$) (2)
branching a Discrete Series of $H_0$ with respect to $L$, i.e. finding its $L$-types with
multiplicity. Both of these steps can be implemented in algorithms, as they are available
for example in the computer program Atlas,  http://atlas.math.umd.edu.

  We would like to point out that T. Kobayashi, T. Kobayashi-Pevzner and Nakahama have shown a duality formula as $(\ddag)$ for holomorphic Discrete Series representation $\pi$. In order to achieve their result, they have shown an explicit isomorphism between the two vector spaces in the formula. Further, with respect to analyze $res_H(\pi)$,  Kobayashi-Oshima have shown a way to compute the irreducible components of $res_H(\pi)$ in the language of Zuckerman modules $A_\mathfrak q (\lambda)$ \cite{KO}\cite{KO2}.

As a consequence of the involved material, we obtain a necessary and sufficient condition  for a symmetry breaking operators to be represented via normal derivatives. This is presented in Proposition~\ref{prop:symmenormal}.

Another consequence is  Proposition~\ref{prop:gentau}. That is,  for    the closure of the linear span of the totality of $H_0$-translates (resp. $H$-translates) of the isotypic component associated to  the lowest $K$-type of $\pi$, we exhibit its explicit decomposition as a finite sum of Discrete Series representations of $H_0$ (resp. $H$).

Our proof is  based on the fact that Discrete Series representations are realized in reproducing kernel Hilbert spaces. As a consequence, in Lemma~\ref{lem:injecrh},  we obtain a general result on the structure of the kernel of a certain restriction map. The proof also relies on the work of Hecht-Schmid \cite{HS}, and a result of Schmid in \cite{Sc}.

It follows from the work of Kobayashi-Oshima, also, from Tables 1,2,3, that whenever a Discrete Series for $G$ has an    admissible restriction to a symmetric subgroup, then, the infinitesimal character of the representation is  dominant with respect  to either a Borel de Siebenthal system of positive roots or to a system of positive roots so that it has two noncompact simple roots, each  one, has multiplicity one in the highest root. Under the $H$-admissible hypothesis, the infinitesimal character of each of the irreducible components of $\widetilde{\Pi}$ in formula $(\ddag)$, has the same property as the infinitesimal character of $\pi$. Thus, for most $H$-admissible Discrete series,   to compute the right hand side of $(\ddag)$, we may appeal to the work of the first author and Wolf \cite{ow}. Their results let us   compute the highest weight of each irreducible factor in the restriction of $\pi$ to $K_1(\Psi)$. Next, we apply  \cite[Theorem 5]{DV} for the general case.

We may speculate that a formula like $(\ddag)$ might be true for $\pi$ whose underlying Harish-Chandra module is equivalent to a  unitarizable Zuckerman module. In this case, the definition of $\sigma_0$ would be the subspace spanned by the lowest $L$-type of $\sigma$ and $\widetilde \Pi$ would be a Zuckerman module attached to the lowest $K$-type of $\pi$.

The paper is organized as follows.  In Section 2, we introduce facts about Discrete Series  representation and   notation. In Section 3,
we state the main Theorem and begin its proof. As a tool, we obtain information on the kernel of the restriction map.

In Section 4, we complete the proof of the main Theorem. As a by-product, we obtain information on the kernel of the restriction map, under admissibility hypothesis.   We present examples and applications of the Main Theorem   in section 5. This includes  lists of multiplicity free restriction of representations, many  of the multiplicity free representations are non holomorphic Discrete Series representations. We also dealt with quaternionic and generalized quaternionic representations.

In Section 6, we analyze when symmetry breaking operators are represented by means of  normal derivatives. Section 7 presents  the list of $H$-admissible Discrete Series and related information.

\medskip

 {\it Acknowledgements:}
The  authors would like to thank T. Kobayashi for much insight  and inspiration on the problems considered here. Also, we thank Michel Duflo, Birgit Speh, Yosihiki Oshima and Jan Frahm for conversations on the subject.
Part of the research in this paper was carried out within the online research community on Representation Theory and Noncommutative Geometry sponsored by the American Institute of Mathematics. Also, some of  the results in this note were the subject of a talk in   the "Conference in honour of Prof. Toshiyuki Kobayashi" to celebrate his  sixtieth birthday, the authors  thank   the organizers for the facilities to present and participate in such a wonderful  meeting via zoom.  Finally, we thank the referees for their truly expert and careful advice in improving the paper.


\smallskip

\section{Preliminaries  and some notation}\label{sec:prelim}

Let $G$ be an arbitrary, matrix,  connected semisimple Lie group. Henceforth, we fix a maximal compact subgroup $K$ for $G$  and a maximal torus $T$ for $K.$   Harish-Chandra showed that $G$ admits square integrable irreducible representations if and only if $T$ is a Cartan subgroup of $G.$  For this paper,  we always assume $T$ is a Cartan subgroup of $G.$ Under these hypothesis, Harish-Chandra  showed that  the set of equivalence classes of irreducible square integrable representations is parameterized by a lattice in $i\mathfrak t^\star.$ In order to state our results we need to make  explicit this parametrization and set up some notation.   As usual,  the Lie algebra of a Lie group is denoted by the corresponding lower case German letter.  To avoid notation, the complexification of  the Lie algebra of a Lie group is  also denoted by the corresponding German letter without any subscript.     $V^\star $ denotes the dual space to a vector space $V.$ Let $\theta$ be the Cartan involution which corresponds to the subgroup $K,$ the associated Cartan decomposition is denoted by $\mathfrak g=\mathfrak k +\mathfrak p.$ Let $\Phi(\mathfrak g,\mathfrak t) $ denote the root system attached to the Cartan subalgebra $\mathfrak t.$ Hence, $\Phi(\mathfrak g,\mathfrak t)=\Phi_c \cup \Phi_n =\Phi_c(\mathfrak g, \mathfrak t) \cup \Phi_n (\mathfrak g, \mathfrak t)$ splits up as the union the set of compact roots and the set of noncompact roots. From now on, we fix a system of positive roots $\Delta $ for $\Phi_c.$   For this paper, either the highest weight or the infinitesimal character of an irreducible representation of $K$ is  dominant with respect to $\Delta.$ The Killing form gives rise to an inner product $(...,...)$ in $i\mathfrak t^\star.$ As usual, let $\rho=\rho_G $ denote half of the sum of the roots for some system of positive roots for $\Phi(\mathfrak g, \mathfrak t).$  \textit{A Harish-Chandra parameter} for $G$ is $\lambda \in i\mathfrak t^\star$ such that $(\lambda, \alpha)\not= 0 , $ for every $\alpha \in \Phi(\mathfrak g,\mathfrak t) ,$    and so that $\lambda + \rho$ lifts to a character of $T.$ To each Harish-Chandra parameter $\lambda$, Harish-Chandra, associates a unique irreducible square integrable representation $(\pi_\lambda^G , V_\lambda^G)$ of $G$ of infinitesimal character $\lambda.$ Moreover, he showed the map $\lambda \rightarrow (\pi_\lambda^G, V_\lambda^G)$  is a bijection from the set of Harish-Chandra parameters dominant with respect to $\Delta$  onto the set of equivalence classes of irreducible square integrable representations for $G$ (see \cite[Chap 6]{Wa1}). For short, we will refer to an irreducible square integrable representation as a Discrete Series representation.

 Each Harish-Chandra parameter $\lambda$ gives rise to a system of positive roots \\
\phantom{xxxxxxxxxxxxxxx}$\Psi_\lambda =\Psi_{G,\lambda} =\{  \alpha \in \Phi(\mathfrak g, \mathfrak t) : (\lambda, \alpha) >0 \}.$ \\ From now on, we assume that Harish-Chandra parameter for $G$ are dominant with respect to  $\Delta.$ Whence,  $\Delta \subset \Psi_\lambda.$ We write $\rho_n^\lambda =\rho_n =\frac 12 \sum_{ \beta \in \Psi_\lambda \cap \Phi_n} \beta $, $(\Psi_\lambda)_n:=\Psi_\lambda \cap \Phi_n$.                 We define $\rho_c=\frac12 \sum_{\alpha \in \Delta} \alpha$.

We denote by  $ (\tau ,W ):= (\pi_{\lambda +\rho_n}^K , V_{\lambda + \rho_n}^K) $   the lowest $K-$type of $\pi_\lambda :=\pi_\lambda^G.$ The highest weight of $(\pi_{\lambda +\rho_n}^K , V_{\lambda + \rho_n}^K)$ is $\lambda +\rho_n -\rho_c.$  We recall a Theorem of Vogan's thesis \cite{VoT}\cite{EW} which states that $(\tau,W)$ determines $(\pi_\lambda, V_\lambda^G)$ up  to unitary equivalence.   We recall the set of square integrable sections of the vector bundle determined by the principal bundle $K\rightarrow G \rightarrow G/K$ and the representation $(\tau, W)$ of $K$ is isomorphic to the space \begin{eqnarray*}\lefteqn{L^2(G\times_\tau W)  }\hspace{1.0cm} \\ & & := \{ f \in L^2(G) \otimes W :   f(gk)=\tau(k)^{-1} f(g),   g  \in G, k \in K \}.
   \end{eqnarray*}

   Here, the action of $G$ is by left translation $L_x, x \in G.$   The inner product on $L^2(G)\otimes W$ is given by \begin{equation*}(f,g)_{V_\lambda} =\int_G (f(x),g(x))_W dx, \end{equation*} where $(...,...)_W$ is a $K-$invariant inner product on $W.$
   Subsequently, $L_D $ (resp. $R_D)$ denotes the left infinitesimal (resp. right infinitesimal) action on functions from $G$ of an element  $D$ in universal enveloping algebra $\mathcal U(\mathfrak g)$ for the Lie algebra $\mathfrak g$.  As usual, $\Omega_G$ denotes the Casimir operator for $\mathfrak g.$   Following Hotta-Parthasarathy \cite{ho}, Enright-Wallach \cite{EW}, Atiyah-Schmid \cite{AS}, we realize $V_\lambda :=V_\lambda^G $ as the space
\begin{eqnarray*}
 \lefteqn{ H^2(G, \tau) =\{ f \in L^2(G) \otimes W : f(gk)=\tau(k)^{-1} f(g)} \hspace{3.0cm} \\ & & g\in G, k \in K, R_{\Omega_G} f= [(\lambda, \lambda) -(\rho, \rho)] f  \}.
 \end{eqnarray*}
We also  recall, $R_{\Omega_G}=L_{\Omega_G} $ is an elliptic $G-$invariant operator on the vector bundle $W \rightarrow G \times_\tau W \rightarrow G/K$ and hence,  $\,H^2(G,\tau)$ consists of smooth sections, moreover point evaluation $e_x$ defined by  $ \,H^2(G,\tau) \ni f \mapsto f(x) \in W $ is continuous for each $x \in G$ (cf. \cite[Appendix A4]{OV2}). Therefore, the orthogonal projector $P_\lambda$ onto $\,H^2(G,\tau)$ is an integral map (integral operator) represented by the smooth {\it matrix  kernel} or {\it reproducing kernel}   \cite[ Appendix A1, Appendix A4, Appendix A6]{OV2} \begin{equation} \label{eq:Klambda}K_\lambda : G\times G \rightarrow End_\mathbb C (W) \end{equation} which satisfies $  K_\lambda (\cdot ,x)^\star w$ belongs to $\,H^2(G,\tau)$ for each $x \in G, w \in W$ and $$ (P_\lambda (f)(x), w)_W=\int_G (f(y), K_\lambda (y,x)^\star w)_W dy, \,     f\in L^2(G\times_\tau W).$$
For a closed reductive subgroup  $H$, after conjugation by an inner automorphism of $G$ we may and will assume  $  L:=K\cap H $ is a maximal compact subgroup for $H.$ That is, $H$ is $\theta-$stable. In this paper for irreducible square integrable representations $(\pi_\lambda, V_\lambda)$  for $G $ we would like to analyze its restriction to $H.$ In particular, we study the irreducible $H-$subrepresentations for $\pi_\lambda$. A known fact is that any irreducible $H-$subrepresentation of $V_\lambda$ is a square integrable representation for $H$, for a proof  (cf. \cite{GW}). Thus, owing to the result of Harish-Chandra on the existence of square integrable representations,  from now on, we  may and will assume {\it $H$ admits a compact Cartan subgroup}. After conjugation, we may assume $U :=H\cap T$ is a maximal torus in $L=H\cap K.$ From now on, we set  a square integrable representation  $V_\mu^H\equiv H^2(H, \sigma) \subset L^2 (H \times_\sigma Z)$ of lowest $L-$type $(\pi_{\mu +\rho_n^\mu}^L, V_{\mu+\rho_n^\mu}^L)\equiv :(\sigma, Z)$.

For a representation $M$ and irreducible representation $N$, $M[N]$ denotes the isotypic
 component of $N$, that is, $M[N]$ is  the linear span of the irreducible subrepresentations of $M$ equivalent to $N$. If topology is involved $M[N]$ is the closure of the linear span.

For a $H$-admissible representation $\pi$, $Spec_H(\pi)$, denotes the set of Harish-Chandra parameters of the irreducible $H$-subrepresentations of $\pi$.
\section{Duality Theorem, explicit isomorphism}
\subsection{Statement and proof of the duality  result }
The unexplained notation is as in section~\ref{sec:prelim},  our hypotheses are $(G,H=(G^\sigma)_0)$ is a symmetric pair and $(\pi_\lambda, V_\lambda^G)$ is a $H$-admissible,  square integrable irreducible representation for $G$. $K=G^\theta$ is a maximal compact subgroup of $G$,  $H_0 :=(G^{\sigma \theta})_0$ and $K$ is  so that $L=H\cap K=H_0 \cap K$ is a maximal compact subgroup of both $H$ and $H_0$. By definition, $H_0$ is the {\it associated} subgroup to $H$.

In  this section, under our  hypothesis, for a $H$-irreducible      factor $V_\mu^H$ for $res_H(\pi_\lambda)$,
      we   show  an explicit isomorphism  from the  space \\  \phantom{xxxx} $Hom_H(V_\mu^H, V_\lambda^G)$ onto $Hom_L(V_{\mu +\rho_n^\mu}^L, \,  \pi_\lambda(\mathcal U(\mathfrak h_0))V_\lambda^G[V_{\lambda+\rho_n}^K])$.

We also   analyze   the restriction map $r_0 :H^2(G,\tau)\rightarrow L^2(H_0 \times_\tau W)$.

To follow, we present the necessary definitions and facts involved in the main statement.

 \subsubsection{ } We   consider
              the linear subspace $\mathcal L_\lambda$  spanned by the lowest $L$-type subspace of each irreducible $H$-factor of $res_H((L,H^2(G,\tau)))$. That is, \begin{center} $ \mathcal L_\lambda$ is the linear span of $\cup_{\mu \in Spec_H(\pi_\lambda)} H^2(G,\tau)[V_\mu^H][V_{\mu +\rho_n^\mu}^L]$. \end{center}  We recall that our hypothesis yields that the subspace of $L$-finite vectors in $V_\lambda^G$ is equal to the subspace of $K$-finite vectors \cite[Prop. 1.6 ]{Kob}. Whence,  we have  $\mathcal L_\lambda$ is a subspace of the space of $K$-finite vectors in $H^2(G,\tau)$.

              \subsubsection{} We also need the subspace
              \begin{equation*}    \mathcal U(\mathfrak h_0)W:=L_{\mathcal U(\mathfrak h_0)}H^2(G,\tau)[V_{\lambda+\rho_n^\lambda}^K]\equiv\pi_\lambda(\mathcal U(\mathfrak h_0))(V_\lambda^G[V_{\lambda+\rho_n^\lambda}^K]). \end{equation*}      We write $\mathrm{Cl}(\mathcal U(\mathfrak h_0)W)$ for the closure of $\mathcal U(\mathfrak h_0)W$.  Hence, $\mathrm{Cl}(\mathcal U(\mathfrak h_0)W)$ is the closure of the left translates by the algebra  $\mathcal U(\mathfrak h_0) $ of the subspace of $K$-finite vectors \begin{equation*} H^2(G,\tau)[V_{\lambda +\rho_n^\lambda}^K]= \{ K_\lambda (\cdot, e)^\star w: w \in W   \}\equiv W. \end{equation*} Thus, $\mathcal U(\mathfrak h_0)W$ consists of analytic vectors for $\pi_\lambda$. Therefore,    $\mathrm{Cl}(\mathcal U(\mathfrak h_0)W)$  is invariant under left translations by $H_0$. In Proposition~\ref{prop:gentau} we present the decomposition of $\mathcal U(\mathfrak h_0)W$ as a sum of irreducible representations for $H_0$.

              We point out
              \begin{center}  The  $L$-module   $\mathcal L_\lambda$  is equivalent to the underlying  $L$ -module in
          $\mathcal U(\mathfrak h_0)W$. \end{center}
              This has been proven in   \cite[(4.5)]{Vaint}. For completeness we   present  a proof in Proposition~\ref{prop:gentau}.

            \smallskip
           Under the extra  assumption $res_{L}(\tau)$ is irreducible, we have  $ \mathcal U(\mathfrak h_0)W $ is a irreducible $(\mathfrak h_0,L)$-module, and,  in this case, the  lowest $L$-type  of  $ \mathcal U(\mathfrak h_0)W $ is $(res_L(\tau),W)$. That is, $\mathcal U(\mathfrak h_0)W $ is equivalent to the underlying Harish-Chandra module for  $ H^2(H_0,res_L(\tau))$. The Harish-Chandra parameter $\eta_0 \in i\mathfrak u^\star $ for $\mathrm{Cl}(\mathcal U(\mathfrak h_0)W)$ is computed in \ref{sub:paramhc}.

  For scalar holomorphic Discrete Series,    the classification of the symmetric  pairs $(G,H)$ such that     the equality  $\mathcal U(\mathfrak h_0)W =\mathcal L_\lambda$ holds, is:

 \smallskip
   $ (\mathfrak{su}(m,n), \mathfrak{su} (m,l) +\mathfrak{su} ( n-l)+\mathfrak u(1)),$ $(\mathfrak{so}(2m,2), \mathfrak u(m,1)),$ \\ $(\mathfrak{so}^\star (2n), \mathfrak u(1,n-1)),$ $(\mathfrak{so}^\star(2n),  \mathfrak{so}(2) +\mathfrak{so}^\star(2n-2)),$ $(\mathfrak e_{6(-14)}, \mathfrak{so}(2,8)+\mathfrak{so}(2)).$  \cite[(4.6)]{Vaint}.
Thus,  there exists scalar holomorphic Discrete Series with $\mathcal U(\mathfrak h_0)W \not=\mathcal L_\lambda$.

\subsubsection{} To follow,  we set some more notation. We fix a representative for $(\tau,W)$. We write

\smallskip
  $(res_L(\tau), W)=\sum_{1\leq j \leq r} q_j (\sigma_j, Z_j)$, $q_j=\dim  Hom_L(Z_j, res_L(W))$

\noindent
and the decomposition in isotypic components
$$W=\oplus_{1\leq j \leq r} W[(\sigma_j, Z_j)]=\oplus_{1\leq j \leq r} W[\sigma_j].$$
From now on, we fix respective representatives for $(\sigma_j, Z_j)$ with $Z_j \subset W[(\sigma_j, Z_j)]$.

  Henceforth,  we denote   by $$ \mathbf H^2(H_0, \tau):= \sum_{j }  \,\dim  Hom_L(\tau, \sigma_j) \,  H^2(H_0 , \sigma_j).$$  We think the   later module as   a linear  subspace of $$  \sum_{j} L^2(H_0 \times_{\sigma_j}   W[\sigma_j])_{H_0-disc}  \equiv L^2(H_0 \times_\tau W)_{H_0-disc}.$$ Hence,  $\mathbf H^2(H_0, \tau) \subset L^2(H_0 \times_\tau W)_{H_0-disc}$. We note that when $res_{L}(\tau)$ is irreducible, then $\mathbf H^2(H_0, \tau)=H^2(H_0,res_L(\tau))$.

\subsubsection{} \label{sec:kernelprop} Owing to both spaces $H^2(H,\sigma),H^2(G,  \tau)$ are reproducing kernel spaces,   we represent  each  $T \in Hom_H(H^2(H,\sigma),H^2(G,  \tau))$  by a kernel $K_T :H\times G \rightarrow Hom_\mathbb C(Z,W)$ so that $K_T(\cdot,x)^\star w \in H^2(H,\sigma)$ and $(T(g)(x),w)_W =\int_H (g(h),K_T(h,x)^\star w)_Z dh$. Here, $x \in G, w \in W, g\in   H^2(H,\sigma)$. In \cite{OV2},  it is shown:  $K_T$ is a smooth function,     $ K_T(h,\cdot)z=K_{T^\star}(\cdot,h)^\star z \in H^2(G,\tau)$ and \begin{equation}\label{eq:kten} K_T(e,\cdot)z \in H^2(G,\tau)[V_\mu^H][V_{\mu+\rho_n^H}^L] \end{equation} is a $L$-finite vector in $  H^2(G,\tau)$.

\subsubsection{} Finally, we recall the restriction map \begin{equation*}\label{eq:r0} r_0 : H^2(G,\tau) \rightarrow  L^2(H_0 \times_\tau W),\,\, r_0(f)(h_0)=f(h_0), h_0 \in H_0, \end{equation*} is $(L^2,L^2)$-continuous \cite{OV1}.

  The main result of this section is,
\begin{thm}\label{prop:Rwelldefgc}   We assume $(G,H)$ is a symmetric pair and $res_H(\pi_\lambda)$ is admissible. We fix a irreducible factor $V_\mu^H$ for $res_H(\pi_\lambda)$.  Then,  the following statements hold.\\
i) The   map  $r_0 : H^2(G,\tau) \rightarrow  L^2(H_0 \times_\tau W)$  restricted to  $\mathrm{Cl}(\mathcal U(\mathfrak h_0)W )$ yields a isomorphism between $\mathrm{Cl}(\mathcal U(\mathfrak h_0)W )$ onto $\mathbf H^2(H_0, \tau)$. \\ ii) For each fixed intertwining $L$-equivalence

\smallskip
   \phantom{xxxxxx}  $D: \mathcal L_\lambda[V_{\mu +\rho_n^\mu}^L]=H^2(G,\tau)[V_\mu^H][V_{\mu +\rho_n^\mu}^L] \rightarrow  (\mathcal U(\mathfrak h_0)W)[V_{\mu +\rho_n^\mu}^L]$, \\ the map

    \smallskip
    \phantom{xxx} $r_0^D : Hom_H(H^2(H, \sigma),H^2(G,  \tau))\rightarrow Hom_L(V_{\mu +\rho_n^\mu}^L,   \mathbf H^2(H_0,\tau))$\\ defined by   $$ T \stackrel{r_0^D}\longmapsto(V_{\mu +\rho_n^\mu}^L\ni z \mapsto  r_0( D(K_T(e,\cdot)z))   \in \mathbf H^2(H_0,\tau))$$ is a linear  isomorphism.
\end{thm}

  \begin{rmk}  When the natural inclusion $H/L \rightarrow G/K$ is a holomorphic map, T. Kobayashi, M. Pevzner and Y. Oshima in  \cite{Kob3},\cite{KP2} have shown a similar dual multiplicity result  after replacing the underlying Harish-Chandra module in  $\mathbf H^2(H_0,\tau)$ by its representation as a  Verma module. Also, in the holomorphic setting,   Jakobsen-Vergne in \cite{JV} has shown the isomorphism $H^2(G,\tau)\equiv \sum_{r\geq 0} H^2(H, \tau_{\vert_L} \otimes S^{(r)}((\mathfrak h_0 \cap \mathfrak{ p}^+))^\star)$. On the papers, \cite{Na} \cite{La}, we find applications of  the result of Kobayashi for their work on decomposing   holomorphic Discrete Series.
    H. Sekiguchi \cite{Se} has obtained a similar result of branching laws for singular holomorphic representations.
\end{rmk}

\begin{rmk} The   proof of Theorem~\ref{prop:Rwelldefgc} requires to show the map $r_0^D$ is well defined as well as several structure Lemma's. Once we verify the map is well defined, we will show injectivity, Corollary~\ref{prop:r0dinj}, Proposition~\ref{prop:equaldim} and linear algebra will give the surjectivity.   In Proposition~\ref{prop:gentau}, we show $i)$, and,  in the same Proposition we give a proof of the existence of the map $D$ as well as its bijectivity, actually this result has been shown in \cite{Vaint}. However, we sketch a proof in this note. The surjectivity also depends heavily on a result in \cite{Vaint}, for completeness we give a proof. We may say that our proof of Theorem 1 is rather long and intricate, involving both linear algebra for finding the multiplicities, and analysis of the kernels of the intertwining operators in question to set up the equivalence of the $H$-morphisms and the $L$-morphisms. The structure of the branching and corresponding symmetry breaking is however very convenient to apply in concrete situations, and we give several illustrations.

  We explicit the inverse map to the bijection $r_0^D$ in subsection~\ref{sub:inverserd}.
\end{rmk}

 \begin{rmk} \label{rmk:estruD}  When $\mathcal L_\lambda =\mathcal U(\mathfrak h_0)W$ we may take $D$ equal to the identity map.

  \end{rmk}

\begin{rmk} A mirror statement to Theorem~\ref{prop:Rwelldefgc} for symmetry breaking operators is as follows: $Hom_H(H^2(G,\tau),H^2(H,\sigma))$ is isomorphic to $Hom_L(Z,\mathbf H^2(H_0,\tau))$ via the map $S\mapsto (z\mapsto (H_0 \ni x \mapsto r_0^D(S^\star)(z)(x)=r_0(D(K_S(\cdot,x)^\star)(z))\in W))$.

\end{rmk}
 \medskip

  \subsubsection{} We verify  $r_0(D(K_T(e,\cdot)z))(\cdot)$ belongs to $L^2(H_0 \times_\tau W)_{H_0-disc}$.

\smallskip
 Indeed, owing to our hypothesis, a result    \cite{DV} (see \cite[Proposition 2.4]{DGV}) implies $\pi_\lambda$ is $L$-admissible. Hence,     \cite[Theorem 1.2]{Kob1} implies $\pi_\lambda$ is  $H_0$-admissible. Also, \cite[Proposition 1.6]{Kob} shows the subspace of $L$-finite vectors in $H^2(G,\tau)$ is equal to the subspace of $K$-finite vectors and $res_{\mathcal U(\mathfrak h_0)}(H^2(G,\tau)_{K-fin})$ is a admissible,  completely algebraically decomposable representation. Thus, the subspace $H^2(G,\tau)[W]\equiv W$ is contained in a finite sum of irreducible $\mathcal U(\mathfrak h_0)$-factors. Hence, $\mathcal U(\mathfrak h_0)W$ is a finite sum of irreducible $\mathcal U(\mathfrak h_0)$ factors. In \cite{GW}, we find a proof that each irreducible summand for $res_{H_0}(\pi_\lambda)$ is a square integrable representations for $H_0$, hence,  the equivariance and continuity of $r_0$ yields  $r_0(\mathrm{Cl}(\mathcal U(\mathfrak h_0)W))$ is contained in $L^2(H_0\times_{\tau} W)_{H_0-disc}$. \ref{sec:kernelprop} shows $K_T(e, \cdot)\in V_\lambda[V_\mu^H][V_{\mu+\rho_n^H}^L]$, hence, $D(K_T(e,\cdot)z) (\cdot)\in \mathcal U(\mathfrak h_0)W$, and   the claim follows.

\subsubsection{}   The map $Z \ni z \mapsto r_0(D(K_T(e,\cdot)z)) (\cdot)\in  L^2(H_0 \times_\tau W)$ is a $L$-map.

For this,
   we recall the equalities $$K_T(hl,gk)=\tau(k^{-1})K_T(h,g)\sigma (l), k \in K, l\in L, g\in G, h\in H. $$ $$K_T(hh_1,hx)=K_T(h_1,x), h, h_1 \in H, x\in G . $$ Therefore,  $K_T(e, hl_1)\sigma(l_2)z=\tau(l_1^{-1})K_T(l_2,h)z=\tau(l_1^{-1})K_T(e, l_2^{-1}h)z$ for $l_1, l_2 \in L, h\in H_0$ and we have shown the claim.

 \medskip

We have enough information to verify   the injectivity we have claimed  in $i)$ as well as the injectivity of the map $r_0^D$. For these, we show a   fact valid for a arbitrary reductive pair $(G,H)$ and arbitrary Discrete Series representation.

\subsection{  Kernel   of the restriction map  } In this paragraph we show  a fact valid for any reductive pair $(G,H)$ and arbitrary representation $\pi_\lambda$. The objects involved here are the  restriction map $r$ from $H^2(G,\tau)$ into $L^2(H\times_\tau W)$ and the subspace

 \begin{equation}\label{eq:uhzerow}    \mathcal U(\mathfrak h)W:=L_{\mathcal U(\mathfrak h)}H^2(G,\tau)[V_{\lambda+\rho_n^\lambda}^K]\equiv\pi_\lambda(\mathcal U(\mathfrak h))(V_\lambda^G[V_{\lambda+\rho_n^\lambda}^K]). \end{equation}      We write $\mathrm{Cl}(\mathcal U(\mathfrak h)W)$ for the closure of $\mathcal U(\mathfrak h)W$.  The subspace  $\mathrm{Cl}(\mathcal U(\mathfrak h)W)$ is the closure of the left translates by the algebra  $\mathcal U(\mathfrak h) $ of the subspace of $K$-finite vectors

 \smallskip
\phantom{xxxxxxxxxxxxx} $ \{ K_\lambda (\cdot, e)^\star w: w \in W   \}=H^2(G,\tau)[W]$.  \\ Thus, $\mathcal U(\mathfrak h)W$ consists of analytic vectors for $\pi_\lambda$. Hence,   $\mathrm{Cl}(\mathcal U(\mathfrak h)W)$  is invariant by left translations by $H$.
 Therefore the subspace \\ \phantom{xx} $ L_H(H^2(G,\tau)[W])=\{ K_\lambda (\cdot, h)^\star w=L_h(K_\lambda (\cdot,e)^\star w) : w \in W,   h\in H\}$ \\ is contained in $\mathrm{Cl}(\mathcal U(\mathfrak h)W)$. Actually,

\phantom{xxxxxxx} $\mathrm{Cl}(L_H(H^2(G,\tau)[W]))=\mathrm{Cl}(\mathcal U(\mathfrak h)W)$.

 The other inclusion follows from that $\mathrm{Cl}(L_H(H^2(G,\tau)[W]))$ is invariant by left translation by $H$ and $ \{ K_\lambda (\cdot, e)^\star w: w \in W   \}$ is contained in the subspace of smooth vectors  in $\mathrm{Cl}(L_H(H^2(G,\tau)[W]))$.

\smallskip

  The result pointed out in the title of   this paragraph is: \begin{lem} \label{lem:injecrh}Let $(G,H)$ be a arbitrary reductive pair and an arbitrary representation $(\pi_\lambda, H^2(G,\tau))$.   Then, $Ker(r)$  is equal to the orthogonal    subspace to $\mathrm{Cl}(\mathcal U(\mathfrak h)W)$.
\end{lem}
\begin{proof}  Since, \cite{OV1}, $r: H^2(G,\tau)\rightarrow L^2(H\times_\tau W)$ is a continuous map, we have $Ker(r)$ is a closed subspace of $H^2(G,\tau)$. Next, for $f\in H^2(G,\tau)$,  it holds the  identity   \\ \phantom{xxx} $(f(x),w)_W=\int_G(f(y), K_\lambda (y,x)^\star w)_W dy, \forall x\in G, \forall w\in W$. \\ Thus,  $r(f)=0$ if and only if $f$ is orthogonal to the subspace spanned by $\{ K_\lambda (\cdot, h)^\star w: w \in W,   h\in H\}$.

 Hence, $\mathrm{Cl}(Ker(r))=(\mathrm{Cl}(L_H(H^2(G,\tau)[W])))^\perp$.  Applying the considerations after the definition of $\mathrm{Cl}(\mathcal U(\mathfrak h)W)$ we obtain $Ker(r)^\perp =  \mathrm{Cl}(\mathcal U(\mathfrak h)W)$.
    Thus $Ker(r)=(Ker(r)^\perp)^\perp =\mathrm{Cl}(\mathcal U(\mathfrak h)W)$.
\end{proof}
\begin{cor} Any irreducible $H$-discrete factor $M$ for $\mathrm{Cl}(\mathcal U(\mathfrak h)W)$ contains a $L$-type in $res_L(\tau)$. That is, $M[res_L(\tau)]\not=\{0\}.$
\end{cor} The corollary follows from that $r$ restricted to $\mathrm{Cl}(\mathcal U(\mathfrak h)W)$ is injective and that Frobenius reciprocity for $L^2(H\times_\tau W)$ holds.
 \subsection{The map $r_0^D$ is injective} As a consequence of the general fact shown in the previous subsection, we obtain the injectivity in $i)$ and the map $r_0^D$ is injective.
\begin{cor}Let $(G,H)$ be a symmetric pair and $H_0=G^{\sigma \theta}$. Then, the restriction map  $r_0 :H^2(G,\tau)\rightarrow L^2(H_0\times_\tau W) $ restricted to the subspace $\mathrm{Cl}(\mathcal U(\mathfrak h_0)W)$ is  injective.
\end{cor}

\begin{cor}\label{prop:r0dinj} Let $(G,H)$ be a symmetric pair, $H_0=G^{\sigma \theta}$ and we assume $res_H(\pi_\lambda)$ is $H$-admissible. Then, the  map $r_0^D$ is injective. \end{cor} In fact,  for $T\in Hom_H(H^2(H,\sigma), H^2(G,\tau))$, if $ r_0(D(K_T(e,\cdot)z))=0 \, \forall z \in Z$, then, since $ D(K_T(e,x)z) \in \mathcal U(\mathfrak h_0)W$,  the previous corollary    implies $ D(K_T(e,x)z)=0 \, \forall z, x \in G$. Since, $ K_T(e,\cdot)z \in V_\lambda^G[V_\mu^H][V_{\mu+\rho_n^H}^L]$, and $D$ is injective we obtain $ K_T(e,x)z=0 \, \forall z, \forall x$. Lastly, we recall equality $K_T(h,x)=K_T(e,h^{-1}x)$. Whence we have verified the corollary.

Before we show the surjectivity for the map $r_0^D$ we would like to comment   on other works on the topic   of this note.

\subsection{Previous work on   duality formula and Harish-Chandra parameters} The setting for this subsection is: $(G,H)$ is a symmetric pair and $(\pi_\lambda , V_\lambda^G)$ is a irreducible square integrable representation of $G$ and $H$-admissible. As before, we fix $K,L=H\cap K, T, U=H\cap T$.
The following Theorem has been shown by \cite{GW}, a different proof is in \cite{KO}.
\begin{thm}[Gross-Wallach, T. Kobayashi-Y. Oshima]\label{prop:GWKO} We assume $(G,H)$ is symmetric pair, $\pi_\lambda^G$-is $H$-admissible,  then

\smallskip
a) $res_H(\pi_\lambda^G)$ is the Hilbert sum of {\bf inequivalent} Square integrable representations for $H$,  $\pi_{\mu_j}^H, j=1, 2, \dots$, with respective finite multiplicity $0<m_j<\infty$.

\smallskip

b) The Harish-Chandra parameters of the totality of discrete factors for $res_H(\pi_\lambda^G)$ belong to a "unique" Weyl Chamber in $i\mathfrak u^\star $.

\smallskip
 That is, \phantom{xx} $V_\lambda^G =\oplus_{1\leq j<\infty} V_\lambda^G[V_{\mu_j}^H]\equiv \oplus_{ j} \,\, Hom_H(V_{\mu_j}^H, V_\lambda^G)\otimes V_{\mu_j}^H$,

\medskip
\phantom{xxxx}  $\dim  Hom_H(V_{\mu_j}^H, V_\lambda^G)=m_j$,  \phantom{xx} $\pi_{\mu_j}^H \neq \pi_{\mu_i}^H $ iff $i\not= j$, \\   and  there exists a system of positive roots  $ \Psi_{H,\lambda} \subset \Phi(\mathfrak h,\mathfrak u),$  such that for all $j$, $(\alpha,\mu_j)>0 $\,\, for all $\alpha \in  \Psi_{H,\lambda} $.

\end{thm}
  In \cite{Vaint}\cite{KO} (see Tables 1,2,3) we find the list of pairs $(\mathfrak g,\mathfrak h)$, as well as systems of positive roots $\Psi_G \subset \Phi(\mathfrak g, \mathfrak t), \Psi_{H,\lambda}\subset \Phi(\mathfrak h, \mathfrak u)$   such that,

-  $ \lambda$ dominant with respect to $ \Psi_G  $ implies $res_H(\pi_\lambda^G)$ is admissible.

- For all $\mu_j$ in $a)$ we have  $(\mu_j, \Psi_{H,\lambda})>0$.

\smallskip
- When $U=T$,  we have $\Psi_{H,\lambda}=\Psi_\lambda \cap \Phi(\mathfrak h,\mathfrak t).$

\smallskip Since $(G,H_0)$ is a symmetric pair, Theorem~\ref{prop:GWKO} as well as its comments apply to $(G,H_0)$ and $\pi_\lambda$.  Here,   when $U=T$,  $\Psi_{H_0, \lambda}=\Psi_\lambda \cap \Phi(\mathfrak h_0,\mathfrak t)$.
\medskip

 From the tables in \cite{Vaint} it follows that any of  the system  $\Psi_\lambda$,  $\Psi_{H,\lambda}, \Psi_{H_0, \lambda}$  has,  at most,   two  noncompact simple roots,  and the sum  of the respective multiplicity   of each noncompact simple root in  the highest root   is less or equal than two.

\subsubsection{Computing Harish-Chandra parameters from Theorem~\ref{prop:Rwelldefgc}}\label{sub:paramhc}
As usual,  $\rho_n=\frac12 \sum_{\beta \in \Psi_\lambda \cap \Phi_n} \beta$, $  \rho_n^H=\frac12 \sum_{\beta \in \Psi_{H,\lambda} \cap \Phi_n} \beta $, $\rho_K= \frac12 \sum_{\alpha \in \Psi_\lambda \cap \Phi_c} \alpha$, \\ $\rho_L= \frac12 \sum_{\alpha \in \Psi_{H,\lambda} \cap \Phi_c} \alpha$.   We write $res_L(\tau)=res_L(V_{\lambda+\rho_n}^K )=\oplus_{1\leq j\leq r} \,  q_j\, \pi_{\nu_j}^L=\sum_j q_j \sigma_j$, with $\nu_j$  dominant with respect to $\Psi_{H,\lambda} \cap \Phi_c$.
we recall  $\nu_j$ is the infinitesimal character (Harish-Chandra parameter) of $\pi_{\nu_j}^L$. Then, the Harish-Chandra parameter for $H^2(H_0, \pi_{\nu_j}^L)$ is $\eta_j=\nu_j -\rho_n^{H_0}$.

 According to  \cite[Lemma 2.22]{Sc}(see Remark~\ref{ktypessch}), the infinitesimal character of a $L$-type of $H^2(H_0, \pi_{\nu_j}^L)$ is equal to $\nu_j +B=\eta_j+\rho_n^{H_0}+B $ where $B$ is a sum of roots in $\Psi_{H_0,\lambda}\cap \Phi_n$.

The isomorphism $r_0^D$ in Theorem~\ref{prop:Rwelldefgc}, let us conclude:

For each subrepresentation  $V_{\mu_s}^H $ of $res_H(\pi_\lambda)$, we have $\mu_s +\rho_n^H$ is a $L$-type of \\ \phantom{xxxxxxx} $\mathbf H^2(H_0,\tau)\equiv \oplus_j\, q_j \, H^2(H_0,\pi_{\nu_j}^L)\equiv \oplus_j \,  \underbrace { V_{\eta_j}^{H_0}\oplus \cdots \oplus V_{\eta_j}^{H_0}}_{q_j}$, \\  and the multiplicity  of $V_{\mu_s}^H$ is equal to the multiplicity of $V_{\mu_s +\rho_n^H}^L$   in $\mathbf H^2(H_0,\tau)$.

\subsubsection{Gross-Wallach multiplicity formula}\label{sec:gwmultip}

To follow we describe the duality Theorem due to \cite{GW}. $(G,H)$ is a symmetric pair. For this paragraph, in order to avoid subindexes we write  $\mathfrak g=Lie(G), \mathfrak h=Lie(H) $ etc.   We recall $\mathfrak h_0=\mathfrak g^{\sigma \theta}$. We have the decompositions $\mathfrak g=\mathfrak k+\mathfrak p=\mathfrak h+\mathfrak q=\mathfrak h_0 +\mathfrak p \cap \mathfrak h + \mathfrak q \cap \mathfrak k$. The {\it dual}   real Lie algebra to $\mathfrak g$ is $\mathfrak g^d=\mathfrak h_0 +i(\mathfrak p \cap \mathfrak h + \mathfrak q \cap \mathfrak k)$, the  algebra  $\mathfrak g^d$ is a real form for  $\mathfrak g_\mathbb C$.  A  maximal compactly embedded subalgebra for $\mathfrak g^d$  is $\widetilde{\mathfrak k}=\mathfrak h \cap \mathfrak k +i(\mathfrak h \cap \mathfrak p)$. Let $\pi_\lambda$ be a $H$-admissible  Discrete Series for $G$. One of the main results of \cite{GW} attach to $\pi_\lambda$ a finite sum of the underlying Harish-Chandra module of finitely many fundamental representations  for $G^d$, $(\Gamma_{H\cap L}^{\widetilde K})^{p_0+q_0}(N(\Lambda))$,  so that for each subrepresentation $V_\mu^H$ of $V_\lambda^G$ we compute the multiplicity $m^{G,H}(\lambda, \mu)$ of $V_\mu^H$ by means of   Blattner's formula \cite{HS} applied to $(\Gamma_{H\cap L_1}^{\widetilde K})^{p_0+q_0}(N(\Lambda))$ . In more detail, since $Lie(H)_\mathbb C =Lie(\widetilde K)_\mathbb C$, and the center of $H$ is equal to the center of $\widetilde K$, for the infinitesimal character $\mu$ and the central character $\chi$ of $V_\mu^H$, we may associate a  finite dimensional irreducible representation $F_{\mu,\chi}$ for $\widetilde K$. Then, they show \begin{equation*} \dim  Hom_{\mathfrak h,H\cap K}(V_\mu^H, V_\lambda^G)=\dim  Hom_{\widetilde K}(F_{\mu,\chi}, (\Gamma_{H\cap L_1}^{\widetilde K})^{p_0+q_0}(N(\Lambda))), \end{equation*}
\begin{equation*} m^{G,H}(\lambda,\mu) =(-1)^{\frac12 \dim  ( H/H\cap L_1)} \sum_{i=1}^d \sum_{s\in W_{\widetilde K}} \epsilon(s)p(\Lambda_i +\rho_{\widetilde K} +ss_{ H\cap K } \mu), \end{equation*}
where $\tau=F^\Lambda=\sum_i M^{\Lambda_i}$ as a sum of irreducible $H\cap L_1$-module and $p$ is the partition function associated to $\Phi(\mathfrak u_1 /\mathfrak u_1 \cap\mathfrak h_\mathbb C, \mathfrak u)$, here, $\mathfrak u_1$ is the nilpotent radical of  certain parabolic subalgebra $\mathfrak q=\mathfrak l_1 +\mathfrak u_1$ used to define the $A_\mathfrak q(\lambda)$-presentation for $\pi_\lambda$.  {\it Explicit example IV} presents the result of \cite{GW} for the pair $(SO(2m,2n), SO(2m,2n-1))$.

\subsubsection{Duflo-Vargas multiplicity formula,  \cite{DV}}   We keep notation and hypothesis as in the previous paragraph. Then, \begin{equation*} m^{G,H}(\lambda,\mu)=\pm \sum_{w \in W_K} \epsilon (w)p_{S_w^H}(\mu -q_\mathfrak u(w\lambda)).\end{equation*}
Here, $q_\mathfrak u :\mathfrak t^\star \rightarrow \mathfrak u^\star $ is the restriction map. $p_{S_w^H}$ is the partition function associated to the multiset \begin{equation*} S_w^H :=S_w^L \backslash \Phi(\mathfrak h/\mathfrak l,\mathfrak u), \,\mathrm{where,} \,\, S_w^L:=q_\mathfrak u(w(\Psi_\lambda)_n)\cup \Delta(\mathfrak k/\mathfrak l, \mathfrak u). \end{equation*}

We recall for a strict multiset of elements in vector space $V$  the partition function   attached to $S$, roughly speaking,  is the function that counts the number of ways of expressing each vector as a  nonnegative integral linear combinations of elements of $S$. For a precise definition see \cite{DV} or the proof of Lemma~\ref{lem:ddzhfinite}.
\subsubsection{Harris-He-Olafsson multiplicity formula, \cite{HHO}}
Notation and hypothesis as in the previous paragraphs. Let \begin{equation*} r_m : H^2(G,\tau)\rightarrow L^2(H \times_{S^m(Ad)\boxtimes \tau} (S^m(\mathfrak p \cap \mathfrak q)^\star \otimes W)). \end{equation*} the normal derivative map defined in \cite{OV1}.  Let $\Theta_{\pi_\mu^H}$ denote the Harish-Chandra character of $\pi_\mu^H$.   For $f$ a tempered function in $H^2(G,\tau)$, they define $\phi_{\pi_\lambda, \pi_\mu^H, m}(f)=\Theta_{\pi_\mu^H}
\star r_m(f)$ . They show: \begin{equation*} m^{G,H}(\lambda, \mu)= \lim_{m\rightarrow \infty} \dim  \phi_{\pi_\lambda, \pi_\mu^H, m}((H^2(G,\tau)\cap \mathcal C(G,\tau)) [V_{\mu +\rho_n^H}]). \end{equation*}

 \subsection{Completion of the Proof of Theorem~\ref{prop:Rwelldefgc}, the map $r_0^D$ is surjective }
Item $i)$ in Theorem~\ref{prop:Rwelldefgc} is shown in Proposition~\ref{prop:gentau} $c)$. The existence of the map $D$ is shown in Proposition~\ref{prop:gentau} $e)$.

To show the surjectivity of   $r_0^D$ we appeal to Theorem~\ref{prop:equaldim},  \cite[Theorem 1]{Vaint}, where we   show the initial space and the target space are equidimensional, linear algebra concludes  de proof of  Theorem~\ref{prop:Rwelldefgc}. Thus, we conclude the proof of Theorem~\ref{prop:Rwelldefgc} as soon as we complete the proof of Theorem~\ref{prop:equaldim} and Proposition~\ref{prop:gentau}.

\section{Duality Theorem, proof of dimension equality}\label{sec:equaldim}

The purpose of this subsection is to sketch  a proof of  the equality of dimensions in the duality formula presented in Theorem~\ref{prop:Rwelldefgc} as well as some consequences.     Part of the notation has already been introduced in the previous section.   Sometimes notation will be explained after it has been used.  Unexplained notation   is as in \cite{DV}, \cite{OV2}, \cite{Vaint}.

\subsection{Dimension equality Theorem, statement}
The setting is as follows, $(G,H)$ is a symmetric pair,$(\pi_\lambda, V_\lambda^G)=(L,H^2(G,\tau))$ a $H$-admissible irreducible square integrable representation. Then,  the Harish-Chandra parameter $\lambda$ gives rise to systems of positive roots $\Psi_\lambda $ in $\Phi(\mathfrak g,\mathfrak t)$ and by mean of $\Psi_\lambda$,  in \cite{DV} is defined a nontrivial normal connected subgroup $K_1(\Psi_\lambda)=:K_1$ of $K$, and,  it is shown that the $H$-admissibility yields $K_1 \subset H$\footnote{This also follows from the tables in \cite{KO}}.  Thus,   $\mathfrak k=\mathfrak k_1 \oplus \mathfrak k_2$, $\mathfrak l = \mathfrak k_1 \oplus  \mathfrak l \cap \mathfrak k_2$ (as ideals), and $\mathfrak t=\mathfrak t \cap \mathfrak k_1 + \mathfrak t \cap \mathfrak k_2$, $\mathfrak u:=\mathfrak t\cap \mathfrak l =\mathfrak u \cap \mathfrak k_1 + \mathfrak u \cap \mathfrak k_2$ is a Cartan  subalgebra of $\mathfrak l$. Let $q_\mathfrak u$ denote restriction map from $\mathfrak t^\star$ onto $\mathfrak u^\star$.  Let $K_2$ denote the analytic subgroup corresponding to $\mathfrak k_2$. We recall   $H_0:=(G^{\sigma \theta})_0$, $L =K\cap H =K\cap H_0$. We have $K=K_1 K_2$,  $L=K_1 (K_2\cap L)$.  We set    $\Delta:=\Psi_\lambda \cap \Phi(\mathfrak k,\mathfrak t)$. Applying Theorem~\ref{prop:GWKO} to both $H$ and $H_0$ we obtain respective systems of positive roots  $\Psi_{H,\lambda}$ in $\Phi(\mathfrak h, \mathfrak u)$, $\Psi_{H_0,\lambda}$ in $\Phi(\mathfrak h_0, \mathfrak u)$.  For a list of  six-tuples $(G,H,  \Psi_\lambda,$ $ \Psi_{H,\lambda}, \Psi_{H_0,\lambda}, K_1)$ we refer to \cite[Table 1, Table 2, Table 3]{Vaint}. Always, $ \Psi_{H,\lambda}\cap \Phi_c(\mathfrak l, \mathfrak u)=  \Psi_{H_0,\lambda}\cap \Phi_c(\mathfrak l, \mathfrak u)$. As usual, either $\Phi_n (\mathfrak g, \mathfrak t)$ or $\Phi_n$ denotes the subset of noncompact roots in $\Phi(\mathfrak g,\mathfrak t)$, $\rho_n^\lambda$ (resp. $\rho_n^H, \rho_n^{H_0} $) denotes one half of the sum of the elements in $\Psi_\lambda \cap \Phi_n (\mathfrak g, \mathfrak t)$ (resp. $\Phi_n \cap \Psi_{H,\lambda}, \Phi_n \cap \Psi_{H_0,\lambda} $). When $\mathfrak u=\mathfrak t,  \rho_n^\lambda=  \rho_n^H + \rho_n^{H_0} $.   From now on, the infinitesimal character   of an irreducible representation of $K$ (resp $L$) is dominant with respect to $\Delta$ (resp. $\Psi_{H,\lambda} \cap \Phi(\mathfrak l, \mathfrak u))$. The lowest $K$-type $(\tau,W)$ of $\pi_\lambda$ decomposes  $\pi_{\lambda +\rho_n^\lambda}^K =\pi_{\Lambda_1}^{K_1}\boxtimes \pi_{\Lambda_2}^{K_2}$, with $\pi_{\Lambda_s}^{K_s}$ an irreducible representation for $K_s, s=1,2$. We express $\gamma=(\gamma_1, \gamma_2)\in  \mathfrak t^\star=\mathfrak t_1^\star +\mathfrak t_2^\star $. Hence, \cite{HS}\cite{DHV}, $\Lambda_1 =\lambda_1 +(\rho_n^\lambda)_1 , \Lambda_2 =\lambda_2 +(\rho_n^\lambda)_2$.  Sometimes $(\rho_n^\lambda)_2 \not=0$. This happens only for $\mathfrak{su}(m,n)$ and some particular systems $\Psi_\lambda$ (see proof of Lemma~\ref{lem:equalm}).  Harish-Chandra parameters for the irreducible factors of  either $res_H(\pi_\lambda)$ (resp. $res_{H_0}(\pi_\lambda)$))  will always be dominant with respect to $\Psi_{H,\lambda} \cap \Phi (\mathfrak l, \mathfrak u)$ (resp. $\Psi_{H_0,\lambda} \cap \Phi (\mathfrak l, \mathfrak u))$.

 For short, we write  $\pi_{\Lambda_2} :=\pi_{\Lambda_2}^{K_2}$.
  We write $$ res_{L\cap K_2}(\pi_{\Lambda_2})= res_{L\cap K_2}(\pi_{\Lambda_2}^{K_2})= \sum_{\nu_2 \in (\mathfrak u \cap \mathfrak k_2)^\star}  \, m^{K_2, L\cap K_2}(\Lambda_2, \nu_2) \, \pi_{\nu_2}^{L\cap K_2},$$ as a sum of irreducible representations of $L\cap K_2$.\\
  The set of $\nu_2$ so that   $ m^{K_2, L\cap K_2}(\Lambda_2, \nu_2)\not= 0$ is denoted by $Spec_{L\cap K_2}(\pi_{\Lambda_2}^{K_2})$.  Thus,
\begin{equation} \label{eq:resKL1} res_{L}(\pi_{\Lambda_1}^{K_1} \boxtimes \pi_{\Lambda_2}^{K_2})= \sum_{\nu_2 \in Spec_{L\cap K_2}(\pi_{\Lambda_2}^{ K_2})}  \, m^{K_2, L\cap K_2}(\Lambda_2, \nu_2) \, \pi_{\Lambda_1}^{ K_1}\boxtimes \pi_{\nu_2}^{L\cap K_2}, \end{equation} as a sum of irreducible representations of $L$. Besides, for a Harish-Chandra parameter $\eta=(\eta_1,\eta_2)$ for $H_0$, we write
$$  res_L( \pi_{(\eta_1,\eta_2)}^{H_0})= \sum_{(\theta_1,\theta_2) \in Spec_L(\pi_{(\eta_1,\eta_2)}^{H_0}) } \, m^{H_0, L}( (\eta_1,\eta_2),(\theta_1,\theta_2) ) \, \pi_{(\theta_1,\theta_2)}^{L}.$$ The restriction of $\pi_\lambda$ to $H$ is expressed by (see \ref{prop:GWKO})

$$ res_H(\pi_{\lambda})= res_H( \pi_{\lambda}^{G})= \sum_{\mu \in    Spec_H(\pi_\lambda) } \, m^{G, H}(\lambda, \mu) \, \pi_{\mu}^{H}.$$
In the above formulaes, $m^{\cdot,\cdot }(\cdot,\cdot)$ are non negative  integers and represent multiplicities; for $\nu_2 \in Spec_{L\cap K_2}(\pi_{\Lambda_2}^{K_2})$, $\nu_2$ is dominant with respect to $\Psi_{H,\lambda} \cap \Phi(\mathfrak k_2, \mathfrak u \cap \mathfrak k_2)$, and  $(\Lambda_1,\nu_2)$ is $\Psi_{H_0,\lambda}$-dominant  (see \cite{Vaint}); in the third formulae,  $(\eta_1, \eta_2)$ is dominant with respect to $\Psi_{H_0,\lambda} $ and  $(\theta_1, \theta_2)$ is dominant with respect to $\Psi_{H_0,\lambda} \cap \Phi_c(\mathfrak h_0, \mathfrak u )$; in the fourth formula, $\mu$ is dominant with respect to $\Psi_{H,\lambda}$. Sometimes, for $\mu \in Spec_{H}(\pi_{\lambda}^{G})$, we
 replace $\rho_n^\mu$   by $\rho_n^H$.

We make a change of notation:

\phantom{xxxxxxxxxxx}$\sigma_j=\pi_{\Lambda_1}^{K_1}\boxtimes \pi_{\nu_2}^{L\cap K_2}$ and $q_j= m^{K_2, L\cap K_2}(\Lambda_2, \nu_2)$.\\
Then, in order to show either the existence of the map $D$ or the surjectivity of the map $r_0^D$,  we need to show:

\begin{thm} \label{prop:equaldimeigensrea}  \begin{eqnarray*} \lefteqn{ m^{G, H}(\lambda, \mu)=\dim  Hom_H(H^2(H,V_{\mu+\rho_n^H}^L), H^2(G,V_{\lambda +\rho_n^\lambda}^K))  } \\  &  &   \hspace{1.0cm} \mbox{ }\mbox{ }  =  \sum_{\nu_2 \in Spec_{L\cap K_2}(\pi_{\Lambda_2}^{K_2}) }  \, m^{K_2, L\cap K_2}(\Lambda_2, \nu_2) \\  &   & \hspace{4.2cm}\times \dim  Hom_L(  V_{\mu +\rho_n^H}^L,    H^2(H_0, \pi_{\Lambda_1}^{K_1}\boxtimes \pi_{ \nu_2}^{L\cap K_2}))  .\end{eqnarray*}

\end{thm}

A complete proof of the  result is in \cite{Vaint}. However, for sake of completeness and clarity  we would like to sketch a proof. We also present some consequences of the Theorem.

Next, we compute the infinitesimal character, $ic(...)$,   equivalently the Harish-Chandra parameter, for $  H^2(H_0, \pi_{\Lambda_1}^{K_1}\boxtimes \pi_{ \nu_2}^{L\cap K_2})$ and  restate the previous Theorem.

 \begin{equation*}\begin{split} & ic( H^2(H_0, \pi_{\Lambda_1}^{K_1}\boxtimes \pi_{ \nu_2}^{L\cap K_2}) )=(\Lambda_1,\nu_2)-\rho_n^{H_0} \\   & =(\lambda_1 +(\rho_n^{\lambda})_1-(\rho_n^{H_0})_1,\nu_2 -(\rho_n^{H_0})_2)  =(\lambda_1 +(\rho_n^{H})_1,\nu_2-(\rho_n^{H_0})_2)  \\   & =(\lambda_1, \nu_2 -(\rho_n^{H})_2  -(\rho_n^{H_0})_2  )  +\rho_n^H  =(\lambda_1, \nu_2 -q_\mathfrak u(\rho_n^{\lambda})_2\, )  +\rho_n^H. \end{split}\end{equation*}

To follow, we state Theorem~\ref{prop:equaldimeigensrea} regardless of the realization of the involved Discrete Series.
   \begin{thm} \label{prop:equaldim} Duality, dimension formula. The hypothesis is $(G,H)$ is a symmetric pair and $\pi_\lambda$ is a $H$-admissible representation. Then,  \begin{equation*}\begin{split}& m^{G, H}(\lambda, \mu)=\dim  Hom_H(V_{\mu }^H, V_\lambda^G)     \\  &  =  \sum_{\nu_2 \in Spec_{L\cap K_2}(\pi_{\Lambda_2}^{K_2}) }  \, m^{K_2, L\cap K_2}(\Lambda_2, \nu_2)     \dim  Hom_L \big(  V_{\mu +\rho_n^H}^L,    V_{(\lambda_1, \nu_2-q_\mathfrak u (\rho_n^\lambda)_2)+\rho_n^H}^{H_0}\big) \\ &  =  \sum_{\nu_2^\prime \in Spec_{L\cap K_2}(\pi_{\Lambda_2}^{K_2})-q_\mathfrak u (\rho_n^\lambda)_2  }  \, m^{K_2, L\cap K_2}(\Lambda_2, \nu_2^\prime +q_\mathfrak u (\rho_n^\lambda)_2 )  \\ &  \phantom{xxxxxxxxxxxxxxxxxxxxxxxxxxxxxxxxxxxxx} \times \dim  Hom_L\big(  V_{\mu +\rho_n^H}^L,    V_{(\lambda_1, \nu_2^\prime) +\rho_n^H}^{H_0}\big) .\end{split} \end{equation*}  \end{thm}

{\it Remark:} In the proof of Lemma~\ref{lem:equalm} we obtain that $(\rho_n^\lambda)_2\not= 0$ forces $\mathfrak g=\mathfrak{su}(m,n)$, $\mathfrak u=\mathfrak t$ and $(\rho_n^\lambda)_2$ is orthogonal to $(\mathfrak z_\mathfrak k)^\perp$. That is, $(\rho_n^\lambda)_2 \in \mathfrak z_\mathfrak k^\star  $, i.e., for every symmetric pair $(G,H)$ we have $(\rho_n^\lambda)_2$ determines a central character of $\mathfrak k$.   It is also verified that since $\Lambda_2 = (\rho_n^\lambda)_2+ \lambda_2$, we have $Spec_{L\cap K_2}(\pi_{\Lambda_2}^{K_2})=q_\mathfrak u(\rho_n^\lambda)_2 + Spec_{L\cap K_2}(\pi_{\lambda_2}^{K_2})$ and for $Spec_{L\cap K_2}(\pi_{\Lambda_2}^{K_2}) \ni \nu_2=q_\mathfrak u(\rho_n^\lambda)_2 +\nu_2^\prime, \nu_2^\prime \in  Spec_{L\cap K_2}(\pi_{\lambda_2}^{K_2})$ it holds the equality $ m^{K_2, L\cap K_2}(\Lambda_2, \nu_2)= m^{K_2, L\cap K_2}(\lambda_2, \nu_2^\prime)$. Finally, Lemma~\ref{lem:multshift} yields the formula in Theorem~\ref{prop:equaldim} is equivalent to the formula
$$ m^{G, H}(\lambda, \mu) = \sum_{\nu_2^\prime \in Spec_{L\cap K_2}(\pi_{\lambda_2}^{K_2}) } m^{K_2 , L\cap K_2}(\lambda_2, \nu_2^\prime) \dim  Hom_L (V_\mu^L, V_{(\lambda_1, \nu_2^\prime)}^{H_0}).$$
This formula is the one we show in Lemma~\ref{lem:ddzhfinite}.

 \medskip
 The following diagram helps to understand the equalities in   Theorem~\ref{prop:equaldim} and in the next three Lemmas.

\medskip

\xymatrix{
Spec_H( V_{(\lambda_1, \lambda_2)}^G ) \ar[r]^-{\mu \mapsto \mu} \ar[dr]_{\nu \mapsto \nu+\rho_n^H}
&  \cup_{\nu_2^\prime \in Spec_{L\cap K_2}(\pi_{\lambda_2}^{K_2})} Spec_L(V_{(\lambda_1,\nu_2^\prime)}^{H_0}) \ar[d]^{\nu \mapsto \nu+\rho_n^H}  \\
   & Spec_L(\mathbf H^2(H_0,\tau))= \cup_{\nu_2^\prime \in Spec_{L\cap K_2}(\pi_{\lambda_2}^{K_2})} Spec_L(V_{(\lambda_1, \nu_2^\prime) +\rho_n^H}^{H_0}) }

\medskip

\smallskip
A consequence of Theorem~\ref{prop:equaldim},  Lemma~\ref{lem:multshift} and  Lemma~\ref{lem:ddzhfinite} is:

\begin{cor}
  \begin{multline*}Spec_H( \pi_\lambda) +\rho_n^H \\ =Spec_L(\mathbf H^2(H_0,\tau))= \cup_{\nu_2^\prime \in Spec_{L\cap K_2}(\pi_{\lambda_2}^{K_2})} Spec_L(V_{(\lambda_1, \nu_2^\prime) +\rho_n^H}^{H_0}). \end{multline*}

\phantom{xxxxxx} $Spec_H( \pi_\lambda)  =  \cup_{\nu_2^\prime \in Spec_{L\cap K_2}(\pi_{\lambda_2}^{K_2})} Spec_L(V_{(\lambda_1, \nu_2^\prime) }^{H_0}).$
\end{cor}

  As we pointed out, Theorem~\ref{prop:equaldim}     follows after we verify the next three Lemmas.
\begin{lem}\label{lem:ddzhfinite}   The hypothesis is $(G,H)$ is a symmetric space,   $\pi_\lambda$ is $H$-admissible.  Then
\begin{eqnarray*} \lefteqn{  \dim  Hom_H(V_\mu^H, V_\lambda^G) }\hspace{1.0cm} \\ & & = \sum_{\nu_2^\prime \in Spec_{L\cap K_2}(\pi_{\lambda_2}^{K_2})} m^{K_2 , L\cap K_2}(\lambda_2, \nu_2^\prime) \dim  Hom_L (V_\mu^L, V_{(\lambda_1, \nu_2^\prime)}^{H_0}). \end{eqnarray*}
\end{lem}

\begin{proof}[Proof of Lemma \ref{lem:ddzhfinite}]
The hypothesis $(G,H)$ is a symmetric pair and $\pi_\lambda$ is $H$-admissible, let us to apply    notation and facts in \cite{DV}, \cite{Vaint} as well as in \cite{H} \cite{DHV} \cite{GW} \cite{KO}. The proof is based on an idea in \cite{DHV} of piling up multiplicities by means of Dirac delta distributions. That is, let $\delta_\nu$ denote the Dirac delta distribution  at $\nu \in i\mathfrak u^\star$. Under our hypothesis, the function   $m^{G,H}(\lambda,\mu)$ has polynomial growth in $\mu$, whence,  the series $\sum_\mu m^{G,H}(\lambda, \mu) \,\,\delta_\mu $ converges in the space of distributions in $i\mathfrak u^\star$. Since Harish-Chandra parameter is regular, we may and will extend the function $m^{G,H}(\lambda, \cdot)$ to a $W_L$-skew symmetric function  by the rule $m^{G,H}(\lambda, w\mu) =\epsilon (w) m^{G,H}(\lambda, \mu), w \in W_L$. Thus, the series $\sum_{\mu \in HC-param(H)} m^{G,H}(\lambda, \mu)\delta_\mu$ converges in the space of distributions in $i\mathfrak u^\star$.  Next, for $0\not= \gamma \in i\mathfrak u^\star$ we consider the discrete Heaviside distribution $y_\gamma :=\sum_{n\geq 0}\delta_{\frac{\gamma}{2} +n\gamma}$,  and for a strict, finite, multiset $S=\{\gamma_1, \dots, \gamma_r\}$ of elements in $i\mathfrak u^\star$, we set \begin{equation*}y_S:= y_{\gamma_1}\star \cdots \star y_{\gamma_r}=\sum_{\mu \in i\mathfrak u^\star} p_S(\mu) \delta_\mu .\end{equation*}  Here, $\star$ is  the convolution product in the space of distributions on $i\mathfrak u^\star$. $p_S$ is called the {\it partition} function attached to the set $S$. Then, in \cite{DV} there is presented the equality
$$\sum_{\mu   \in HC-param(H)}  m^{G,H}(\lambda, \mu) \,\,\delta_\mu =\sum_{w \in W_K} \epsilon(w)\,\,\delta_{q_\mathfrak u(w\lambda)} \star y_{S_w^H}.$$
Here,  $W_S$ is the Weyl group of the compact connected Lie group $S$; for a $ad(\mathfrak u)$-invariant linear subspace $R$ of $\mathfrak g_\mathbb C$, $\Phi(R,\mathfrak u)$ denotes the multiset of elements in $\Phi(\mathfrak g,\mathfrak u)$ such that its root space is contained in $R$, and    $S_w^H=[q_\mathfrak u(w(\Psi_\lambda)_n)\cup \Delta(\mathfrak k/\mathfrak l,\mathfrak u)]\backslash \Phi(\mathfrak h/\mathfrak l,\mathfrak u)$.

Since, $K=K_1K_2$,  $W_K =W_{K_1}\times W_{K_2}$, we write $W_K \ni  w=st, s\in W_{K_1}, t\in W_{K_2}$. We recall the hypothesis yields $K_1\subset L$. It readily follows:  $s \Phi(\mathfrak h/\mathfrak l,\mathfrak u)=\Phi(\mathfrak h/\mathfrak l,\mathfrak u)$,  $ s\Delta(\mathfrak k/\mathfrak l,\mathfrak u)=\Delta(\mathfrak k/\mathfrak l,\mathfrak u),   t(\Psi_\lambda)_n=(\Psi_\lambda)_n,  t\eta_1=\eta_1, s\eta_2=\eta_2 \, \text{for}\, \eta_j \in \mathfrak k_j \cap \mathfrak u  $, $sq_\mathfrak u (\cdot)=q_\mathfrak u(s\cdot)$.
Hence,
$$S_w^H= s([q_\mathfrak u ( (\Psi_\lambda)_n) \cup \Delta(\mathfrak k/\mathfrak l,\mathfrak u)]\backslash \Phi(\mathfrak h/\mathfrak l,\mathfrak u))=s(\Psi_n^{H_0}) \cup \Delta(\mathfrak k/\mathfrak l,\mathfrak u).$$
Thus,
\begin{eqnarray*}\lefteqn{  \sum_{w \in W_K} \epsilon(w)\,\,\delta_{q_\mathfrak u(w\lambda)} \star y_{S_w^H}  =\sum_{s,t} \epsilon(st)\delta_{q_\mathfrak u(st\lambda)}\star y_{s(\Psi_n^{H_0}) \cup \Delta(\mathfrak k/\mathfrak l,\mathfrak u)}}  \\ & & \mbox{\phantom{xxxxx}} = \sum_{s,t} \epsilon(st)\delta_{q_\mathfrak u(s\bcancel{t}\lambda_1+\bcancel{s}t\lambda_2) }\star y_{s(\Psi_n^{H_0})}\star y_{ \Delta(\mathfrak k/\mathfrak l,\mathfrak u)}  \\ & & \mbox{\phantom{xxxxxxxxxx}} = \sum_{s} \epsilon(s)\delta_{(s\lambda_1,0)} \star y_{s(\Psi_n^{H_0})}\star \sum_{t} \epsilon(t) \delta_{q_\mathfrak u(t\lambda_2) }  \star y_{ \Delta(\mathfrak k/\mathfrak l,\mathfrak u)}. \end{eqnarray*}

Following \cite{H}, we write   the restriction of $\pi_{\lambda_2}^{K_2}$  to $L\cap K_2$ in the language of Dirac, Heaviside distributions in $i\mathfrak u^\star$, whence     \begin{eqnarray*} \lefteqn{\sum_{t \in W_{K_2}} \epsilon(t) \delta_{ q_\mathfrak u (t \lambda_{ 2 }) }  \star y_{ \Delta(\mathfrak k_2 /(\mathfrak k_2 \cap \mathfrak l),\mathfrak u)}} \\ & & \mbox{\phantom{xxxxxx}} =\sum_{\nu_2^\prime \in Spec_{L\cap K_2}(\pi_{\lambda_2}^{K_2})} m^{K_2 , L\cap K_2}(\lambda_2, \nu_2^\prime)\sum_{w_2 \in W_{K_2\cap L}} \epsilon(w_2) \delta_{(0,w_2 \nu_2^\prime)}. \end{eqnarray*} In the previous formula, we will apply   $ \Delta(\mathfrak k_2 /(\mathfrak k_2 \cap \mathfrak l),\mathfrak u)= \Delta(\mathfrak k /\mathfrak l,\mathfrak u)$.\\
We also write in the same language the restriction to $L$ of a Discrete Series $\pi_{(\lambda_1, \nu_2^\prime)}^{H_0}$ for $H_0$. This is.
$$\sum_{\nu \in i\mathfrak u^\star} m^{H_0,L}((\lambda_1,\nu_2^\prime), \nu) \, \delta_\nu =                                                        \sum_{s\in W_{K_1},t\in W_{K_2\cap L}} \epsilon(st)\delta_{ st(\lambda_1, \nu_2^\prime) }\star y_{st(\Psi_n^{H_0})}.$$
 Putting together the previous equalities, we obtain
 \begin{eqnarray*}\lefteqn{\sum_\mu  m^{G,H}( \lambda, \mu ) \,\,\delta_\mu }  \\ &   =\sum_{\nu_2^\prime \in Spec_{L\cap K_2}(\pi_{\lambda_2}^{K_2})} m^{K_2 , L\cap K_2}(\lambda_2, \nu_2^\prime) & \\ &   \mbox{\phantom{xxxxxxxxxxccccccccccc}}    \times \sum_{s\in W_{K_1},t\in W_{K_2\cap L}} \epsilon(st)\delta_{(st\lambda_1,st\nu_2^\prime) }\star y_{st(\Psi_n^{H_0})} &  \\ & \mbox{\phantom{xxxx}}
 =\sum_\nu (\sum_{\nu_2^\prime \in Spec_{L\cap K_2}(\pi_{\lambda_2}^{K_2})} m^{K_2 , L\cap K_2}(\lambda_2, \nu_2^\prime)   m^{H_0,L}((\lambda_1,\nu_2^\prime),\nu)) \,\,\delta_\nu . & \end{eqnarray*}

  Since,    the family $\{ \delta_\nu \}_{\nu \in i\mathfrak u^\star}$   is linearly independent,
 we have shown Lemma~\ref{lem:ddzhfinite}.
\end{proof}

 In order to conclude the proof of  the dimension equality  we state and prove a translation invariant property of multiplicity.
\begin{lem}  \label{lem:multshift} For a dominant integral $\mu \in i\mathfrak u^\star$    it holds: \\
\phantom{xxxxxxxxxxxxxx} $ m^{H_0 ,L} ((\lambda_1  ,  \nu_2^\prime )+\rho_n^H,  \mu +\rho_n^H )= m^{H_0 ,L} ((\lambda_1, \nu_2^\prime) ,\mu)  $.

\end{lem}
\begin{proof}
We recall that the hypothesis of the Lemma~\ref{lem:multshift} is: $(G,H)$ is a symmetric pair and $\pi_\lambda$ is $H$-admissible. The proof of Lemma~\ref{lem:multshift} is an application of Blattner's multiplicity formula, facts from \cite{HS} and     observations  from   \cite[Table 1,2,3]{Vaint}. In the next paragraphs we only consider systems $\Psi_\lambda$ so that $res_H(\pi_\lambda)$ is admissible. We   check the following statements by means of case by case analysis and  the tables in \cite{GW} and \cite{Vaint}.

OBS0.    Every   quaternionic system of positive roots that we are dealing with, satisfies the Borel de Siebenthal property, except for the algebra $\mathfrak{su}(2,2n)$ and the systems $\Psi_1 $ (see \ref{obs:wuforsum}).   Its Dynkin diagram is \xymatrix{
{\bullet}   \ar@{-}[r] & {\circ} \ar@{-}[r]& \cdots   \ar@{-}[r]& {\circ}   \ar@{-}[r]
& {\bullet}    }  .
 Bullet represents non compact roots, circle compact.

OBS1.  Always the systems $\Psi_{H,\lambda},  \Psi_{H_0,\lambda}$ have the same compact simple roots.

OBS2.  When $ \Psi_\lambda$ satisfies the Borel de Siebenthal property, it follows that both systems $\Psi_{H,\lambda},  \Psi_{H_0,\lambda}$ satisfy the Borel de Siebenthal property. Except for\footnote{We are indebted to the referee for this observation}: i) $\mathfrak g=\mathfrak{so}(2m,2n), \mathfrak h=\mathfrak{so}(2m,2)+\mathfrak{so}(2n-2)  $ and $\Psi_{H,\lambda} $;  ii) $\mathfrak g=\mathfrak{so}(2m,2n), \mathfrak h=\mathfrak{so}(2m,2n-2)+\mathfrak{so}( 2)  $ and $\Psi_{H_0,\lambda} $, for details see the proof of Remark~\ref{obs:wuforsum}.

OBS3.   $\Psi_\lambda$ satisfies the Borel de Siebenthal property except for two families of algebras: a) the algebra $ \mathfrak{su}(m,n) $ and the systems $\Psi_a,  a=1,\cdots, m-1$, $\tilde \Psi_b,  b=1, \cdots, n-1$,     the corresponding systems $\Psi_{H_0,\lambda}, \Psi_{H,\lambda}$ do not satisfy the Borel de Siebenthal property. They   have two noncompact simple roots; b) For the algebra $\mathfrak{so}(2m,2)$   each system   $\Psi_{\pm }$ does  not satisfy the Borel de Siebenthal property, however,  each associated  system $ \Psi_{SO(2m,1),\lambda},\Psi_{H_0,\lambda} $ satisfies the Borel de Siebenthal property.

OBS4. For   the pair $(\mathfrak{su}(2,2n),
\mathfrak{sp}(1,n))$.  $\Psi_1$ does not satisfy the Borel de Siebenthal property. Here,  $\Psi_{H,\lambda} =\Psi_{H_0,\lambda}$ and they   have Borel de Siebenthal property.

 OBS5. Summing up. Both systems  $\Psi_{H,\lambda},  \Psi_{H_0,\lambda}$ satisfy the Borel de Siebenthal property   except for: i) $(\mathfrak{su}(m,n), \mathfrak{su}(m,k)+ \mathfrak{su}(n-k)+\mathfrak  u(1))$, $(\mathfrak{su}(m,n), \mathfrak{su}(k,n)+ \mathfrak{su}( m-k)+\mathfrak  u(1))$ and the systems $\Psi_a, a=1,\cdots, m-1$, $\tilde \Psi_b, b=1, \cdots, n-1$;  ii) $\mathfrak g=\mathfrak{so}(2m,2n), \mathfrak h=\mathfrak{so}(2m,2)+\mathfrak{so}(2n-2)  $ and $\Psi_{H,\lambda} $;  iii) $\mathfrak g=\mathfrak{so}(2m,2n), \mathfrak h=\mathfrak{so}(2m,2n-2)+\mathfrak{so}( 2)  $ and $\Psi_{H_0,\lambda} $, for details see the proof of Remark~\ref{obs:wuforsum}.

\smallskip

To continue, we explicit   Blattner's formula according to our setting, we recall   fact's from   \cite{HS} and finish the proof of Lemma~\ref{lem:multshift} under the assumption $\Psi_{H_0,\lambda}$ satisfies the property of Borel de Siebenthal. Later on, we consider  other systems.

\medskip

Blattner's multiplicity formula applied to the $L$-type $V_{\mu +\rho_n^H}^L$ of $V_{(\lambda_1, \nu_2^\prime)+\rho_n^H }^{H_0}$ yields
\begin{eqnarray} \label{eqn:(A)}\lefteqn{   \dim  Hom_L(V_{\mu +\rho_n^H}^L, V_{(\lambda_1,\nu_2^\prime) +\rho_n^H}^{H_0}) } & & \\   & & =\sum_{s\in W_L} \epsilon (s) Q_0(s(\mu +\rho_n^H)-((\lambda_1,\nu_2^\prime) +\rho_n^H+\rho_n^{H_0})).\nonumber
\end{eqnarray}
   Here, $Q_0$ is the partition function associated to the set $\Phi_n (\mathfrak h_0) \cap \Psi_{H_0,\lambda}$.

 We recall a fact that allows to simplify the formula of above under our setting.

Fact 1: \cite[Statement 4.31]{HS}. For a system  $\Psi_{H_0,\lambda}$ having  the Borel de Siebenthal property,   it is shown that in the above sum, if the summand attached to $s \in W_L$ contributes nontrivially, then $ s$ belongs to the subgroup $W_U(\Psi_{H_0,\lambda})$   spanned by the reflections about the   compact simple roots in $\Psi_{H_0,\lambda}$.

 From  OBS1 we have
$W_U(\Psi_{H_0,\lambda})= W_U(\Psi_{H,\lambda})$. Owing that either $\Psi_{H_0,\lambda}$ or $\Psi_{H,\lambda}$ has the Borel de Siebenthal property we apply  \cite[Lemma 3.3]{HS}, whence $W_U(\Psi_{H,\lambda})=\{ s \in W_L : s(\Psi_{H,\lambda}\cap \Phi_n(\mathfrak h,\mathfrak u))   =\Psi_{H,\lambda}\cap \Phi_n(\mathfrak h,\mathfrak u)\}$. Thus, for $s \in W_U(\Psi_{H_0,\lambda})$ we have $s\rho_n^H= \rho_n^H$. We apply  the equality $s\rho_n^H= \rho_n^H$ in \ref{eqn:(A)} and we obtain $$\dim  Hom_L (V_{\mu+\rho_n^H}^L, V_{(\lambda_1, \nu_2^\prime)+\rho_n^H }^{H_0})= \sum_{s \in W_U(\Psi_{H_0,\lambda})} \epsilon(s) Q_0( s\mu -((\lambda_1, \nu_2^\prime )+\rho^{H_0}_n)).$$  Blattner's formula and the previous observations gives us that the right hand side of the above equality is

\phantom{cccccccccccccccc}$\dim  Hom_L(V_{\mu}^L, V_{(\lambda_1, \nu_2^\prime)}^{H_0})= m^{H_0,L}((\lambda_1, \nu_2^\prime), \mu)$,\\ whence, we have shown Lemma~\ref{lem:multshift} when  $\Psi_{H_0,\lambda}$ has the Borel de Siebenthal property.
\medskip

  In order to complete the proof of Lemma~\ref{lem:multshift}, owing to OBS5,  we are left to consider the pairs
$ (\mathfrak{su}(m,n), \mathfrak{su}(m,k)+ \mathfrak{su}(n-k)+\mathfrak  u(1)) $ as well as $ (\mathfrak{su}(m,n), \mathfrak{su}(k,n)+ \mathfrak{su}(m-k)+\mathfrak  u(1)) $ and the systems $\Psi_a, a=1,\cdots, m-1$, $\tilde \Psi_b, b=1, \cdots, n-1$; $\mathfrak g=\mathfrak{so}(2m,2n), \mathfrak h=\mathfrak{so}(2m,2)+\mathfrak{so}(2n-2)  $ and $\Psi_{H,\lambda} $;   $\mathfrak g=\mathfrak{so}(2m,2n), \mathfrak h=\mathfrak{so}(2m,2n-2)+\mathfrak{so}( 2)  $ and $\Psi_{H_0,\lambda} $, for details see the proof of Remark~\ref{obs:wuforsum}.

 The previous reasoning says  we are left     to extend Fact 1, \cite[Statement (4.31)]{HS}, for the pairs $ (\mathfrak{su}(m,n), \mathfrak{su}(m,k)+ \mathfrak{su}(n-k)+\mathfrak  u(1)) $ (resp.  $ (\mathfrak{su}(m,n), \mathfrak{su}(k,n)+ \mathfrak{su}(m-k)+\mathfrak  u(1)) $) and the systems $(\Psi_a)_{ a=1,\cdots, m-1}$  (resp. $(\tilde \Psi_b)_{ b=1, \cdots, n-1}$);        $\mathfrak g=\mathfrak{so}(2m,2n), \mathfrak h=\mathfrak{so}(2m,2n-2)+\mathfrak{so}( 2)  $ and $\Psi_{H_0,\lambda} $, for details see the proof of Remark~\ref{obs:wuforsum}. Under this setting   we first    verify:
   \begin{rmk}\label{obs:wuforsum} If $w \in W_L$ and $Q_0(w\mu -(\lambda + \rho_n^{H_0}))\not= 0$, then $ w \in W_U(\Psi_{H_0,\lambda})$.
   \end{rmk}

   To show {\it Remark}~\ref{obs:wuforsum} we follow \cite{HS}. To begin with we consider the pairs associated to $\mathfrak g= \mathfrak {su}(m,n)$.  We fix as Cartan subalgebra   $\mathfrak t$ of $\mathfrak {su}(m,n)$ the set of diagonal matrices in $\mathfrak {su}(m,n).$   For certain orthogonal basis $\epsilon_1, \dots ,\epsilon_m, \delta_1, \dots, \delta_n$ of the dual vector space to the subspace of diagonal matrices in $\mathfrak{gl}(m+n, \mathbb C),$ we may, and will choose  $\Delta =\{ \epsilon_r - \epsilon_s, \delta_p -\delta_q, 1 \leq  r < s \leq m, 1 \leq p < q \leq n \},$ the set of noncompact roots is $ \Phi_n= \{ \pm (\epsilon_r - \delta_q) \}.$ We recall the  positive roots systems for $\Phi(\mathfrak g, \mathfrak t)$  containing $\Delta$ are in a bijective  correspondence with the totality of lexicographic orders for the basis $\epsilon_1, \dots ,\epsilon_m, \delta_1, \dots, \delta_n$ which contains the "suborder" $\epsilon_1 > \dots > \epsilon_m, \,\, \delta_1 > \dots > \delta_n.$ The two holomorphic systems correspond to the orders $\epsilon_1 > \dots > \epsilon_m > \delta_1 > \dots > \delta_n ; \,\,  \delta_1 > \dots > \delta_n >\epsilon_1 > \dots > \epsilon_m.$ We fix $1 \leq a \leq m-1, $ and let $ \Psi_a$ denote the set of positive roots associated to the order $\epsilon_1 >     \dots > \epsilon_a >\delta_1 > \dots > \delta_n > \epsilon_{a+1} > \dots > \epsilon_m$.   We fix $1 \leq b \leq n-1 $ and let $\tilde{\Psi}_b$ denote the set of positive roots associated to the order $\delta_1 >  \dots > \delta_b >\epsilon_1 > \dots > \epsilon_m > \delta_{b+1} > \dots > \delta_n $. Since, $\mathfrak h=\mathfrak{su}(m,k)+\mathfrak u(n-k)$, $\mathfrak h_0=\mathfrak{su}(m,n-k)+\mathfrak u(k)$.  The root systems for  $(\mathfrak h, \mathfrak t)$ and  $(\mathfrak h_0, \mathfrak t)$ respectively are:
       \begin{multline*}
     \Phi(\mathfrak h, \mathfrak t)=\{ \pm (\epsilon_r -\epsilon_s), \pm(\delta_p -\delta_q), \pm (\epsilon_i -\delta_j),
    1 \leq r < s \leq m,  \\  1 \leq p <  q \leq k, \,\,   or, \,\, k+1 \leq p <  q \leq n, 1\leq i \leq m,  1 \leq j \leq k \}.
    \end{multline*}
   \begin{multline*} \Phi(\mathfrak h_0, \mathfrak t)=\{ \pm (\epsilon_r -\epsilon_s), \pm(\delta_p -\delta_q), \pm (\epsilon_i -\delta_j),  1 \leq r < s \leq m, \\  1 \leq p <  q \leq k\,\, or\,\, k+1 \leq p < q \leq n,  1 \leq i \leq m,  k+1 \leq j \leq n \}.
   \end{multline*}

       The system $\Psi_{H,\lambda}=\Psi_\lambda \cap \Phi(\mathfrak h,\mathfrak t),$  $\Psi_{H_0,\lambda}=\Psi_\lambda \cap \Phi(\mathfrak h_0,\mathfrak t)$ which correspond to  $\Psi_a$ are the system associated to the respective lexicographic orders  $$ \epsilon_1> \dots > \epsilon_a > \delta_1> \dots> \delta_k > \epsilon_{a+1}  > \dots > \epsilon_m,\, \delta_{k+1}> \dots> \delta_n.  $$ $$\epsilon_1> \dots > \epsilon_a > \delta_{k+1}> \dots> \delta_n > \epsilon_{a+1}  > \dots > \epsilon_m, \, \delta_1> \dots> \delta_k  .$$

Without loss of generality, and in order to simplify notation,   we may and will assume $\mathfrak h_0=\mathfrak{su}(m,n)$, $\Psi_{H_0,\lambda}=\Psi_a$ and we show {\it Remark}~\ref{obs:wuforsum} for $\mathfrak{su}(m,n)$ and $\Psi_a$.   $Q$   denotes the partition function for $\Psi_a \cap \Phi_n$.

  The subroot system spanned by the compact simple roots in $\Psi_a$ is

  $\Phi_U =\{ \epsilon_i -\epsilon_j, 1\leq i\leq a, 1\leq j\leq a \,\text{or}\,  a+1\leq i\leq m, a+1\leq j\leq m  \} \cup \{ \delta_i -\delta_j, 1\leq i\not= j\leq n  \}.$

   $\Psi_a \cap \Phi_c \backslash \Phi_U =\{ \epsilon_i -\epsilon_j, 1\leq i\leq a, a+1\leq j\leq m \} $.

   $\Psi_a \cap \Phi_n =\{\epsilon_i -\delta_j, \delta_j -\epsilon_r, 1\leq i \leq a, a+1\leq r \leq m, 1\leq j \leq n \}$.

   $2\rho_n^{H_0}= n(\epsilon_1 +\dots +\epsilon_a)- n(\epsilon_{a+1} +\dots +\epsilon_m)-(a-(m-a))(\delta_1 +\dots +\delta_n)$.

   A finite sum of non compact roots in $\Psi_a$ is equal a to

   $B= \sum_{1\leq j \leq a} A_j \epsilon_j -\sum_{a+1\leq i \leq m} B_i \epsilon_i +\sum_r C_r \delta_r $ with $A_j, B_i$ non negative numbers.

   Let $ w \in W_L $ so that $Q(w\mu -(\lambda + \rho_n^{H_0}))\not= 0$. Hence, $\mu =w^{-1}(\lambda +\rho_n^{H_0} +B)$, with $B$ a sum of roots in $\Psi_a\cap \Phi_n$.  Thus, $w^{-1}$ is the unique element in $W_L$ that takes $\lambda +\rho_n^{H_0} +B $  to the Weyl chamber determined by $\Psi_a \cap \Phi_c$.

   Let $w_1 \in W_U(\Psi_a)$ so that $ w_1(\lambda +\rho_n +B)$ is $  \Psi_a \cap \Phi_U  $-dominant. Next we verify $w_1(\lambda +\rho_n +B)$ is $\Psi_a \cap \Phi_c$-dominant. For this, we fix $\alpha \in  \Psi_a  \cap\Phi_c    \backslash \Phi_U$ and check $(w_1(\lambda +\rho_n +B),\alpha)>0 $. $\alpha =\epsilon_i -\epsilon_j , i\leq a < j$,  and $  w_1 \in W_U(\Psi_a)$,   hence, $w_1^{-1}(\alpha)=\epsilon_r -\epsilon_s, r\leq a < s$ belongs to $\Psi_a$. Thus,
$(w_1(\lambda +\rho_n +B),\alpha)=( \lambda +\rho_n +B, w_1^{-1}\alpha)= ( \lambda , w_1^{-1}\alpha) +(\rho_n, w_1^{-1}\alpha) +(B, w_1^{-1}\alpha)= ( \lambda , w_1^{-1}\alpha) +n-(-n) +A_r +B_s $, the first summand is positive because $\lambda$ is $\Psi_a$-dominant, the third and fourth are nonnegative. Therefore, $w^{-1}=w_1$ and  we have shown {\it Remark}~\ref{obs:wuforsum} for $\mathfrak g=\mathfrak {su}(m,n)$.

\smallskip

To follow we fix $\mathfrak g=\mathfrak{so}(2m,2n), n>1, \mathfrak h= \mathfrak{so}(2m,2n-2)+\mathfrak{so}(2), \mathfrak h_0=\mathfrak{so}(2m,2 )+ \mathfrak{so}(2n-2)$, then,    for certain orthogonal basis $\epsilon_1, \dots ,\epsilon_m, \delta_1, \dots, \delta_n$ of the dual vector space to $\mathfrak t$, the system of positive roots $\Psi_\lambda =\{ \epsilon_i \pm \epsilon_j, \delta_r \pm \delta_s, \epsilon_a \pm \delta_b, 1 \leq i<j\leq m, 1 \leq r<s\leq n, 1 \leq a\leq m, 1 \leq b \leq n\}$ is so that $K_1(\Psi_\lambda)=SO(2m)\subset H$.  Since $n>1$ this is a Borel de Siebenthal system.

$\Psi_{H,\lambda}=\{ \epsilon_i \pm \epsilon_j, \delta_r \pm \delta_s, \epsilon_a \pm \delta_b, 1 \leq i<j\leq m, 1 \leq r<s\leq n-1, 1 \leq a\leq m, 1 \leq b \leq n-1  \}$.

  For $n>2 $, $\Psi_{H,\lambda} $ is Borel de Siebenthal, for $n=2$, $\Psi_{H,\lambda}$ is not Borel de Siebenthal.

$\Psi_{H_0,\lambda}=\{ \epsilon_i \pm \epsilon_j, \delta_r \pm \delta_s, \epsilon_a \pm \delta_n, 1 \leq i<j\leq m, 1 \leq r<s\leq n-1, 1 \leq a\leq m      \}$. Since $n>1$, $\Psi_{H_0,\lambda}$ is not a Borel de Siebenthal system. The simple roots for $\Psi_{H_0,\lambda}$ are $\epsilon_1 -\epsilon_2, \cdots, \epsilon_{m-1} -\epsilon_m, \epsilon_m \pm \delta_n, \delta_1 - \delta_2, \cdots, \delta_{n-3}- \delta_{n-2}, \delta_{n-2} \pm \delta_{n-1}$.

$ \Phi_U =\{\epsilon_i - \epsilon_j , \delta_r \pm \delta_s ,  1 \leq i<j\leq m, 1 \leq r<s\leq n-1 \}$.

$W_L=span\{ S_{\epsilon_i \pm \epsilon_j}, S_{\delta_r \pm \delta_s},  1 \leq i<j\leq m, 1 \leq r<s\leq n-1 \}$.

$W_U(\Psi_{H_0,\lambda})=span\{S_{\epsilon_i - \epsilon_j}, S_{\delta_r \pm \delta_s},  1 \leq i<j\leq m, 1 \leq r<s\leq n-1 \}$.

$\Psi_\lambda \cap \Phi_c \backslash \Phi_U=\{\epsilon_i + \epsilon_j , 1\leq i<j\leq m  \}.$

$Q_0$ is the partition function associated to $\Psi_{H_0, \lambda}\cap \Phi_n
 =\{ \epsilon_i \pm \delta_n, 1\leq i \leq m \}$.

  $\rho_n^{H_0}=\epsilon_1 +\dots +\epsilon_m$.

A finite sum of noncompact roots $B$ in $\Psi_{H_0,\lambda}$ is equal to   \\
  $B=\sum_{j=1}^{j=m} (a_j +b_j)\epsilon_j +(\sum_{1\leq j \leq m} (a_j -b_j)) \delta_n, a_j \geq 0, b_j \geq 0 \, \forall j $.

  Let $ w \in W_L $ so that $Q_0(w\mu -(\lambda + \rho_n))\not= 0$. Hence, $\mu =w^{-1}(\lambda +\rho_n +B)$, with $B$ a sum of roots in $\Psi_\lambda\cap \Phi_n$.  Thus, $w^{-1}$ is the unique element in $W_L$ that takes $\lambda +\rho_n +B $  to the Weyl chamber determined by $\Psi_\lambda \cap \Phi_c$.

   Let $w_1 \in W_U(\Psi_{H_0,\lambda})$ so that $ w_1(\lambda +\rho_n +B)$ is $  \Psi_{H_0,\lambda} \cap \Phi_U  $-dominant. Next we verify $w_1(\lambda +\rho_n +B)$ is $\Psi_{H_0,\lambda} \cap \Phi_c$-dominant. For this, we fix $\alpha \in  \Psi_{H_0,\lambda}  \cap\Phi_c    \backslash \Phi_U$ and check $(w_1(\lambda +\rho_n +B),\alpha)>0 $. $\alpha =\epsilon_i +\epsilon_j  $,  and $  w_1 \in W_U(\Psi_{H_0,\lambda})$,   hence, $w_1^{-1}(\alpha)=\epsilon_r +\epsilon_s $ belongs to $\Psi_{H_0,\lambda}$. Thus,
$(w_1(\lambda +\rho_n +B),\alpha)=( \lambda +\rho_n +B, w_1^{-1}\alpha)= ( \lambda , w_1^{-1}\alpha) +(\rho_n, w_1^{-1}\alpha) +(B, w_1^{-1}\alpha) $, the first summand is positive because $\lambda$ is $\Psi_{H_0,\lambda}$-dominant, the second and third   are nonnegative. Therefore, $w^{-1}=w_1$ and  we have shown {\it Remark}~\ref{obs:wuforsum} for $\mathfrak g=\mathfrak {so}(2m,2n)$.

 Whence, we have concluded the proof of Lemma~\ref{lem:multshift}. \end{proof}

\begin{lem}\footnote{We thank the referee for improving this Lemma}\label{lem:equalm} We recall  $\Lambda_2=\lambda_2 +(\rho_n^\lambda)_2$.  We claim: \\
\phantom{xxxxxxxxxxx} $m^{K_2,L\cap K_2}(\Lambda_2,\nu_2^\prime +q_\mathfrak u(\rho_n^\lambda)_2)=m^{K_2,L\cap K_2}(\lambda_2,\nu_2^\prime)$.\\
\phantom{xxxxxxxxxxx} $Spec_{L\cap K_2}(\pi_{\Lambda_2}^{K_2})= Spec_{L\cap K_2}(\pi_{\lambda_2}^{K_2})+ q_\mathfrak u (\rho_n^\lambda)_2$.
\end{lem}

In fact, when $\Psi_\lambda$ is holomorphic, $\rho_n^\lambda$ is in $\mathfrak z_\mathfrak k =\mathfrak k_1$ hence $(\rho_n^\lambda)_2=0$. In \cite{Vaint} it is shown that when $K$ is semisimple $(\rho_n^\lambda)_2=0$. Actually, this is so, owing that the simple roots for  $\Psi_\lambda \cap \Phi(\mathfrak k_2,\mathfrak t_2)$ are simple roots for $\Psi_\lambda$ and that $\rho_n^\lambda$ is orthogonal to every compact simple root for $\Psi_\lambda$. For general $\mathfrak g$, the previous considerations together with that $(\rho_n^\lambda)_2$ is orthogonal to $\mathfrak k_1$ yields that $(\rho_n^\lambda)_2$ belongs to the dual of  the center of $\mathfrak l \cap \mathfrak k_2$. From Tables 1,2,3 we deduce we are left to analyze $(\rho_n^\lambda)_2$ for $\mathfrak{su}(m,n), \mathfrak{so}(2m,2)$.
For $\mathfrak{so}(2m,2)$ we follow the notation in \ref{tauisirred}, then $\mathfrak t_1^\star=span(\epsilon_1,\dots,\epsilon_m), \mathfrak t_2^\star=span(\delta_1)$ and $\rho_n^{\Psi_{\pm}}=c(\epsilon_1+\dots +\epsilon_m) \in \mathfrak t_1^\star$. For $\mathfrak{su}(m,n)$ we follow the notation in Lemma~\ref{lem:multshift}.  It readily follows that for $1\leq a<m$,   $\rho_n^{\Psi_a}=\frac{n}{m}((m-a)(\epsilon_1+\dots +\epsilon_a) -a(\epsilon_{a+1}+\dots +\epsilon_m))+\frac{2a-m}{2m}((n(\epsilon_1+\dots +\epsilon_m)-m(\delta_1+\dots +\delta_n))$. The first summand is in $(\mathfrak t\cap \mathfrak{su}(m))^\star$, the second summand belongs to $\mathfrak z_\mathfrak k^\star$, thus,
$(\rho_n^{\Psi_a})_2=0$ if and only if $2a=m$. Hence, for $(\mathfrak{su}(2,2m), \mathfrak{sp}(1,m))$, we have $(\rho_n^{\Psi_1})_2=0$. For $(\mathfrak{su}(m,n), \mathfrak{su}(m,k)+ \mathfrak{su}( n-k)+\mathfrak z_\mathfrak l) $, always,  $(\rho_n^{\Psi_a})_2$ determines a character of the center of $\mathfrak k$. In this case,  $\lambda_2=\Lambda_2$ except for  $(\mathfrak{su}(m,n), \mathfrak{su}(m,k)+ \mathfrak{su}( n-k)+\mathfrak z_\mathfrak l), \Psi_a $ and $a\not=  2m$, hence, $\pi_{\Lambda_2}^{K_2}$ is equal to $\pi_{\lambda_2}^{K_2}$ times a central character of $K$. Thus, the equality holds.

As a corollary we obtain: $(\rho_n^\lambda)_2 \not= 0$ forces $\mathfrak g= \mathfrak{ su}(m,n), \mathfrak h=\mathfrak s(\mathfrak{ u}(m,k)+\mathfrak{ u}(n-k)), m>1,n>1, 1\leq k \leq n-1$, $\mathfrak u=\mathfrak t$. Thus, we always have $q_\mathfrak u((\rho_n^\lambda)_2)=(\rho_n^\lambda)_2$.

\subsubsection{Conclusion proof of  Theorem~\ref{prop:equaldim}} We just put
 together Lemma~\ref{lem:multshift}, Lemma~\ref{lem:ddzhfinite} and Lemma~\ref{lem:equalm}, hence,   we obtain the equalities we were searching for.  This concludes the proof of Theorem~\ref{prop:equaldim}.

\subsubsection{Existence of $D$}
To follow we show the existence of the isomorphism $D$ in Theorem~\ref{prop:equaldim} $ii)$ and derive the decomposition into irreducible factors of the semisimple  $\mathfrak h_0$-module $\mathcal U(\mathfrak h_0)W$. On the mean time, we also consider some particular cases of Theorem~\ref{prop:equaldim}. Before, we proceed we comment on the structure of the representation $\tau$.

\subsubsection{Representations $\pi_\lambda$ so that $res_L(\tau)$ is irreducible} \label{tauisirred}  Under our $H$-admissibility hypothesis of $\pi_\lambda$ we analyze the cases so that the representation $res_L(\tau)$ is irreducible. The next structure statements are verified in \cite{Vaint}.  To begin with,  we recall the decomposition $K=K_1Z_K K_2 $, (this is not a direct product!, $Z_K$ connected center of $K$)  and the direct product $K=K_1K_2$, we also recall that actually, either $K_1$ or $K_2$ depend on $\Psi_\lambda$.  When $\pi_\lambda$ is a holomorphic representation $K_1=Z_K$ and $\mathfrak k_2=[\mathfrak k,\mathfrak k]$;   when, $Z_K$ is nontrivial and $\pi_\lambda$ is not a holomorphic representation we have $Z_K \subset K_2$; for $\mathfrak g=\mathfrak{su}(m,n), \mathfrak h=\mathfrak{su}(m,k)+\mathfrak{su}(n-k)+\mathfrak{z}_L$,  we have $\mathbf T\equiv Z_K \subset Z_L\equiv \mathbf T^2$. Here, $Z_K\subset L$, and, $\tau_{\vert_L}$ irreducible, forces $\tau=\pi_{\Lambda_1}^{SU(m)}\boxtimes \pi_\chi^{Z_K}\boxtimes \pi_{\rho_{SU(n)}}^{SU(n)}$ ; for $\mathfrak g\ncong \mathfrak{su}(m,n)$  and $G/K$ a Hermitian symmetric space, we have to consider the next two examples.

 For both cases we have  $K_2 =Z_K (K_2)_{ss}$ and $Z_K \nsubseteq L$.

 1) When $\mathfrak g=\mathfrak{so}(2m,2), \mathfrak h=\mathfrak{so}(2m,1)$ and $\Psi_\lambda =\Psi_{\pm}$, then $\mathfrak k_1=\mathfrak{so}(2m)$, $\mathfrak k_2=\mathfrak z_K$ and obviously $res_L(\tau)$ is always an irreducible representation. Here, $\pi_{\Lambda_2}^{K_2}$ is one dimensional representation.

\smallskip

2) When $\mathfrak g=\mathfrak{su}(2,2n), \mathfrak h=\mathfrak{sp}(1,n)$, $\Psi_\lambda=\Psi_1$,  then $\mathfrak k_1=\mathfrak{su}_2(\alpha_{max})$, $\mathfrak k_2=\mathfrak{su}(2n)+\mathfrak z_\mathfrak k$,  $\mathfrak l \equiv \mathfrak{su}_2(\alpha_{max}) +\mathfrak{sp}(n)$. Examples of  $\tau_{\vert_L}$ irreducible are\footnote{This was pointed out by the referee},  $\tau=\pi_{b\Lambda_1}^{K_1}\boxtimes \pi_\chi^{Z_K} \boxtimes \pi_{\rho_{SU(2n)}+a\widetilde{\Lambda}_1 }^{ SU(2n)  }$, $\tau=\pi_{b\Lambda_1}^{K_1}\boxtimes \pi_\chi^{Z_K} \boxtimes \pi_{\rho_{ SU(2n)}+a \widetilde{\Lambda}_{2n-1} }^{ SU(2n)  }, a\geq 0, b\geq 2$, $\Lambda_1$ (resp. $\widetilde \Lambda_1$) is the highest weight of the  first fundamental representation of $\mathfrak{su}_2$ (resp. $\mathfrak{su}(2n)$), $\widetilde{\Lambda}_{2n-1}$ is the highest weight of the dual representation to the first fundamental representation for $\mathfrak{su}(2n)$.

\smallskip

For $ \mathfrak g=\mathfrak{so}(2m,2n), \mathfrak h=\mathfrak{so}(2m,2n-1), n>1,$ $\Psi_\lambda \cap \Phi_n =\{ \epsilon_i \pm \delta_j \},$ $\mathfrak k_1=\mathfrak{so}(2m)$, and if $\lambda$ is so that  $\lambda +\rho_n^\lambda=ic(\tau) =(\sum c_i \epsilon_i, k(\delta_1 +\dots +\delta_{ n-1}  \pm \delta_n) ) +\rho_K $,  then $res_L(\tau)$ is irreducible and $\pi_{\Lambda_2}^{K_2}=\pi_{ k(\delta_1 +\dots +\delta_{ n-1}  \pm \delta_n)   +\rho_{K_2}}^{K_2}$ is not a one dimensional representation for  $k>0$. It follows from the classical branching laws that these are the unique $\tau's$ such that $res_L(\tau)$ is irreducible.

\smallskip
We do not know the pairs $(\pi_\lambda, \tau) $ so that $\pi_\lambda$ is $H$-admissible and
   $res_L(\tau)$ is   irreducible. We believe that for $\mathfrak g \ncong \mathfrak{so}(2m,2n)$ or $\mathfrak g \ncong \mathfrak{su}(m, n)$  we could  conclude that $\tau$ is the tensor product of a irreducible representation of $K_1$ times a one dimensional representation of $K_2$. That is,
$\tau \equiv \pi_{\phi_1}^{K_1}\boxtimes \pi_{\rho_{K_2}}^{K_2}\otimes \pi_{\chi}^{Z_{K_2}}$.

\smallskip
In \ref{sub:existencep} we show that whenever a symmetric pair $(G,H)$ is so that some Discrete Series for $G$ is $H$-admissible, then there exists $H$-admissible Discrete Series for $G$ so that its lowest $K$-type restricted to $L$  is irreducible.

\subsubsection{Analysis of $\mathcal U(\mathfrak h_0)W$, $\mathcal L_\lambda$, existence of $D$,   case $\tau_{\vert_L}$ is irreducible}\label{tauirred} As before, our hypothesis is $(G,H)$ is a symmetric space and $\pi_\lambda^G$ is $H$-admissible. For this paragraph we add the hypothesis  $\tau_{\vert_L}=res_L(\tau)$ is irreducible. We recall that $\mathcal U(\mathfrak h_0)W=L_{\mathcal U (\mathfrak h_0)} (H^2(G,\tau)[W])$, $\mathcal L_\lambda= \oplus_{\mu \in Spec_H(\pi_\lambda)} H^2(G,\tau)[V_\mu^H][V_{\mu +\rho_n^\mu}^L] $. We claim:

\smallskip

   a) if a $H$-irreducible discrete factor of $V_\lambda^G$ contains a copy of $\tau_{\vert L}$, then   $\tau_{\vert L}$ is the lowest $L$-type of such factor.

    b) the multiplicity of $res_L(\tau)$ in $H^2(G,\tau)$ is one.

c) $\mathrm{Cl}(\mathcal U(\mathfrak h_0)W)$ is $H_0$-equivalent to $H^2(H_0,\tau)$.

d) $\mathcal L_\lambda$ is $L$-equivalent to $H^2(H_0,\tau)_{L-fin}$.

e) $\mathcal L_\lambda$ is $L$-equivalent to $ \mathcal U(\mathfrak h_0)W$. Thus, $D$ exists.

\smallskip

We rely on:

\begin{rmk}\label{ktypessch} 1) Two Discrete Series for $H$ are equivalent if and only if their respective lowest $L$-types are equivalent. \cite{VoT}.

2)\label{ktypessch} For any Discrete Series $\pi_\lambda$,  the highest weight (resp. infinitesimal character) of any $K$-type is equal to the highest weight of the lowest $K$-type (resp. the infinitesimal character of the lowest $K$-type) plus a sum of noncompact roots in $\Psi_\lambda$ \cite[Lemma 2.22]{Sc}.
\end{rmk}
 From now on $ic(\phi)$ denotes the infinitesimal character (Harish-Chandra parameter) of the representation $\phi$.

Let $V_\mu^H$ a discrete factor for $res_H(\pi_\lambda)$ so that $\tau_{\vert_L}$ is a $L$-type.
Then, Theorem~\ref{prop:equaldim} implies $V_{\mu+\rho_n^H}^L$ is a $L$-type for $H^2(H_0,\tau)$. Hence, after we apply {\it Remark}~\ref{ktypessch},  we obtain

$\mu +\rho_n^H + B_1 =ic(\tau_{\vert_L})$ with $B_1$ a sum of roots in $\Psi_{H,\lambda}\cap \Phi_n$.

  $\mu +\rho_n^H=ic(\tau_{\vert_L})+ B_0$ with $B_0$ a sum of roots in $\Psi_{H_0,\lambda}\cap \Phi_n$.

Thus, $B_0 +B_1 = 0$, whence $B_0=B_1=0$  and $\mu +\rho_n^H
 =ic(\tau_{\vert_L})$, we have verified a).

Due to $H$-admissibility hypothesis, we have $\mathcal U(\mathfrak h)W$ is a finite sum of irreducible underlying modules of Discrete Series for $H$. Now, Corollary 1 to Lemma~\ref{lem:injecrh}, yields that  a copy of a $V_\mu^H$  contained in $\mathcal U(\mathfrak h)W$  contains a copy of $V_\lambda^G[W]$. Thus, a) implies $\tau_{\vert_L}$ is the lowest $L$-type of such $V_\mu^H$. Hence, $H^2(H,\tau)$ is nonzero. Now, Theorem~\ref{prop:equaldim} together with the fact that the lowest $L$-type of a Discrete Series has multiplicity one  yields that $\dim  Hom_H(H^2(H,\tau), V_\lambda^G) =1$. Also, we obtain $\dim  Hom_{H_0}(H^2(H_0,\tau), V_\lambda^G) =1$. Thus, whenever $\tau_{\vert L}$ occurs in $res_L(V_\lambda^G)$, we have  $\tau_{\vert_L}$ is realized in $V_\lambda^G[W]$. In other words, the isotypic compoent $V_\lambda^G[\tau_{\vert_L}]\subset V_\lambda^G[W]$.  Hence, b) holds.

Owing our hypothesis, we may write $\mathcal U(\mathfrak h_0)W=N_1+...+N_k$, with each $N_j$ being the underlying Harish-Chandra module of a irreducible square integrable representation for $H_0$. Since Lemma~\ref{lem:injecrh} shows $r_0$ is injective in $\mathcal U(\mathfrak h_0)W$,  we have $r_0(\mathrm{Cl}(N_j))$ is a Discrete Series in $L^2(H_0 \times_{res_L(\tau)} W)$, hence Frobenius reciprocity implies    $\tau_{\vert_L}$ is a $L$-type for $N_j$. Hence, b) and a) yields $\mathcal U(\mathfrak h_0)W$ is $\mathfrak h_0$-irreducible and c) follows.

By definition, the subspace $\mathcal L_\lambda$ is the linear span of the subspaces $V_\lambda^G[V_\mu^H][V_{\mu+\rho_n^H}^L]$ with $\mu \in Spec_H(\pi_\lambda)$. Since, \begin{eqnarray*}\lefteqn{\dim Hom_L(V_{\mu +\rho_n^H}^L ,  V_\lambda^G[V_\mu^H][V_{\mu+\rho_n^H}^L]) = \dim  Hom_H(V_\mu^H, V_\lambda^G)}    & & \\ & &   =\dim  Hom_L(V_{\mu +\rho_n^H}^L, H^2(H_0, \tau))=\dim Hom_L ( V_{\mu +\rho_n^H}^L,  H^2(H_0, \tau)[V_{\mu +\rho_n^H}^L]), \end{eqnarray*} and, both $L$-modules are isotypical,   d) follows. Finally,  e) follows from c) and d).

\medskip

 Whenever   $\pi_{\Lambda_2}^{K_2}$  is the trivial representation,    Theorem~\ref{prop:equaldim} and Lemma~\ref{lem:multshift} justifies:
\begin{eqnarray*}\lefteqn{\dim  Hom_H( V_\mu^H, V_\lambda^G)= \dim  Hom_L(  V_{\mu+\rho_n^H}^L , V_{(\lambda_1,\rho_{K_2\cap L})+\rho_n^H }^{H_0})} & & \\ & & = \dim  Hom_L(  V_{\mu+\rho_n^H}^L , H^2(H_0,\tau))= \dim  Hom_L(  V_{\mu}^L , V_{(\lambda_1,\rho_{K_2\cap L})}^{H_0}),\end{eqnarray*}

\noindent
       the infinitesimal character of  $H^2(H_0, \tau)$ is $(\lambda_1, \rho_{K_2\cap L})+q_\mathfrak u(\rho_n^\lambda)-\rho_n^{H_0}=(\lambda_1, \rho_{K_2\cap
        L})+\rho_n^H$.
Thus,  $H^2(H_0, \tau)\equiv V_{(\lambda_1, \rho_{K_2\cap L})+\rho_n^{H}}^{H_0} $.

\subsubsection{
Analysis of $\mathcal U(\mathfrak h_0)W$, $\mathcal L_\lambda$, existence of $D$, for general $(\tau, W)$}

  We recall that by definition,   $\mathcal L_\lambda= \oplus_{\mu \in Spec_H(\pi_\lambda)} H^2(G,\tau)[V_\mu^H][V_{\mu +\rho_n^\mu}^L] $,  $\mathcal U(\mathfrak h_0)W=L_{\mathcal U(\mathfrak h_0)} (H^2(G,\tau)[W])$.
\begin{prop}\label{prop:gentau}  The hypothesis is: $(G,H)$ is a symmetric pair and $\pi_\lambda$ a $H$-admissible square integrable representation  of lowest $K$-type $(\tau,W)$.   We write\\ \phantom{xxxxx} $res_L(\tau)=q_1 \sigma_1 +\cdots + q_r \sigma_r, $ with $(\sigma_j, Z_j)\in \widehat L, q_j >0$. Then,

   a) if a $H$-irreducible discrete factor for $res_H(\pi_\lambda)$  contains a copy of $\sigma_j$, then   $\sigma_j$ is the lowest $L$-type of such factor.

    b) the multiplicity of $\sigma_j $ in $res_L(H^2(G,\tau))$ is equal to $q_j$.

c) $r_0: \mathrm{Cl}(\mathcal U(\mathfrak h_0)W)  \rightarrow \mathbf H^2(H_0,\tau)$
is a $H_0$-equivalence.

d) $\mathcal L_\lambda$ is $L$-equivalent to $\mathbf H^2(H_0,\tau)_{L-fin}$.

e) $\mathcal L_\lambda$ is $L$-equivalent to $ \mathcal U(\mathfrak h_0)W$. Therefore, $D$ exists.

\end{prop}

\begin{proof} Let $V_\mu^H$ a discrete factor for $res_H(\pi_\lambda)$ so that some irreducible factor of $\tau_{\vert_L}$ is a $L$-type.
Then, Theorem~\ref{prop:equaldim} implies $V_{\mu+\rho_n^H}^L$ is a $L$-type for $\mathbf H^2(H_0,\tau)=\oplus_j q_j H^2(H_0,\sigma_j)$. Let's say $V_{\mu+\rho_n^H}^L$ is a subrepresentation of $H^2(H_0,\sigma_i)$.  We recall $ic(\phi)$ denotes the infinitesimal character (Harish-Chandra parameter) of the representation $\phi$. Hence, after we apply {\it Remark}~\ref{ktypessch} we obtain

$\mu +\rho_n^H + B_1 =ic(\sigma_j)$ with $B_1$ a sum of roots in $\Psi_{H,\lambda}\cap \Phi_n$.

  $\mu +\rho_n^H=ic(\sigma_i)+ B_0$ with $B_0$ a sum of roots in $\Psi_{H_0,\lambda}\cap \Phi_n$.

Thus, $B_0 +B_1 = ic(\sigma_j)-ic(\sigma_i)  $. Now, since $\mathfrak k=\mathfrak k_1 +\mathfrak k_2$, $\mathfrak k_1 \subset \mathfrak l$, $\tau=\pi_{\Lambda_1}^{K_1} \boxtimes \pi_{\Lambda_2}^{K_2}$, we may write  $\sigma_s=\pi_{\Lambda_1}^{K_1}\boxtimes \phi_s, $ with $\phi_s \in \widehat {L\cap K_2}$, hence, $ic(\sigma_j)-ic(\sigma_i)=ic(\phi_j)-ic(\phi_i)$. Since, each $\phi_t$ is a irreducible factor of $res_{L\cap K_2}(\pi_{\Lambda_2}^{K_2})$, we have $ic(\phi_j)-ic(\phi_i)$ is equal to the difference of two sum  of roots in $\Phi (\mathfrak k_2, \mathfrak t \cap \mathfrak k_2)$. The hypothesis forces that the simple roots for $\Psi_\lambda \cap \Phi (\mathfrak k_2, \mathfrak t \cap \mathfrak k_2)$ are compact simple roots for $\Psi_\lambda$ (see \cite{DV}) whence $ic(\sigma_j)-ic(\sigma_i)$ is a linear combination of compact simple roots for $\Psi_\lambda$. On the other hand, $B_0 +B_1$ is a sum of  noncompact roots in $\Psi_\lambda$. Now $B_0 +B_1  $ can not be a linear combination of compact simple roots, unless  $B_0=B_1=0$. Thus, $ic(\sigma_i)=ic(\sigma_j)$ and $Z_j \equiv V_{\sigma_j}^L$ is the lowest $L$-type of   $V_\mu^H$, we have verified a).

\smallskip
Due to $H$-admissibility hypothesis, we have $\mathcal U(\mathfrak h)W$ is a finite sum of irreducible underlying Harish-Chandra modules of Discrete Series for $H$. Thus,  a copy of certain $V_\mu^H$  contained in $\mathcal U(\mathfrak h)W$  contains $W[\sigma_j]$. Whence, $\sigma_j$ is the lowest $L$-type of such $V_\mu^H$. Hence, $H^2(H,\sigma_j)$ is nonzero and it is equivalent to a subrepresentation of $\mathrm{Cl}(\mathcal U(\mathfrak h)W)$.

We claim, for $i\not= j$, no $\sigma_j$ is a $L$-type of $\mathrm{Cl}(\mathcal U(\mathfrak h)W)[H^2(H,\sigma_i)]$.

Indeed, if $\sigma_j$ were a $L$-type in $\mathrm{Cl}(\mathcal U(\mathfrak h)W)[H^2(H,\sigma_i)]$, then,  $\sigma_j$ would be a $L$-type of a Discrete Series of lowest $L$-type equal to $\sigma_i$, according to  a) this forces $i=j$, a contradiction.
Now, we compute the multiplicity of $H^2(H,\sigma_j)$ in $H^2(G,\tau)$. For this, we apply Theorem~\ref{prop:equaldim}. Thus,  $\dim  Hom_H(V_\lambda^G, H^2(H,\sigma_j))=\sum_i q_i \dim  Hom_L(\sigma_j, H^2(H,\sigma_i))=q_j$.

In order to realize the isotypic component corresponding to $H^2(H,\sigma_j)$ we write $V_\lambda^G[W][\sigma_j]=R_1 +\cdots +R_{q_j}$ a explicit sum of $L$-irreducibles modules. Then, owing to a), $L_{\mathcal U(\mathfrak h)}(R_r)$ contains a copy $N_r$ of $H^2(H,\sigma_j)$ and  $R_r$ is the lowest $L$-type of $N_r$. Therefore, the multiplicity computation yields $H^2(G,\tau)[H^2(H,\sigma_j)]= N_1 +\cdots +N_{q_j} $. Hence, b) holds.
\noindent
A corollary of this computation is:

\phantom{xxxxxxxxxxxxx} $Hom_H(H^2(H,\sigma_j), (\mathrm{Cl}(\mathcal U(\mathfrak h)W))^\perp ) =\{0\}$.

Verification of c).  After we recall Lemma~\ref{lem:injecrh}, we have $r_0 :\mathrm{Cl}(\mathcal U(\mathfrak h_0)W)\rightarrow  L^2(H_0\times_\tau W)$ is injective and we apply  to the algebra $\mathfrak h:=\mathfrak h_0$, the statement b) together with the computation to show b), we make the choice of the $q_j's$ subspaces  $Z_j$ as a lowest $L$-type  subspace of $W[Z_j]$. Thus, the image via $r_0$ of $\mathcal U(\mathfrak h_0)Z_j$ is a subspace of $L^2(H_0 \times_{ \sigma_j} Z_j)$. Since, Hotta-Parthasarathy \cite{ho}, Atiyah-Schmid \cite{AS},  Enright-Wallach \cite{EW} have shown $H^2(H_0, \sigma_j)$ has multiplicity one in $L^2(H_0 \times_{ \sigma_j} Z_j)$  we obtain the image of $r_0$ is equal to $\mathbf H^2(H_0,\tau)$.

 The proof of d) and e) are word by word as the one for \ref{tauirred}. \end{proof}
\begin{cor} The multiplicity of $H^2(H,\sigma_j)$ in $res_H(H^2(G,\tau))$ is equal to \\ \phantom{xxxxxxxxxxxxxxxxxxxxxxxxxxxxxxxx} $q_j=\dim Hom_L(\sigma_j, H^2(G,\tau))$.
\end{cor}
\begin{cor}For each $\sigma_j$,   $\mathcal L_\lambda[Z_j]=\mathrm{Cl}(\mathcal U(\mathfrak h_0)W)[Z_j]=H^2(G,\tau)[W][Z_j]="W"[Z_j]$. Thus, we may fix $D=I_{"W"[Z_j]}  : \mathcal L_\lambda[Z_j] \rightarrow \mathrm{Cl}(\mathcal U(\mathfrak h_0)W)[Z_j]$.
\end{cor}

\subsection{ Explicit inverse map to $r_0^D$}\label{sub:inverserd}
We consider three cases: $res_L(\tau)$ is irreducible, $res_L(\tau)$ is multiplicity free,  and   general case. Formally, they are quite alike, however, for us it has been illuminating to consider the three cases. As a byproduct, we obtain information on the compositions $r^\star r, r_0^\star r_0$;  a functional equation that must be satisfied by the kernel of a holographic operator;  for some particular discrete factor $H^2(H,\sigma)$ of $res_H(\pi_\lambda)$ the reproducing kernel for $H^2(G,\tau)$ is a extension of the reproducing kernel for $H^2(H,\sigma)$  as well as that the holographic operator from  $H^2(H,\sigma)$ into  $H^2(G,\tau)$ is just plain extension of functions.
\subsubsection{ Case $(\tau, W)$ restricted to $L$    is irreducible} In  Tables 1,2,3,  we show   the list of  the triples $(G,H,\pi_\lambda)$ such that $(G,H)$ is a symmetric pair,   and $\pi_\lambda$ is $H$-admissible. In \ref{sub:existencep} we show that if there exists $(G,H,\pi_\lambda)$ so that $\pi_\lambda$ is $H$-admissible, then there exists a $H$-admissible $\pi_{\lambda'}$ so that its lowest $K$-type restricted to $L$ is irreducible and $\lambda'$ is dominant with respect to $\Psi_\lambda$. We denote by $\eta_0$ the Harish-Chandra parameter for $ H^2(H_0,\tau)\equiv \mathrm{Cl}(\mathcal U(\mathfrak h_0)W)  $.

We set  $d(\pi)$ for the formal degree of a irreducible square integrable representation $\pi$ and define $c=d(\pi_\lambda)\dim  W /d(\pi_{\eta_0}^{H_0})$.  Next,  we show
\begin{prop} \label{prop:ktfromkto}We assume the setting as well as  the hypothesis      in Theorem~\ref{prop:Rwelldefgc}, and further  $(\tau, W)$ restricted to $L$    is irreducible.  \\
  Let $T_0 \in Hom_L(Z,H^2(H_0,\tau))$, then the kernel $K_T$ corresponding to $T:=(r_0^D)^{-1}(T_0) \in Hom_H(H^2(H,\sigma), H^2(G,\tau))$ is  $$K_T(h,x)z=  (D^{-1}[\int_{H_0} \frac{1}{c} K_\lambda (h_0,\cdot)(T_0(z) (h_0)) dh_0 ])(h^{-1}x).$$
\end{prop}
\begin{proof} We systematically apply Theorem~\ref{prop:gentau}. Under our assumptions, we have:  $\mathbf H^2(H_0,\tau)$ is a irreducible representation and $\mathbf H^2(H_0,\tau)= H^2(H_0,\tau) $;    \\ $\mathrm{Cl}(\mathcal U(\mathfrak h_0) ( H^2(G,\tau)[W]))$ is $H_0$-irreducible;   We define\\     $ \tilde r_0:= rest(r_0) :\mathrm{Cl}(\mathcal U(\mathfrak h_0)H^2(G,\tau)[W]) \rightarrow H^2(H_0, \tau)$ is a isomorphism. To follow, we notice the inverse of $\tilde r_0$, is up to a constant, equal to  $r_0^\star$ restricted to $H^2(H_0, \tau)$. This is so,   because functional analysis yields the  equalities  $\mathrm{Cl}(Im(r_0^\star))=ker(r_0)^\perp=\mathrm{Cl}(\mathcal U(\mathfrak h_0)W)$, $Ker(r_0^\star)=Im(r_0)^\perp=H^2(H_0,\tau)^\perp$.   Thus, Schur's lemma  applied to the irreducible  modules $H^2(H_0,\tau), \mathrm{Cl}(\mathcal U(\mathfrak h_0)W)$ implies there exists  non zero constants $b,d$ so that $   (\tilde r_0 r_0^\star)_{\vert_{H^2(H_0,\tau)}}=b I_{H^2(H_0,\tau)}$, $r_0^\star \tilde r_0 =d I_{\mathrm{Cl}(\mathcal U(\mathfrak h_0)W)}$. Whence,  the inverse to $\tilde r_0$ follows. In \ref{sub:valuec}, we show $b=d =d(\pi_\lambda)\dim  W/d(\pi_{\eta_0}^{H_0})=c$.

For $x\in G, f\in H^2(G,\tau)$, the identity $f(x)=\int_G K_\lambda(y,x)f(y) dy$   holds. Thus,
$r_0(f)(p)=f(p)=\int_G K_\lambda (y,p) f(y) dy, \, \text{for}\, p\in H_0, f \in H^2(G,\tau)$, and,  we obtain

\medskip
$K_{r_0}(x,h_0)= K_\lambda (x,h_0), \,\,\, K_{r_0^\star}(h_0,x)=K_{r_0}(x,h_0)^\star =K_\lambda (h_0,x)$.

\smallskip
  Hence, for $g \in H^2(H_0, \tau)$  we have,  \\ \phantom{xxx} $\tilde r_0^{-1}(g)(x)=\frac{1}{c}\int_{H_0} K_{r_0^\star}(h_0,x) g(h_0) dh_0 = \frac{1}{c}\int_{H_0} K_{\lambda}(h_0,x) g(h_0) dh_0$.

Therefore, for   $T_0 \in Hom_L (Z, H^2(H_0,\tau))$,  the kernel $K_T$ of the element $T $ in $Hom_H(H^2(H,\sigma),H^2(G,\tau))$ such that $r_0^D(T)=T_0$,  satisfies for $z\in Z$ $$ D^{-1}([r_0^{-1}(T_0(z)(\cdot))])(\cdot)=K_T(e,\cdot)z \in V_\lambda^G[H^2(H,\sigma)][Z]\subset H^2(G,\tau ).$$ More explicitly, after we recall $K_T(e,h^{-1}x)=K_T(h,x)$,
$$K_T(h,x)z=  (D^{-1}[\int_{H_0}\frac{1}{c} K_\lambda (h_0,\cdot)(T_0(z)(h_0))dh_0 ])(h^{-1}x).$$ \end{proof}
\begin{cor}For any $T$ in   $Hom_H(H^2(H,\sigma),H^2(G,\tau))$ we have
$$K_T(h,x)z=  (D^{-1}[\int_{H_0} \frac{1}{c} K_\lambda (h_0,\cdot)(r_0(D(K_T(e,\cdot)z))(h_0))dh_0 ])(h^{-1}x).$$
\end{cor}
\begin{cor}When $D$ is the identity map, we obtain
\begin{eqnarray*}\lefteqn{K_T(h,x)z  =   \int_{H_0} \frac{1}{c}K_\lambda (h_0,h^{-1}x)(T_0(z)(h_0))dh_0} & & \\ & & \mbox{\phantom{ssssssssssss}} =\int_{H_0} \frac{1}{c}K_\lambda (hh_0,x)K_T(e,h_0)z dh_0.\end{eqnarray*} \end{cor}
The   equality in the conclusion of Proposition~\ref{prop:ktfromkto} is equivalent to  \\ \phantom{xxxxxx} $D(K_T(e,\cdot))(y)=\int_{H_0} \frac{1}{c} K_\lambda(h_0,y) D(K_T(e,\cdot))(h_0)dh_0, y\in G$. \\
Hence,     we have derived  a  formula that let us to recover the kernel $K_T$ (resp. $D(K_T(e,\cdot)) (\cdot)$) from $K_T(e,\cdot)$ (resp.  $D(K_T(e,\cdot)) (\cdot)$) restricted to $H_0$!
\begin{rmk}

  We notice,  \begin{equation} \label{eq:q}  r_0^\star r_0(f)(y) = \int_{H_0}   K_\lambda(h_0,y) f(h_0)dh_0  , \,\, f\in H^2(G,\tau),\,y\in G.  \end{equation}

  Since we are assuming $\tau_{\vert_L}$ is irreducible, we have  $\mathrm{Cl}(\mathcal U(\mathfrak h_0)W)$ is irreducible, hence, Lemma~\ref{lem:injecrh} let us  to obtain  that a scalar multiple of $r_0^\star r_0$ is the orthogonal projector onto the irreducible factor $\mathrm{Cl}(\mathcal U(\mathfrak h_0)W)$.

\noindent
Whence, the orthogonal projector onto $\mathrm{Cl}(\mathcal U(\mathfrak h_0)W)$
is given by    $ \frac{d(\pi_{\eta_0}^{H_0})}{d(\pi_\lambda) \dim  W}\,\,  r_0^\star r_0$.

 Thus, the kernel $K_{\lambda,\eta_0}$ of the orthogonal projector onto $\mathrm{Cl}(\mathcal U(\mathfrak h_0)W)$ is

\phantom{xxxxxx} $  K_{\lambda, \eta_0}(x,y):=   \frac{d(\pi_{\eta_0}^{H_0})}{d(\pi_\lambda) \dim  W} \int_{H_0}K_{\lambda}(p,y)K_{\lambda}(x,p) dp$.

Doing $H:=H_0$ we obtain a similar result for the kernel of the orthogonal projector  onto $\mathrm{Cl}(\mathcal U(\mathfrak h)W)$.
\end{rmk}

\smallskip

The equality $(r_0r_0^\star)_{\vert_{H^2(H_0,\tau)}} =cI_{H^2(H_0,\tau)}$ yields the first claim in:
\begin{prop}Assume $res_L(\tau) $ is irreducible. Then, \\
 a) for every  $g\in H^2(H_0, \tau_{\vert_L})$ (resp. $g\in H^2(H,\tau_{\vert_L})$), the function $r_0^\star(g)$ (resp. $r^\star(g)$),   is an extension of a scalar multiple of $g$.\\ b) The kernel  $K_\lambda^G$ is a extension of a scalar multiple of $K_{\tau_{\vert_L}}^H$.
 \end{prop}
When we restrict holomorphic Discrete Series,
this fact naturally happens, see \cite{Na}, \cite[ Example 10.1]{OV2}  and references therein.
\begin{proof} Let $r:H^2(G,\tau)\rightarrow L^2(H\times_\tau W)$ the restriction map. The duality $H,H_0$, and Theorem~\ref{prop:Rwelldefgc} applied to $H:=H_0$ implies $H^2(H,\tau)=r(\mathrm{Cl}(\mathcal U(\mathfrak h)W))$, as well as  that there exists,   up to a constant, a unique $T\in Hom_H (H^2(H,\tau),H^2(G,\tau))\equiv Hom_L(W,H^2(H_0,\tau))\equiv \mathbb C$. It follows  from the proof of  Proposition~\ref{prop:ktfromkto},  that, up to a constant,  $T=r^\star$ restricted to $H^2(H,\tau)$. After we apply the equality $ T(K_\mu^H(\cdot, e)^\star z)(x)= K_T(e,x)z$, (see \cite{OV1}),  we obtain,     \\ \phantom{xxxxxxxxx} $r^\star(K_\mu^H(\cdot,e)^\star z)(y)=K_\lambda(y,e)^\star z$.\\  Also, Schur's lemma implies $rr^\star $ restricted to $H^2(H,\tau)$ is a constant times the identity map. Thus, for $h\in H$,  we have  $rr^\star   (K_\mu^H(\cdot,e)^\star w)(h) =qK_\mu^H(h,e)^\star w$. For the value of $q$ see \ref{sub:valuec}. Putting together, we obtain,\\ \phantom{xxcccxxxx}  $K_\lambda(h,e)^\star z=r(K_\lambda(\cdot,e)^\star z)(h)=qK_\mu^H(h,e)^\star z$.

Whence, for $h,h_1 \in H$ we have
\begin{center}
$K_\lambda(h_1,h)^\star z=K_\lambda(h^{-1}h_1,e)^\star z=qK_\mu^H(h^{-1}h_1,e)^\star z=qK_\mu^H( h_1,h)^\star z$ \end{center} as we have claimed. \end{proof}

By the same token, after we set $H:=H_0$ we obtain:

\smallskip

  {\it For $res_L(\tau) $ irreducible, $(\sigma,Z)=(res_L(\tau),W)$, and $V_{\eta_0}^{H_0}= H^2(H_0,\sigma)$,   the kernel  $K_\lambda$ extends a scalar multiple of   $K_{\eta_0}^{H_0}$}. Actually, $r_0(K_\lambda(\cdot,e)^\star w )=cK_{\eta_0}^{H_0}(\cdot,e)^\star w$.

\medskip

\begin{rmk} We would like to point out that the equality \\ \phantom{xxxxxxxxxxxxx} $r^\star(K_\mu^H(\cdot,e)^\star(z))(y)= K_\lambda(y,e)^\star z$\\ implies $res_H(\pi_\lambda)$ is $H$-algebraically discretely decomposable. Indeed, we apply  a Theorem shown by Kobayashi \cite[Lemma 1.5]{Kob}, the Theorem says that when $(V_\lambda^G)_{K-fin}$ contains an irreducible $(\mathfrak h,L)$ irreducible submodule, then $V_\lambda^G$ is discretely decomposable. We know $K_\lambda(y,e)^\star z$ is a $K$-finite vector,  the equality implies $K_\lambda(y,e)^\star z$ is $\mathfrak z(\mathcal U(\mathfrak h))$-finite.  Hence, owing to Harish-Chandra \cite[Corollary 3.4.7 and Theorem 4.2.1]{Wa1},  $H^2(G,\tau)_{K-fin}$ contains a nontrivial irreducible $(\mathfrak h,L)$-module and the fact shown by  Kobayashi applies.
\end{rmk}

\subsubsection{Value of $b=d=c$ when $res_L(\tau)$ is irreducible}\label{sub:valuec} We show $b=d=d(\pi_\lambda)\dim  W/d(\pi_{\eta_0}^{H_0})=c $.  In fact, the constant $b,d$ satisfies  $(r_0^\star r_0)_{\mathcal U(\mathfrak h_0)W} =d I_{\mathcal U(\mathfrak h_0)W}$, $  (r_0 r_0^\star)_{\vert_{H^2(H_0,\tau)}}=b I_{H^2(H_0,\tau)}$. Now, it readily follows $b=d$.  To evaluate $r_0^\star r_0$ at $K_\lambda (\cdot,e)^\star w  $,    for $h_1 \in H_0$ we compute, for $h_1 \in H_0$, \\ $bK_\lambda (h_1,e)^\star w =r_0^\star r_0(K_\lambda (\cdot,e)^\star w) (h_1) $
$=  \int_{H_0} K_\lambda (h_0,h_1) K_\lambda (h_0,e)^\star   dh_0 w$ \\ \phantom{nnnnnnnnnnnnnnnnnnnnnnnnnnn} $=d(\pi_\lambda)^2 \int_{H_0}  \Phi(h_1^{-1}h_0 ) \Phi(h_0)^\star  dh_0 w$. \\ Here, $\Phi$ is the spherical function attached to the lowest $K$-type of $\pi_\lambda$. Since, we are assuming $res_L(\tau)$ is a  irreducible representation, we have $\mathcal U(\mathfrak h_0)W$ is a irreducible $(\mathfrak h_0, L)$-module and it is equivalent to the underlying Harish-Chandra module for $H^2(H_0,res_L(\tau))$. Thus, the restriction of  $\Phi$ to $H_0$ is the spherical function attached to the  lowest $L$-type of the irreducible square integrable representation $\mathrm{Cl}(\mathcal U(\mathfrak h_0)W)\equiv H^2(H_0,res_L(\tau)) $.   We fix a orthonormal basis $\{w_i\}$ for $\mathcal U(\mathfrak h_0)W[W]$.    We recall,

$\Phi(x)w=P_W \pi(x)P_W w=\sum_{1\leq i \leq \dim  W} (\pi(x)w,w_i)_{L^2} w_i$,

$\Phi(x^{-1})=\Phi(x)^\star$.

For $h_1 \in H_0$, we compute, to justify steps we appeal to  the invariance of Haar measure and to the orthogonality relations for matrix coefficients of irreducible square integrable representations and we recall $d(\pi_{\eta_0}^{H_0})$ denotes the formal degree for $H^2(H_0,res_L(\tau)) $.
\begin{equation*}
\begin{split}
 \int_{H_0} \Phi (h_1^{-1}h)\Phi(h)^\star w dh & = \sum_{i,j}\int_{H_0} (\pi(h_1^{-1}h)w_j ,w_i)_{L^2} (\pi(h^{-1})w,w_j)_{L^2} w_i dh \\ & = \sum_{i,j}\int_{H_0} (\pi(h)w_j ,\pi(h_1) w_i)_{L^2} \overline{(\pi(h)w_j,w)_{L^2}} w_i dh \\ & =1/d(\pi_{\eta_0}^{H_0}) \sum_{i,j} (w_j,w_j)_{L^2} \overline { (\pi(h_1)w_i,w)_{L^2}} w_i \\ & =\dim  W /d(\pi_{\eta_0}^{H_0})\sum_i (\pi(h_1^{-1}) w,w_i)_{L^2}w_i \\ & = \dim  W /d(\pi_{\eta_0}^{H_0}) \Phi(h_1)^\star w  . \end{split}  \end{equation*}

Thus, \begin{equation*}
\begin{split} r_0^\star r_0 (K_\lambda (\cdot, e)^\star w) (h_1) & = d(\pi_\lambda)^2 \int_{H_0} \Phi (h_1^{-1}h)\Phi(h)^\star w dh  \\ & = \frac{d(\pi_\lambda)^2 \dim  W }{d(\pi_{\eta_0}^{H_0})  d(\pi_\lambda)} K_\lambda (h_1,e)^\star w. \end{split}  \end{equation*}
The functions $ K_\lambda (\cdot ,e)^\star w, r_0^\star r_0 (K_\lambda (\cdot, e)^\star w) (\cdot)$ belong to $\mathrm{Cl}(\mathcal U(\mathfrak h_0)W)$, the injectivity of $r_0$ on $\mathrm{Cl}(\mathcal U(\mathfrak h_0)W)$, forces, for every $x \in G$

\phantom{xxxxxxx} $r_0^\star r_0 (K_\lambda (\cdot, e)^\star w) (x)= d(\pi_\lambda) \dim  W /d(\pi_{\eta_0}^{H_0})  K_\lambda (x,e)^\star w$.

 Hence, we have computed $b=d=c$.

\subsubsection{Analysis of $r_0^D$ for arbitrary $(\tau, W), (\sigma, Z)$ }
We recall the decomposition $W=\sum_{\nu_2^\prime \in Spec_{L\cap K_2}(\pi_{\Lambda_2}^{K_2}) }  W[ \pi_{\Lambda_1}^{ K_1}\boxtimes \pi_{\nu_2^\prime}^{L\cap K_2}]$.

A consequence of Proposition~\ref{prop:gentau} is that $r_0^\star$ maps $\mathbf H^2(H_0, W[\pi_{\Lambda_1}^{ K_1}\boxtimes \pi_{\nu_2}^{L\cap K_2}])$ into $\mathrm{Cl}(\mathcal U(\mathfrak h_0)W[\pi_{\Lambda_1}^{ K_1}\boxtimes \pi_{\nu_2}^{L\cap K_2}])$. In consequence, $r_0  r_0^\star$ restricted to $\mathbf H^2(H_0, W[\pi_{\Lambda_1}^{ K_1}\boxtimes \pi_{\nu_2}^{L\cap K_2}])$ is a bijective $H_0$-endomorphism $C_j$.  Hence, the inverse map of $r_0$ restricted to $\mathrm{Cl}(\mathcal U(\mathfrak h_0)W[\pi_{\Lambda_1}^{ K_1}\boxtimes \pi_{\nu_2}^{L\cap K_2}]) $ is $r_0^\star C_j^{-1}$. Since, $H^2(H_0, \pi_{\Lambda_1}^{ K_1}\boxtimes \pi_{\nu_2}^{L\cap K_2})$ has a unique lowest $L$-type, we conclude $C_j$ is determined  by an element of $Hom_L( \pi_{\Lambda_1}^{ K_1}\boxtimes \pi_{\nu_2}^{L\cap K_2},H^2(H_0,\pi_{\Lambda_1}^{ K_1}\boxtimes \pi_{\nu_2}^{L\cap K_2})[\pi_{\Lambda_1}^{ K_1}\boxtimes \pi_{\nu_2}^{L\cap K_2}])$. Since for $D \in \mathcal U(\mathfrak h_0), w \in W$ we have  \ $C_j(L_D w)=L_D C_j(w)$, we obtain $C_j$ is a zero order differential operator on the underlying Harish-Chandra module of $H^2(H_0,\pi_{\Lambda_1}^{ K_1}\boxtimes \pi_{\nu_2}^{L\cap K_2})$.   Summing up, we have that the inverse to $r_0 : \mathrm{Cl}(\mathcal U(\mathfrak h_0)W) \rightarrow \mathbf H^2(H_0, \tau)$ is the function $ r_0^\star (\oplus_j C_j^{-1})$. \\ For $T \in Hom_H(H^2(H,\sigma), H^2(G,\tau))$ and  $T_0 \in Hom_L(Z, \mathbf H^2(H_0,\tau))$ so that $r_0^D(T)=T_0$ we obtain the equalities\begin{equation*}K_T(e,x)z=  (D^{-1}[\int_{H_0} K_\lambda (h_0,\cdot)((\oplus_j C_j^{-1})T_0(z))(h_0)dh_0 ])(x). \end{equation*}
  \begin{equation*}  K_T(h,x)z      =  (D^{-1}[\int_{H_0} K_\lambda (h_0,\cdot)      (\oplus_j C_j^{-1})(r_0 (D (K_T(e,\cdot)z))(\cdot))(h_0)dh_0 ])(h^{-1}x).
 \end{equation*}
When $D$ is the identity the formula simplifies as the one in the second Corollary to Proposition~\ref{prop:ktfromkto}.

 \subsubsection{Eigenvalues of $r_0^\star r_0$}\label{sub:eigenvalq} For general case, we recall $r_0^\star r_0$  intertwines the action of $H_0$. Moreover, Proposition~\ref{prop:gentau}  and its Corollary gives that for each   $L$-isotypic component    $Z_1 \subseteq W$ of $res_L(\tau)$, we have $\mathcal U(\mathfrak h_0)W[Z_1]=Z_1$. Thus,  each isotypic component of $res_L((\mathcal U(\mathfrak h_0)W)[W])$ is invariant by  $r_0^\star r_0$, in consequence, $r_0^\star r_0$ leaves invariant the subspace
 $"W"=H^2(G,\tau)[W]=\{K_\lambda(\cdot,e)^\star w, w \in W\}$. Since, $Ker(r_0)=(\mathrm{Cl}(\mathcal U(\mathfrak h_0)W))^\perp$, we have  $r_0^\star r_0$ is determined by the values it takes on $"W"$.
 Now, we assume $res_L(\tau)$ is a multiplicity free representation, we write $Z_1^\perp=Z_2\oplus \dots \oplus Z_q$, where $Z_j$ are $L$-invariant and $L$-irreducible. Thus, Proposition~\ref{prop:gentau} implies $\mathrm{Cl}(\mathcal U(\mathfrak h_0)W)=   \mathrm{Cl}(\mathcal U(\mathfrak h_0)Z_1) \oplus \cdots \oplus  \mathrm{Cl}(\mathcal U(\mathfrak h_0)Z_q)$. This a orthogonal decomposition, each summand is irreducible and no irreducible factor is equivalent to   other. For $1\leq i\leq q$, let $\eta_i$ denote the Harish-Chandra parameter for $ \mathrm{Cl}(\mathcal U(\mathfrak h_0)Z_i)$.
  \begin{prop} \label{prop:eigenvalq}  When $res_L(\tau)$ is a multiplicity free representation, the linear operator $r_0^\star r_0$ on  $\mathrm{Cl}(\mathcal U(\mathfrak h_0)Z_i )$ is equal to $\frac{d(\pi_\lambda)\dim  Z_i}{d(\pi_{\eta_i}^{H_0}) }$ times  the identity map.
 \end{prop}

\begin{proof} For the subspace $\mathrm{Cl}(\mathcal U(\mathfrak h_0)W)[W]$,  we choose a $L^2(G)$-or\-tho\-nor\-mal basis $\{w_j\}_{1\leq j \leq \dim  W}$   equal to the union of respective $L^2(G)$-or\-tho\-nor\-mal basis for $\mathrm{Cl}(\mathcal U(\mathfrak h_0)Z_i)[Z_i]$. Next, we compute and freely make use of the notation  in \ref{sub:valuec}. Owing to our multiplicity free hypothesis, we have that $r_0^\star r_0$ restricted to $ \mathrm{Cl}(\mathcal U(\mathfrak h_0)Z_i)$ is equal to a constant $d_i$ times the identity map. Hence, on $ w\in   \mathrm{Cl}(\mathcal U(\mathfrak h_0)Z_i)[Z_i]$ we have $d_i w=d(\pi_\lambda)^2 \int_{H_0} \Phi(h_0) \Phi(h_0)^\star w dh_0$.

Now, $\Phi(h_0)=(a_{i j})=((\pi_\lambda(h_0)w_j, w_i)_{L^2(G)})$, Hence, the  $p q$-coefficient  of the product $\Phi(h_0)\Phi(h_0)^\star $ is equal to

 $\sum_{1\leq j \leq \dim  W} (\pi_\lambda(h_0)w_j, w_p)_{L^2(G)}\overline{(\pi_\lambda(h_0)w_j, w_q)}_{L^2(G)}$.

Let $I_i$ denote the set of indexes $j$ so that $w_j \in Z_i$. Thus,  $\{1,\dots,\dim  W\}$ is equal to the disjoint union $\cup_{1\leq i \leq q }  I_i$. A consequence of Proposition~\ref{prop:gentau} is the $L^2(G)$-orthogonality of the subspaces $\mathrm{Cl}(\mathcal U(\mathfrak h_0)Z_j)$,  hence,  for $t \in I_a, q \in I_d$ and $a\not= d$  we have  $(\pi_\lambda(h_0)w_q, w_t)_{L^2(G)} =0$. Therefore,   the previous observation and the disjointness of the sets $I_r$, let us obtain  that  for  $i\not= d, p  \in I_i, q \in I_d $ each summand in
\begin{center}
 $\sum_{1\leq j \leq \dim  W} \int_{H_0} (\pi_\lambda(h_0)w_j, w_p)_{L^2(G)}\overline{(\pi_\lambda(h_0)w_j, w_q)}_{L^2(G)}
dh_0  $ \end{center}
is equal to zero.

   For $p,q \in I_i$, we apply the previous computation and the orthogonality relations   to the irreducible representation  $\mathrm{Cl}(\mathcal U(\mathfrak h_0)Z_i)$.   We obtain

$\sum_{1\leq j \leq \dim  W} \int_{H_0} (\pi_\lambda(h_0)w_j, w_p)_{L^2(G)}\overline{(\pi_\lambda(h_0)w_j, w_q)}_{L^2(G)}
dh_0$

$ =  \sum_{j \in I_i} \int_{H_0} (\pi_\lambda(h_0)w_j, w_p)_{L^2(G)}\overline{(\pi_\lambda(h_0)w_j, w_q)}_{L^2(G)}
dh_0$

$ =\sum_{j\in I_i} \frac{ 1}{d(\pi_{\eta_i}^{H_0})} (w_j,w_j)_{L^2(G)} (w_q,w_p)_{L^2(G)}= \frac{ \dim  Z_i }{d(\pi_{\eta_i}^{H_0})} \delta_{pq}$.

Thus, we have shown Proposition~\ref{prop:eigenvalq}. \end{proof}
 \begin{rmk} Even, when $res_L(\tau)$ is not multiplicity free,  the conclusion in Proposition~\ref{prop:eigenvalq} holds. In fact,  let us denote the $L$-isotypic component of $res_L(\tau)$ again by $Z_i$. Now,   the proof goes as the one for Proposition~\ref{prop:eigenvalq} till we need to compute

  $ =  \sum_{j \in I_i} \int_{H_0} (\pi_\lambda(h_0)w_j, w_p)_{L^2(G)}\overline{(\pi_\lambda(h_0)w_j, w_q)}_{L^2(G)}
dh_0$.

For this, we decompose $"Z_i"=\sum_s Z_{i, s}$ as a $L^2(G)$-orthogonal sum of irreducible $L$-modules and we choose the orthonormal basis for $"Z_i"$ as a union of orthonormal basis for each $Z_{i,s}$. Then, we have the $L^2(G)$-orthogonal decomposition  $\mathrm{Cl}(\mathcal U(\mathfrak h_0)Z_i)=\sum_s \mathrm{Cl}(\mathcal U(\mathfrak h_0)Z_{i,s})$.  Then,  the proof follows as in the case $res_L(\tau)$ is multiplicity free.
 \end{rmk}

\section{Examples} We present three type of examples. The first is: {\it Multiplicity free representations.} A simple consequence of the duality theorem is that it readily follows  examples of  symmetric pair $(G,H)$  and square integrable representation $\pi_\lambda^G$ so that $res_H(\pi_\lambda)$ is $H$-admissible and the multiplicity of each  irreducible factor is one. This is equivalent to determinate when the representation $res_L(\mathbf H^2(H_0,\tau))$ is multiplicity free.
The second is: {\it  Explicit examples}. Here, we compute the Harish-Chandra parameters of the irreducible factors for some $res_H(H^2(G,\tau))$. The third is: {\it Existence of representations so that its lowest $K$-types restricted to $L$  is a irreducible representation.}

 In order to present the examples we need information on certain families of representations.

\subsection{Multiplicity free representations} In this paragraph  we generalize  work of T.  Kobayashi and his coworkers in  the setting of Hermitian symmetric spaces and holomorphic Discrete Series.

\smallskip

Before  we present the examples, we would like to   comment.

\smallskip

a) Assume a Discrete Series $\pi_\lambda$ has an admissible restriction to a subgroup $H$. Then, any Discrete Series $\pi_{\lambda^\prime}$ for $\lambda^\prime$ dominant with respect to $\Psi_\lambda$ is $H$-admissible \cite{Kob1}.

\smallskip

b) If $res_H(\pi_\lambda)$ is $H$-admissible and a multiplicity free representation. Then the restriction to $L$ of the lowest $K$-type for $\pi_\lambda$  is multiplicity free.  This follows from the duality theorem.

\smallskip

c) In the next paragraphs we will list families $\mathcal F$ of Harish-Chandra parameters   of Discrete Series   for $G$ so that each representation in the family has  a multiplicity free restriction to $H$. We find that it may happen that  $\mathcal F$ is the whole set of Harish-Chandra parameters on a Weyl chamber or $\mathcal F$ is a proper subset of a Weyl Chamber. Information on $\mathcal F$ for holomorphic reprentations  is  in  \cite{KO}, \cite{KO2}.

\smallskip

d) Every irreducible $(\mathfrak g, K)$-module for either $\mathfrak g\equiv \mathfrak{su}(n,1)$ or $\mathfrak g\equiv \mathfrak{so}(n,1)$, restricted to $K$, is a multiplicity free representation.

\subsubsection{Holomorphic representations}  For $G$ so that  $G/K$ is a Hermitian symmetric space, it has been shown by Harish-Chandra that  $G$ admits Discrete Series representations with one dimensional lowest $K$-type. For this paragraph we further  assume that  the smooth imbedding $H/L \rightarrow G/K$ is holomorphic, equivalently the center of $K$ is contained in $L$,  and $\pi_\lambda$ is a holomorphic representation. Under this hypothesis,  it was shown  by Kobayashi \cite{Kob3}  that a holomorphic Discrete Series for $G$ has a multiplicity free restriction to the subgroup $H$ whenever   it is a scalar holomorphic Discrete Series. Moreover, in \cite[Theorem  8.8]{ Kob3}  computes the Harish-Chandra parameter of each irreducible factor.   Also, from the work of Kobayashi and Nakahama we find a description of the restriction to $H$ of a arbitrary holomorphic Discrete Series representations. As a consequence, we find restrictions which are not multiplicity free.

    In \cite{KO2} we find a complete list of the pairs $(\mathfrak g, \mathfrak h)$ so that $H/L \rightarrow G/K$ is a holomorphic embedding. From the  list   in   \cite{Kob3}, it can be constructed the list bellow.

\smallskip

 Also, Theorem~\ref{prop:Rwelldefgc} let us verify that the following pairs $(\mathfrak g,\mathfrak h)$ are so that $res_H(\pi_\lambda)$ is multiplicity free for any  holomorphic $\pi_\lambda$. For this, we  list the associated $\mathfrak h_0$.

\noindent
 $(\mathfrak{su}(m,n), \mathfrak{s}(\mathfrak{u}(m-1,n)+ \mathfrak u(1)))  $,   $\mathfrak h_0=  \mathfrak{su}(1,n)+\mathfrak{su}(m-1)+ \mathfrak u(1)   $.

\noindent
$(\mathfrak{su}(m,n),  \mathfrak{s}( \mathfrak{u}(m,n-1)+ \mathfrak u(1)))  $,   $\mathfrak h_0=  \mathfrak{su}(n-1)+\mathfrak{su}(m,1)+ \mathfrak u(1)   $.

The list is correct, owing to any Discrete Series for $SU(n,1)$ restricted to $K$ is a multiplicity free representation.

\

\subsubsection{Quaternionic real forms,  quaternionic representations} In \cite{GW}, the authors considered and classified quaternionic real forms, and also,   they made a careful study of  quaternionic representations. To follow we bring out the essential facts for us. From \cite{GW} we read that the list of Lie algebra of quaternionic groups is: $\mathfrak{su}(2,n)$, $\mathfrak{so}(4,n)$, $\mathfrak{sp}(1,n)$,   $\mathfrak e_{6(2)}$, $\mathfrak e_{7(-5)}$, $\mathfrak e_{8(-24)}$, $\mathfrak f_{4(4)}$, $\mathfrak g_{2(2)}$.     For each  quaternionic real form $G$,   there exists  a system of positive roots $\Psi \subset \Phi (\mathfrak g,\mathfrak t)$ so that the maximal root $\alpha_{max}$ in $\Psi$ is compact, $\alpha_{max}$ is orthogonal to all compact simple roots and $\alpha_{max}$ is not orthogonal to each noncompact simple roots. Hence, $\mathfrak k_1(\Psi)\equiv \mathfrak{su}_2(\alpha_{max}) $. The system $\Psi$ is not unique. We appel such a system of positive roots  {\it a quaternionic system}.

 Let us recall that a {\it quaternionic representation} is a Discrete Series for a quaternionic real form  $G$ so that its Harish-Chandra parameter is dominant with respect to a quaternionic system of positive roots,  and so that its lowest $K$-type is equivalent to a irreducible representation for $ K_1(\Psi)$ times the trivial representation for $K_2$.
 A  fact shown in \cite{GW} is: Given a quaternionic system of positive roots,  for all but finitely many representations  $(\tau, W)$ equivalent to the tensor  of  a nontrivial representation for $K_1(\Psi)$ times the trivial representation of $K_2$, it  holds:  $\tau$ is the lowest $K$-type of a quaternionic (unique) irreducible square integrable representation $H^2(G,\tau)$. We define a {\it generalized quaternionic representation} to be a Discrete Series representation $\pi_\lambda$ so that its Harish-Chandra parameter is dominant with respect to a quaternionic system of positive roots.

From Table 1,2 we readily read the pairs $(\mathfrak g,\mathfrak h)$ so that $\mathfrak g$ is a quaternionic Lie algebra and hence, we have a list of  generalized quaternionic representations of $G$ with admissible restriction to $H$.

Let $(G,H)$ denote a symmetric pair so that a quaternionic representation $(\pi_\lambda, H^2(G,\tau))$ is $H$-admissible. Then, from \cite{Vaint} \cite{DV} \cite{DGV} we have:   $ \mathfrak k_1(\Psi_\lambda) \equiv \mathfrak{su}_2(\alpha_{max}) \subset \mathfrak l $ and  $\pi_\lambda$ is $L$-admissible.  In consequence, \cite{Kob}, $\pi_\lambda$ is $H_0$-admissible. By definition, for a quaternionic representation $\pi_\lambda$, we have $\tau_{\vert_L}$ is irreducible,  hence,   $\mathbf{H}^2(H_0,\tau)$ is irreducible. Moreover, after checking on \cite{Vaint} or Tables 1,2,  the list of systems $\Psi_{H_0,\lambda}$, it follows     that $H^2(H_0, \tau)$ is again a quaternionic representation. Finally, in order to present a list of quaternionic representations with multiplicity free restriction to $H$ we recall that  it follows   from the duality Theorem  that $res_H(H^2(G,\tau))$ is multiplicity free if and only if $res_L(H^2(H_0,\tau))$ is a multiplicity free representation, and that    on \cite[Page 88]{GW} it is shown that  a quaternionic representation   for $H_0$ is $L$-multiplicity free if and only if  $\mathfrak h_0 =\mathfrak{sp}(n,1), n\geq 1 $.

 \smallskip

 To follow, we list  pairs $(\mathfrak g,\mathfrak h)$ where multiplicity free restriction holds for all quaternionic representations.

$(\mathfrak{su}(2,2n), \mathfrak{sp}(1,n))$,   $\mathfrak h_0=\mathfrak{sp}(1,n)$, $n\geq 1$.

$(\mathfrak{so}(4,n),\mathfrak{so}(4,n-1))$, $\mathfrak h_0 =\mathfrak{so}(4,1) +\mathfrak{so}(n-1)$ ($n$ even or  odd).

$(\mathfrak{sp}(1, n),\mathfrak{sp}(1,k)+\mathfrak{sp}(n-k) ) $, $\mathfrak h_0=\mathfrak{sp}(1,n-k)+\mathfrak{sp}(k)$.

$(\mathfrak{f}_{4(4)},  \mathfrak{so}(5,4))$,   $\mathfrak h_0 =\mathfrak{sp}(1,2)\oplus \mathfrak{su}(2)$.

$(\mathfrak e_{6(2)}, \mathfrak f_{4(4)})$, $\mathfrak h_0= \mathfrak{ sp}(3,1)$.

\smallskip

 A special pair is:

$(\mathfrak{su}(2,2), \mathfrak{sp}(1,1)),  $ $\mathfrak h_0=\mathfrak{sp}(1,1)$.

 Here, multiplicity free holds for any $\pi_\lambda$ so that $\lambda$ is dominant with respect to a system of positive roots that defines a quaternionic structure on $G/K$. For details see \cite[Table 2]{Vaint} or {Explicit example II}.

 \medskip

\subsubsection{More examples of multiplicity free restriction}

 Next, we list pairs $(\mathfrak g,\mathfrak h)$ and systems of positive roots  $\Psi \subset \Phi(\mathfrak g,\mathfrak t)$ so that    $\pi_{\lambda^\prime}$ is $H$- admissible and multiplicity free for every element  $ \lambda^\prime$ dominant with respect to  $\Psi$.    We follow  either  Table 1,2,3 or \cite{KO2}.   For each $(\mathfrak g, \mathfrak h)$ we list the corresponding  $  \mathfrak h_0 $.

\medskip

$(\mathfrak{su}(m,n),\mathfrak{su}(m,n-1)+\mathfrak u(1))$, $\Psi_a, \tilde\Psi_b $, \\ \phantom{xxxxxxxxxxxxxxxxxxxx} $\mathfrak h_0 =\mathfrak{su}(m,1) +\mathfrak{su}(n-1)+\mathfrak u(1)$.

$(\mathfrak{so}(2m,2n+1),\mathfrak{so}(2m,2n))$,   $\Psi_{\pm}$,  $\mathfrak h_0= \mathfrak{so}(2m,1)+\mathfrak{so}(2n)$.

$(\mathfrak{so}(2m,2),\mathfrak{so}(2m,1))$,  $\Psi_{\pm}$, $\mathfrak h_0= \mathfrak{so}(2m,1)$.

$(\mathfrak{so}(2m,2n),\mathfrak{so}(2m,2n-1)), n>1$, $\Psi_{\pm}$, $\mathfrak h_0 =\mathfrak{so}(2m,1) +\mathfrak{so}(2n-1)$.

\subsection{Explicit examples}
\subsubsection{Quaternionic representations for $Sp(1,b)$}\label{sub:Ltypesspd1} For further use we present a intrinsic description for the $Sp(1)\times Sp(b)$-types of a quaternionic representation for $Sp(1,b)$, a proof of the statements is in \cite{GW0}. The quaternionic representations for $Sp(1,b)$ are the representations of lowest  $Sp(1)\times Sp(b)$-type $S^n(\mathbb C^{2})\boxtimes \mathbb C, n\geq 1$.  We label the simple roots for the quaternionic system of positive roots  $\Psi$ as in \cite{GW},   $\beta_1, \dots, \beta_{b+1}$, the long root is
$\beta_{b+1}$, $\beta_1$ is adjacent to just one simple root and  the maximal root $\beta_{max}$ is adjacent to $-\beta_1$.  Let $\Lambda_1, \dots, \Lambda_{d+1}$ the associated fundamental weights. Thus, $\Lambda_1=\frac{\beta_{max}}{2}$.
Let $\tilde{\Lambda}_1, \dots, \tilde{\Lambda}_b$ denote the fundamental weights for $"\Psi \cap \Phi(\mathfrak{sp}(b))"$.
The irreducible $L=Sp(1)\times Sp(b)$-factors of

  \phantom{xxx} $ H^2(Sp(1,b), \pi_{n\frac{\beta_{max}}{2}}^{Sp(1)}\boxtimes \pi_{\rho_{Sp(b)}}^{Sp(b)})= H^2(Sp(1,b), S^{n-1}(\mathbb C^2) \boxtimes \mathbb C)$

\noindent
are

 \hskip 0.7cm $\{S^{n-1+m}(\mathbb C^2)\boxtimes S^m(\mathbb C^{2b} )  \equiv \pi_{(n+m)\frac{\beta_{max}}{2}}^{Sp(1)}\boxtimes \pi_{m\tilde{\Lambda}_1 +\rho_{Sp(b)}}^{Sp(b)}  ,  \, m\geq 0 \}$.

The multiplicity of each $L$-type in $ H^2(Sp(1,b), S^{n-1}(\mathbb C^2) \boxtimes \mathbb C)$ is one.

 \subsubsection{Explicit example I} \label{sub:expexam} We develop this example in detail. We restrict quaternionic representations for $Sp(1,d)$ to $Sp(1,k)\times Sp(d-k)$. For this,  we need to review definitions and facts in \cite{GW0}\cite{KO2} \cite{Vaint}. The group $G:=Sp(1,d)$ is a subgroup of $GL(\mathbb C^{2+2d})$. A maximal compact subgroup of $Sp(1,d)$ is the usual immersion of $Sp(1)\times Sp(d)$.   Actually, $Sp(1,d)$ is  a quaternionic real form for $Sp(\mathbb C^{1+d})$.  $Sp(1,d)$ has a compact Cartan subgroup $T$ and there exists a  orthogonal basis $\delta, \epsilon_1, \dots, \epsilon_d$ for $i\mathfrak t^\star$ so that

 $\Phi(\mathfrak{sp}(d+1,\mathbb C), \mathfrak t)=\{ \pm 2\delta, \pm 2\epsilon_1, \dots, \pm 2\epsilon_d, \pm (\epsilon_i \pm \epsilon_j), 1\leq i < j \leq d, \pm (\delta \pm \epsilon_s), 1\leq s \leq d \}$.

  We fix $1 \leq k <d$. We consider  the usual immersion of  $H:=Sp(1,k)\times Sp(d-k)$ into $Sp(1,d)$. Thus,

  \noindent
  $\Phi(\mathfrak h,\mathfrak t):= \{ \pm 2\delta, \pm 2\epsilon_1, \dots, \pm 2\epsilon_d, \pm (\epsilon_i \pm \epsilon_j), 1\leq i < j \leq k, \,\\    \phantom{xxxxxxxxxxxxxxxxxxxxxxxxxxxxxxc} \text{or},\,   k+1\leq i < j \leq d, \pm (\delta \pm \epsilon_s), 1\leq s \leq k \}$.

  Then,  $H_0$ is isomorphic to $ Sp(1,d-k) \times Sp(k) $. We have

  $\Phi(\mathfrak h_0,\mathfrak t):= \{ \pm 2\delta, \pm 2\epsilon_{1}, \dots, \pm 2\epsilon_d, \pm (\epsilon_i \pm \epsilon_j), k+1\leq i < j \leq d, \, \\    \phantom{xxxxxxxxxxxxxxxxxxxxxxxxxxxxxxc} \text{or},\,  1\leq i < j \leq k, \pm (\delta \pm \epsilon_s), k+1\leq s \leq d \}$.

From now on, we fix the quaternionic system of positive roots\\
\indent
$\Psi:= \{  2\delta,   2\epsilon_1, \dots,   2\epsilon_d,  (\epsilon_i \pm \epsilon_j), 1\leq i < j \leq d,   (\delta \pm \epsilon_s), 1\leq s \leq d \}$. \\ Then,   $\alpha_{max}=2\delta$, $\rho_n^\Psi = d \delta $.   The Harish-Chandra parameter $\lambda$ of a quaternionic representation $\pi_\lambda$ is dominant with respect to $\Psi$. Hence, $ \Psi_\lambda=\Psi$. The systems in Theorem~\ref{prop:GWKO} are  $\Psi_{H,\lambda}=\Phi(\mathfrak h,\mathfrak t)\cap \Psi$, $\Psi_{H_0,\lambda}=\Phi(\mathfrak h_0,\mathfrak t)\cap \Psi$. Also, \cite{DV}, $\Phi (\mathfrak k_1:=\mathfrak k_1(\Psi), \mathfrak t_1:=\mathfrak t \cap \mathfrak k_1)=\{ \pm 2\delta \}$, $\Phi (\mathfrak k_2:=\mathfrak k_2(\Psi), \mathfrak t_2:=\mathfrak t \cap \mathfrak k_2)=\{ \pm 2\epsilon_1, \dots, \pm 2\epsilon_d, \pm (\epsilon_i \pm \epsilon_j), 1\leq i < j \leq d\}$.   Thus, $K_1(\Psi)\equiv SU_2(2\delta)\equiv Sp(1)\subset H$, $K_2\equiv Sp(d)$. Hence, for a Harish-Chandra parameter $\lambda=(\lambda_1, \lambda_2), \lambda_j \in i\mathfrak t_j^\star $ dominant with respect to $\Psi$, the representation $\pi_\lambda$ is $H$-admissible.

 The lowest $K$-type of a generalized quaternionic representation $\pi_\lambda$ is the representation $\tau= \pi_{\lambda +\rho_n^\lambda}^K= \pi_{\lambda_1 +d\delta}^{K_1}\boxtimes \pi_{\lambda_2}^{K_2}$. Since, $\rho_{K_2}= d\epsilon_1 +(d-1)\epsilon_2 +\dots +\epsilon_d $,  for $n\geq d+1$, the functional $\mathfrak t^\star \ni \lambda_n:= n \delta +\rho_{K_2} $ is a Harish-Chandra parameter dominant with respect to $\Psi$   and the lowest $K$-type $\tau_n$ of $\pi_{\lambda_n}$ is
$\pi_{(n+d)\delta}^{K_1}\boxtimes \pi_{\rho_{K_2}}^{K_2}$. That is, $\pi_{ \lambda_n +\rho_n^\lambda}^K$ is equal to a irreducible representation of $K_1\equiv Sp(1)=SU_2(2\delta)$ times the trivial representation of $K_2\equiv Sp(d)$. The family  $(\pi_{\lambda_n})_n$ exhausts, up to equivalence,  the set of quaternionic representations for $Sp(1,d)$. Now, $\mathbf{H}^2(H_0,\tau_n)$ is the irreducible representation of lowest $L$-type   equal to the irreducible representation $\pi_{(n+d)\delta}^{K_1}$ of $K_1$ times the trivial representation of $K_2 \cap L$.    Actually, $\mathbf{H}^2(H_0, \pi_{(n+d)\delta}^{K_1}\boxtimes \pi_{\rho_{K_2}}^{K_2})$ is a realization of the  quaternionic  representation  $H^2(Sp(1,n-k),\pi_{(n+d)\delta}^{Sp(1)}\boxtimes \pi_{\rho_{Sp(n-k)}}^{Sp(n-k)})$ for $Sp(1,d-k)$ times the trivial representation of $Sp(k)$.  In \cite[Proposition 6.3]{GW0} it is shown that  the representation $H^2(Sp(1,n-k),\pi_{(n+d)\delta}^{Sp(1)}\boxtimes \pi_{\rho_{Sp(n-k)}}^{Sp(n-k)})$  restricted to $L$ is a multiplicity free representation, and also, they list the highest weights of the totality of  $L$-irreducible factors. To follow we explicit such a computation.
For this we recall \ref{sub:Ltypesspd1} and notice $b=d-k$; $\Lambda_1=\delta$, $\beta_{max}=2\delta$,  $\tilde{\Lambda}_1=\epsilon_1$;  as $Sp(1)$-module,  $S^p(\mathbb C^2)\equiv \pi_{(p+1) \delta}^{SU_2(2\delta)}$;  for $p\geq 1$,  as $Sp(p)$-module $S^m(\mathbb C^{2p})\equiv \pi_{ m\epsilon_1+\rho_{Sp(p)}}^{Sp(p)}$. Then,

\noindent
 the irreducible $L=Sp(1)\times Sp(d-k)\times Sp(k)$-factors of

\bigskip
   $ \mathbf H^2(H_0, \pi_{(n+d)\delta }^{K_1}\boxtimes \pi_{\rho_{K_2}}^{K_2})  \equiv H^2(Sp(1, d-k), \pi_{(n+d)\delta }^{Sp(1)}\boxtimes \pi_{\rho_{Sp(d-k)}}^{Sp(d-k)})\boxtimes \mathbb C$

   \bigskip
   \noindent
    are multiplicity free and   it is the set of inequivalent representations

\bigskip
 $\{S^{n+d-1+m}(\mathbb C^2)\boxtimes S^m(\mathbb C^{2(d-k)} )\boxtimes \mathbb C \\ \phantom{xxxxxxxxxxxxxxxxx} \equiv \pi_{(n+d+m)\delta}^{Sp(1)}\boxtimes \pi_{m\epsilon_{k+1} +\rho_{Sp(d-k)}}^{Sp(d-k)} \boxtimes \pi_{\rho_{Sp(k)}}^{Sp(k)},  \, m\geq 0 \}$.

 Here, $\rho_{Sp(d-k)}=(d-k)\epsilon_{k+1}+ (d-k-1)\epsilon_{k+2}+\dots +\epsilon_d   $ and $\rho_{Sp(k)}=k\epsilon_1 + (k-1)\epsilon_2 +\dots +\epsilon_k$.

 We compute $\Psi_{H,\lambda}= \{  2\delta,   2\epsilon_1, \dots,   2\epsilon_d,   (\epsilon_i \pm \epsilon_j), 1\leq i < j \leq k \,\text{or}\,   k+1\leq i < j \leq d,  (\delta \pm \epsilon_s), 1\leq s \leq k \}$.
$\rho_n^\mu=\rho_n^H=k\delta$.
  Now, from  Theorem~\ref{prop:Rwelldefgc}  we have $Spec_H(\pi_\lambda) +\rho_n^H =Spec_L(\mathbf H^2(H_0,\tau))$, whence,  we conclude:

\smallskip

 The representation $res_{Sp(1,k)\times Sp(d-k)}(\pi_{\lambda_n}^{Sp(1,d)})$ is a multiplicity free representation and the totality of Harish-Chandra parameters of the $Sp(1,k)\times Sp(d-k)$-irreducible factors is the set
\bigskip

 $\{ (n+d+m)\delta + m \epsilon_{k+1}+\rho_{Sp(k)}+\rho_{Sp(d-k)} \}-\rho_n^H = \\  \phantom{xxx}(n+d+m-k)\delta + m \epsilon_{k+1}    + (d-k)\epsilon_{k+1}+\dots +\epsilon_d + k\epsilon_1 +\dots +\epsilon_k , m\geq 0 \}$.

\bigskip
Whence, $res_{Sp(1,k)\times Sp(d-k)}(\pi_{\lambda_n}^{Sp(1,d)})$ is equivalent  to  the Hilbert sum

\bigskip

 $\oplus_{m\geq 0} V_{(n+d+m-k) \delta  + m \epsilon_{k+1}+\rho_{Sp(k)}+\rho_{Sp(d-k)}}^{Sp(1,k)\times Sp(d-k)}\\ \phantom{xxx}\equiv \oplus_{m\geq 0}H^2(Sp(1,k)\times Sp(d-k), \pi_{(n+d+m)\delta + m \epsilon_{k+1}+\rho_{Sp(k)}+\rho_{Sp(d-k)}}^{Sp(1)\times Sp(k)\times Sp(d-k)}) $.

\medskip
A awkward point of our decomposition is that does not provide an explicit  description of the $H$-isotypic components for $res_H(V_\lambda^G)$.

 \smallskip

 \subsubsection{Explicit example II} We restrict from  $Spin(2m,2), m \geq 2$, to $Spin(2m,1)$. We notice the   isomorphism between $(Spin(4,2), Spin(4,1))$  and  the pair $(SU(2,2),Sp(1,1))$. In this setting   $K=Spin(2m)\times Z_K$, $L=Spin(2m)$, $ Z_K\equiv \mathbb T$.  Obviously, we may conclude that any irreducible representation of $K$ is irreducible when restricted to $L$. In this case $H_0 \equiv Spin(2m,1)$, and (for $m=2$,  $H_0 \equiv Sp(1,1)$) and $\mathbf H^2(H_0,\tau)$ is irreducible. Therefore, the duality theorem together with that  any irreducible representation for $    Spin(2m,1) $ is $L=Spin(2m)$-multiplicity free   \cite[page 11]{Th}, we obtain:

 \smallskip
  {\it Any $    Spin(2m,1) $-admissible representation $ \pi_\lambda^{Spin(2m,2)} $ is multiplicity free.}
  \smallskip

For $(Spin(2m,2), Spin(2m,1))$    in \cite[Table 2 ]{Vaint}, \cite{KO2} it is verified that any $\pi_\lambda$, with $\lambda$ dominant with respect to one of the systems $\Psi_{\pm}$ (see proof of \ref{lem:equalm})  has  admissible restriction to $Spin(2m,1)$ and no other $\pi_\lambda$ has admissible restriction to $Spin(2m,1)$.

  \smallskip
  In \cite[Table 2 ]{Vaint} \cite{Kob1} \cite{Kob} it is verified that any  square integrable representation $\pi_\lambda$ with $\lambda$ dominant with respect to a quaternionic system  for $SU(2,2)$, has  admissible restriction to $Sp(1,1)$.  As in \ref{sub:expexam}, we may compute the Harish-Chandra parameters for the irreducible components of $res_{Sp(1,1)}(\pi_\lambda^{SU(2,2)})$.

\subsubsection{Explicit example III}
 \,\,  $ (\mathfrak{e}_{6(2)}, \mathfrak f_{4(4)}). $ We fix a compact Cartan subgroup $T\subset K$ so that $U:=T\cap H$ is a compact Cartan subgroup of $L=K\cap H$. Then, there exist a quaternionic and Borel de Siebenthal positive root system $\Psi_{BS}$ for $\Phi(\mathfrak e_6, \mathfrak t)$ so that,   after   we write the simple roots   as in Bourbaki (see \cite{GW0}\cite{Vaint}),    the compact simple roots are $\alpha_1,  \alpha_3, \alpha_4, \alpha_5, \alpha_6$ (They determinate the $A_5$-Dynkin sub-diagram)  and $\alpha_2$ is   noncompact. $\alpha_2$ is adjacent to $-\alpha_{max}$ and to $\alpha_4$. In  \cite{Vaint}, it is verified $\Psi_{BS}$ is the unique  system of positive roots   such that $\mathfrak k_1( \Psi_{BS}) =\mathfrak{su}_2(\alpha_{max}).$

     The automorphism  $\sigma$  of $\mathfrak g$  acts on the simple roots as follows $$\sigma (\alpha_2)=\alpha_2, \,\,  \sigma (\alpha_1)=\alpha_6, \,\,  \sigma( \alpha_3 )=\alpha_5, \,\, \sigma (\alpha_4)=\alpha_4.$$ Hence, $\sigma(\Psi_{BS})=\Psi_{BS}.$  Let $h_2 \in i\mathfrak t^\star $ be so that $\alpha_j(h_2)=\delta_{j 2} $ for $j=1,\dots, 6.$  Then,   $h_2=\frac{2H_{\alpha_{max}}}{(\alpha_{max}, \alpha_{max})}$ and $\theta= Ad(exp(\pi i h_2)).$   A straightforward computation yields: $\mathfrak k \equiv  \mathfrak{su}_2(\alpha_{max}) + \mathfrak{su}(6)$, $\mathfrak l\equiv \mathfrak{su}_2(\alpha_{max}) + \mathfrak{sp}(3) $;  the fix point subalgebra for $\theta \sigma$ is isomorphic to $\mathfrak{sp}(1,3)$. Thus, the pair $(\mathfrak e_{6(2)}, \mathfrak{sp}(1,3))$ is the associated pair to  $ (\mathfrak e_{6(2)}, \mathfrak f_{4(4)}). $ Let $q_\mathfrak u$ denote the restriction map from $\mathfrak t^\star$ into $\mathfrak u^\star$. Then, for $\lambda$ dominant with respect to $\Psi_{BS}$,  the simple roots for $\Psi_{H,\lambda}=\Psi_{\mathfrak f_{4(4)},\lambda}$,  respectively $\Psi_{\mathfrak{sp}(1,3), \lambda},$   are:
\begin{center}
   $\alpha_2, \,\, \alpha_4, \,\,  q_\mathfrak u(\alpha_3)=q_\mathfrak u (\alpha_5), \,\,  q_\mathfrak u(\alpha_1)=q_\mathfrak u (\alpha_6). $

 $  \beta_1=q_\mathfrak u(\alpha_2 + \alpha_4 +\alpha_5)=q_\mathfrak u (\alpha_2 + \alpha_4 +\alpha_3), \,\, \beta_2=q_\mathfrak u(\alpha_1)=q_\mathfrak u (\alpha_6)$,  \\ $\,\,   \beta_3=q_\mathfrak u(\alpha_3)=q_\mathfrak u (\alpha_5), \,\, \beta_4= \alpha_4.   $
\end{center}

 The fundamental weight $\tilde{\Lambda}_1$ associated to $\beta_1$ is equal to $\frac{1}{2} \beta_{max}$. Hence,     $ \tilde\Lambda_1=\beta_1+\beta_2+  \beta_3 +\frac{1}{2}  \beta_4=\alpha_2+\frac{3}{2}\alpha_4 + \alpha_3+\alpha_5+\frac{1}{2}(\alpha_1+\alpha_6) $.

Thus,   from the Duality Theorem,   for the quaternionic representation

\smallskip
\phantom{xxxxxxxxxxxxxxxxxxxx}$H^2( E_{6(2)}, \pi_{n\frac{\alpha_{max}}{2}+\rho_{SU(6)}}^{SU_2(\alpha_{max})\times SU(6) })$

\medskip

\noindent
  {\it the set of Harish-Chandra parameters of the irreducible $F_{4(4)}$-factors is equal to:

  \medskip

   $-\rho_n^H  $ plus the set of infinitesimal characters of the  $L\equiv SU_2(\alpha_{max})  \times Sp(3) $-irreducible factors for}

\phantom{xxxxxxxxx} $res_{SU_2(\alpha_{max}) \times Sp(3)} (H^2(Sp(1,3),  \pi_{n \frac{\alpha_{max}}{2} + \rho_{Sp(3)}}^{SU_2(\alpha_{max})  \times Sp(3)}))$.

Here, $-\rho_n^H =-d_H \frac{\alpha_{max}}{2} $,  $d_H=d_{\mathfrak f_{4(4)}}=7$ (see \cite{GW}).

\medskip

Therefore, from  the computation in \ref{sub:Ltypesspd1},  we obtain:

\begin{equation*} res_{ F_{4(4)}}( \pi_{n\frac{\alpha_{max}}{2}+\rho_{SU(6)}-\rho_n^G}^{ E_{6(2)}}) =\oplus_{ m\geq 0 }\,\, V_{(n-7+m)\frac{\alpha_{max}}{2} +m\tilde\Lambda_1 +\rho_{Sp(3)}}^{F_{4(4)}}.
 \end{equation*}

Here, $\rho_{Sp(3)}= 3\beta_2 + 5\beta_3 +3\beta_4=\frac{3}{2} (\alpha_1+\alpha_6)+\frac{5}{2} (\alpha_3+\alpha_5)+ 3\alpha_4$.

\subsubsection{Comments on admissible restriction of quaternionic representations} As usual $(G,H)$ is a symmetric pair, $H$ is not a compact  group and $(\pi_\lambda ,H^2(G,\tau))$ is a $H$-admissible, non-holomorphic,  square integrable representation. We further assume $G/K$ as well as $H/L$ holds a quaternionic structure and the inclusion $H/L \hookrightarrow G/K$ respects the   respective quaternionic structures. Then, from Tables 1,2,3 it follows:

a) $ \lambda $ is dominant with respect to a quaternionic system of positive roots. That is, $\pi_\lambda$ is a generalized quaternionic representation.

b)   $H_0/L$ has a quaternionic structure.

c) Each system $\Psi_{H,\lambda}$, $\Psi_{H_0,\lambda}$ is a quaternionic system.

d) The representation $\mathbf H^2(H_0,\tau)$ is a sum of generalized quaternionic representations.

e) When $\pi_\lambda$ is quaternionic, then  the  representation $\mathbf{H}^2(H_0, \tau)$ is equal to $H^2(H_0, res_L(\tau))$, hence, it is quaternionic. Moreover,  in \cite{GW0}  is computed the highest weight and the respective multiplicity of  each of its $L$-irreducible factors.

f) Thus, the duality Theorem~\ref{prop:Rwelldefgc} together with a)---e) let us compute the Harish-Chandra parameters of the irreducible $H$-factors for  a quaternionic representation $\pi_\lambda$. Actually, the computation of the Harish-Chandra parameters is quite similar to the computation in {\it Explicit example I, Explicit example III}.
\medskip

\medskip
In the following paragraph we consider particular quaternionic symmetric pairs. One pair  is $(\mathfrak f_{4(4)}, \mathfrak{so}(5,4))$. Here,  $\mathfrak h_0 \equiv \mathfrak{sp}(1,2)+ \mathfrak{su}(2)$. Thus, for any Harish-Chandra parameter $\lambda$ dominant with respect to the quaternionic system of positive roots, we have $\pi_\lambda$ restricted to $SO(5,4)$ is an admissible representation and the Duality theorem allows us compute either multiplicities or Harish-Chandra parameters of the restriction. Moreover,   since
  quaternionic Discrete Series for $Sp(1,2)\times SU(2)$ are multiplicity free, \cite{GW0}, we have that quaternionic Discrete Series for $\mathfrak f_{4(4)}$, restricted to $SO(5,4)$ are multiplicity free. It seems that it can be deduced from the branching rules for the pair $(Sp(3),Sp(1)\times Sp(2))$ that  a generalized quaternionic representation,  $res_{SO(5,4)}(\pi_\lambda)$ is multiplicity free if and only $\pi_\lambda$ is quaternionic.

\medskip

  For the pair $(\mathfrak f_{4(4)}, \mathfrak{so}(5,4))$, if we attempt to deduce our decomposition result from the work of \cite{GW}, we have to consider the group of Lie algebra $\mathfrak g^d \equiv \mathfrak f_{4(-20)}$,  its maximal compactly embedded subalgebra  is isomorphic to $\mathfrak{so}(9)$, a simple Lie algebra, hence no Discrete Series for $G^d$ has an admissible restriction to $H_0$ (see \cite{KO} \cite{DV}). Thus, it is not clear to us how to deduce our Duality result from the Duality Theorem in  \cite{GW0}.

  \medskip
  For the pairs  $  (\mathfrak e_{6(2)}, \mathfrak{so}(6,4)+\mathfrak{so}(2)),$ $ (\mathfrak e_{7(-5)}, \mathfrak{e}_{6(2)}+\mathfrak{so}(2))$, for each $G$,    generalized quaternionic representations do exist and they are $H$-admissible.  For these pairs, the respective $\mathfrak{h}_0$ are: $ \mathfrak{su}(2,4)+\mathfrak{su}(2), \mathfrak{su}(6,2)$.
In these two cases, the Maple soft developed by  Silva-Vergne\cite{BV}, allows to compute the $L$-Harish-Chandra parameters and respective multiplicity for each Discrete Series for $H_0\equiv SU(p,q)\times SU(r)$, hence, the duality formula yields the Harish-Chandra Parameters for $res_H(\pi_\lambda)$ and their multiplicity.
\subsubsection{Explicit example IV} The pair $(SO(2m,n), SO(2m,n-1))$. This pair is considered in \cite{GW}. We recall their result and we sketch how to derive   the result from our duality Theorem. We only consider the case $\mathfrak g=\mathfrak{so}(2m,2n+1)$. Here, $\mathfrak k=\mathfrak{so}(2m) +\mathfrak{so}(2n+1)$, $\mathfrak h=\mathfrak{so}(2m,2n)$, $\mathfrak h_0=\mathfrak{so}(2m,1)+\mathfrak{so}(2n)$,  $\mathfrak l=\mathfrak{so}(2m)+ \mathfrak{so}(2n)$. We fix a Cartan subalgebra $\mathfrak t \subset \mathfrak l \subset \mathfrak k$. Then, there exists a orthogonal basis
$\epsilon_1, \dots, \epsilon_m, \delta_1, \dots, \delta_n$  for $i\mathfrak t^\star $ so that

 $\Delta = \{(\epsilon_i \pm \epsilon_j), 1 \leq i < j\leq m, (\delta_r \pm \delta_s), 1 \leq r < s \leq n \} \cup \{\delta_j\}_{1\leq j \leq n} .$ $$ \Phi_n = \{ \pm (\epsilon_r \pm \delta_s), r=1, \dots,m, s=1,\dots,n \} \cup \{ \pm \epsilon_j, j=1,\dots,m \}.$$ The systems of positive roots $\Psi_\lambda$ so that $\pi_\lambda^G$ is an admissible representation of $H$ are the systems $\Psi_{\pm}$ associated to the lexicographic orders $ \epsilon_1> \dots> \epsilon_m> \delta_1> \dots> \delta_n $, $  \epsilon_1> \dots>\epsilon_{m-1}> -\epsilon_m> \delta_1> \dots> \delta_{n-1} > -\delta_n. $
Here, for $m \geq 3$, $\mathfrak k_1(\Psi_{\pm})=\mathfrak{so}(2m).$ For $m=2, \mathfrak k_1(\Psi_{\pm})=\mathfrak{su}_2(\epsilon_1 \pm \epsilon_2).$ Then,

$\Psi_{H, +}=  \{(\epsilon_i \pm \epsilon_j), 1 \leq i < j\leq m, (\delta_r \pm \delta_s), 1 \leq r < s \leq n \}\cup  \{  (\epsilon_r \pm \delta_s), r=1, \dots,m, s=1,\dots,n \}$,

 $\Psi_{H_0, +}=  \{(\epsilon_i \pm \epsilon_j), 1 \leq i < j\leq m, (\delta_r \pm \delta_s), 1 \leq r < s \leq n \}\cup \{  \epsilon_j, j=1,\dots,m \}.$

 \smallskip
$\mathfrak g^d=\mathfrak{so}(2m+2n,1)$. Thus, from either our duality Theorem or from \cite{GW}, we infer that whenever  $res_H(\pi_\lambda)$ is $H$-admissible, then, $res_H(\pi_\lambda)$ is a multiplicity free representation.
   Whence, we are left to compute the Harish-Chandra parameters for $res_{SO(2m,2n)}(H^2(SO(2m,2n+1), \pi_{\Lambda_1}^{SO(2m)}\boxtimes \pi_{\Lambda_2}^{SO(2n+1)}))$. For this,    according to the duality Theorem,  we have to compute the infinitesimal characters of  each   irreducible factor of the underlying $L$-module in \\
   $$\mathbf H^2(H_0,\tau)= \sum_{\nu \in Spec_{SO(2n)}(\pi_{\Lambda_2}^{SO(2n+1)})} H^2\big(SO(2m,1), \pi_{\Lambda_1}^{SO(2m)}\big)\boxtimes V_{\nu}^{SO(2n)}$$.    The  branching rules for $ res_{SO(2m)}(H^2(SO(2m,1), \pi_{\Lambda_1}^{SO(2m)}))$ are found in \cite{Th} and  other references, the branching rule for   $res_{SO(2n)}(\pi_{\Lambda_2}^{SO(2n+1)})$  can be found in \cite{Th}. From both computations, we deduce: \cite[Proposition 3]{GW},  for $\lambda=\sum_{1\leq i \leq m} \lambda_i \epsilon_i +\sum_{1\leq j \leq n}  \lambda_{m+j} \delta_j$, then $V_\mu^H$ is a $H$-subrepresentation of $H^2(G,\tau)\equiv V_\lambda^{SO(2m,2n+1)}$ ($\mu=\sum_{1\leq i \leq m} \mu_i \epsilon_i +\sum_{1\leq j \leq n}  \mu_{j+m} \delta_j$) if and only if \\ \phantom{xxxx}$\mu_1>\lambda_1> \dots> \mu_m> \lambda_m, \lambda_{m+1}>\mu_{m+1}>\dots\lambda_{m+n}>\vert \mu_{m+n}\vert.$

\subsection{ Existence of Discrete Series whose  lowest $K$-type restricted to $K_1(\Psi)$ is irreducible} Let $G$ a semisimple Lie group that admits square integrable representations. This hypothesis allows  to fix a compact Cartan subgroup   $T\subset K$ of $G$. In \cite{DV} it is defined for each system of positive roots $\Psi \subset \Phi(\mathfrak g,\mathfrak t)$ a normal subgroup $K_1(\Psi)\subset K$ so that for a symmetric pair $(G,H)$,  with $H$ a $\theta$-invariant subgroup, it holds: for any Harish-Chandra parameter dominant with respect to $\Psi$, the representation $res_H(\pi_\lambda)$ is $H$-admissible if and only if $K_1(\Psi)$ is a subgroup of $H$. For a holomorphic system $\Psi$,   $K_1(\Psi)$ is equal to the center of $K$; for a quaternionic system of positive roots $K_1(\Psi) \equiv SU_2(\alpha_{max})$.   Either for the holomorphic family or for  a quaternionic real forms we find that among the $H$-admissible Discrete Series for     $G$, there are many examples of the following nature: the lowest $K$-type of $\pi_\lambda$ is equal to a irreducible representation of $K_1(\Psi)$ tensor with the trivial representation for $K_2$, \cite{GW}. To follow, under the general setting at the beginning of this paragraph, we verify.

\subsubsection{}\label{sub:existencep}{\it For each  system of positive roots $\Psi \subset \Phi(\mathfrak g,\mathfrak t)$,  there exist Discrete Series with Harish-Chandra parameter dominant with respect to $\Psi$ and so that its lowest $K$-type is equal to a irreducible representation of $K_1(\Psi)$ tensor with the trivial representation for $K_2(\Psi)$.}

We may assume $K_1(\Psi)$ is a proper subgroup of $K$. Then, when $K_1(\Psi)=Z_K$,  Harish-Chandra showed there exists such a representation. For $G$ a quaternionic real form, $\Psi$ a quaternionic system of positive roots,  $K_1(\Psi)=SU_2(\alpha_{max})$, then, in \cite{GW} we find a proof of the statement. From the tables in \cite{DV}\cite{Vaint}, we are left to consider   the triples $(G, K,K_1(\Psi))$ so that their respective Lie algebras are in  the triples

\smallskip

$(\mathfrak{su}( m,  n ), \mathfrak{su}(m )+\mathfrak{su}(n)+\mathfrak u(1), \mathfrak{su}(m ) ), m>2$,

$(\mathfrak{sp}(m, n), \mathfrak{sp}(m)+\mathfrak{sp}(n) ,   \mathfrak{sp}(m ) )$.

$(\mathfrak{so}(2m, n), \mathfrak{so}(2m )+\mathfrak{so}(n ) ,  \mathfrak{so}(2m )) $.

\smallskip

We   analyze  the second triple of the list. To follow, $G$ is so that its Lie algebra is
 $\mathfrak{sp}(m, n)$, $ n\geq 2, m>1$.  We want  to produce Discrete Series representations  so that the lowest $K$-type restricted to $K_1(\Psi)$ is still irreducible. Here, $\mathfrak k= \mathfrak{sp}(m)+\mathfrak{sp}(n)$. We fix maximal torus $T \subset K$ and describe the root system as in \cite{Vaint}. For the system of positive roots   $\Psi :=\{ \epsilon_i \pm \epsilon_j, i<j, \delta_r \pm \delta_s, r<s, \epsilon_a \pm \delta_b, 2\epsilon_a, 2\delta_b, 1\leq a, i,j \leq m, 1\leq b,r,s \leq n\}$,
we have $K_1(\Psi)=K_1 \equiv Sp(m), K_2 (\Psi)=K_2\equiv Sp(n)$. Obviously, there exists a system of positive roots $\tilde \Psi$ so that $K_1(\tilde \Psi) \equiv Sp(n), K_2 (\tilde \Psi)\equiv Sp(m)$. For any other system of positive roots in $\Phi(\mathfrak g,\mathfrak t)$ we have that the associated subgroup $K_1$ is equal to $K$.  It readily follows that $\lambda:=\sum_{1\leq j\leq m} a_j \epsilon_j +\rho_{K_2}$ is a $\Psi$-dominant Harish-Chandra parameter when the coefficients $a_j$  are all integers so that $a_1>\dots >a_m>>0$. Since $\rho_n^\lambda$ belongs to $span_\mathbb C \{\epsilon_1,\dots,\epsilon_m\}$, it follows that the lowest $K$ type of $\pi_\lambda$ is equivalent to a irreducible representation for $Sp(m)$ times the trivial representation for $Sp(n)$.  With the same proof it is verified that the statement holds for the third triple. For the first triple, we further assume $G=SU(p,q)$. Thus, $K$ is the product of two simply connected subgroups times a one dimensional torus $Z_K$,  we notice $\rho_n^{\Psi_a}=\rho_\mathfrak g^\lambda-\rho_K$, hence, $\rho_n^{\Psi_a}$ lifts to a character of $K$. Thus, as in the case $\mathfrak{sp}(m,n)$, we obtain $\pi_\lambda$ with $\lambda$ dominant with respect to $\Psi_a$ so that its lowest    $K=SU(p)SU(q)Z_K$-type  is the tensor product of a irreducible representation for $SU(p)Z_K$ times the trivial representation for $SU(q)$. Since $\rho_n^{\Psi_a}$ lifts to a character of $K$, after some computation the claim follows.

\section{Symmetry breaking operators and  normal derivatives } For this subsection $(G,H)$ is a symmetric pair and $\pi_\lambda$ is a   square integrable representation. Our aim is to generalize a result  in   \cite[Theorem 5.1]{Na}. In \cite{KP2} are considered symmetry breaking operators expressed by means of normal derivatives, they obtain results for holomorphic embedding of a rank one symmetric pairs.   As before, $H_0=G^{\sigma \theta}$ is the associated subgroup. We recall $\mathfrak h\cap \mathfrak p$ is orthogonal to $\mathfrak h_0\cap \mathfrak p$ and that $\mathfrak h\cap \mathfrak p\equiv T_{eL}(H/L)$, $\mathfrak h_0\cap \mathfrak p\equiv T_{eL}(H_0/L)$. Hence, for $X \in \mathfrak h_0\cap \mathfrak p$, more generally for $X \in \mathcal U(\mathfrak h_0)$, we say $R_X$ is a normal derivative  to $H/L$ differential operator. For short, {\it normal derivative}.  Other ingredient necessary for the next Proposition are the subspaces $\mathcal L_\lambda$ and $\mathcal U(\mathfrak h_0)W$. The latter subspace is contained in the subspace of $K$-finite vectors, whereas, according to a well known conjecture,  whenever $res_H(\pi_\lambda)$ is not discretely decomposable, the former subspace is disjoint to the subspace of $G$-smooth vectors. When, $res_H(\pi_\lambda)$ is $H$-admissible $\mathcal L_\lambda$ is contained in the subspace of $K$-finite vectors. However, it might not be equal to $\mathcal U(\mathfrak h_0)W$ as we have pointed out. The next Proposition and its converse, dealt with consequences of the equality $\mathcal L_\lambda = \mathcal U(\mathfrak h_0)W$.
\begin{prop}\label{prop:symmenormal} We assume $(G,H)$ is a symmetric pair.  We also assume there exists a irreducible representation $(\sigma, Z)$ of $L$ so that $H^2(H,\sigma)$ is a irreducible factor of $H^2(G,\tau)$ and  $ H^2(G,\tau)[H^2(H,\sigma)][Z]=\mathcal L_\lambda[Z]= \mathcal U(\mathfrak h_0)W[Z]=L_{\mathcal U(\mathfrak h_0)}(H^2(G,\tau)[W])[Z]$. Then,  $res_H(\pi_\lambda)$ is $H$-admissible. Moreover, any symmetry breaking operator from $H^2(G,\tau)$ into $H^2(H,\sigma)$ is represented by a normal derivative  differential operator.
\end{prop}
We show a converse to   Proposition~\ref{prop:symmenormal} in \ref{sub:convprop}.
\begin{proof} We begin by recalling that  $H^2(G,\tau)[W]=\{ K_\lambda(\cdot,e)^\star w, w\in W\}$ is a subspace of $H^2(G,\tau)_{K-fin}$, hence $L_{\mathcal U(\mathfrak h_0)}(H^2(G,\tau)[W])[Z]$ is a subspace of $H^2(G,\tau)_{K-fin}$. Owing to our hypothesis we then have   $\mathcal L_\lambda[Z]$ is a subspace of $H^2(G,\tau)_{K-fin}$. Next, we quote a
result
	of Harish-Chandra: a $\mathcal U(\mathfrak h)-$finitely generated, $\mathfrak z(\mathcal U(\mathfrak h))-$finite,   module has a finite composition
series. Thus, $ H^2(G,\tau)_{K-fin} $ contains an irreducible $(\mathfrak h,L)$-submodule.  For a  proof  (cf. \cite[Corollary 3.4.7 and Theorem 4.2.1]{Wa1}).  Now, in \cite[Lemma 1.5]{Kob1}  we find a proof of: if a $(\mathfrak g,K)-$mod\-ule contains an irreducible $(\mathfrak h,L)-$submodule, then the $(\mathfrak g,K)-$module is $\mathfrak h-$al\-ge\-bra\-i\-cal\-ly
decomposable. Thus,  $res_H(\pi_\lambda)$ is  al\-ge\-bra\-i\-cal\-ly  discretely decomposable.  In \cite[Theorem 4.2]{Kob}, it is shown that under the hypothesis $(G,H)$ is a symmetric pair, for Discrete Series, $\mathfrak h$-algebraically discrete decomposable is equivalent to $H$-admissibility,  hence $res_H(\pi_\lambda)$ is $H$-admissible. Let $S: H^2(G,\tau)\rightarrow H^2(H,\sigma)=V_\mu^H$ a continuous intertwining linear map. Then, we have shown in \ref{eq:kten}, for $z \in Z$, $K_S(\cdot,e)^\star z\in H^2(G,\tau)[V_\mu^H][Z]$.      We fix a orthonormal basis $\{z_p \},\, p=1,\dots,\dim  Z$ for $Z$. The hypothesis \\ \phantom{ccccccccc} $H^2(G,\tau)[V_\mu^H][Z]=L_ {\mathcal U(\mathfrak h_0)}(H^2(G,\tau)[W])[Z]$\\   implies for each  $p$, there exists $D_p\in \mathcal U(\mathfrak h_0)$ and $w_p \in W$ so that $K_S(\cdot,e)^\star z_p=L_{D_p} K_\lambda(\cdot,e)^\star w_p$. Next, we fix $f_1\in H^2(G,\tau)^\infty, h\in H$ and set $f:=L_{h^{-1}}(f_1)$, then $f(e)=f_1(h)$ and we recall $D^\star $ (resp. $ \check D$) is the formal adjoint of $D\in \mathcal U(\mathfrak g)$, (resp. is the image under the anti homomorphism of $\mathcal U(\mathfrak g)$ that extends minus the identity of $\mathfrak g$).  We have, \begin{equation}\label{eq:xxx}
\begin{split}
(S(f)(e),z_p)_Z & =\int_G(f(y),K_S(y,e)^\star z_p)_W dy \\
& =\int_G(L_{D_p^\star}f(y),K_\lambda(y,e)^\star w_p)_W dy \\
& =(L_{D_p^\star}f(e),w_p)_W \\ & =(R_{\check D_p^\star}f(e),w_p)_W.
\end{split}
\end{equation}
Thus, for each  $z\in Z$ and $f_1$ smooth vector we obtain
\begin{center} $(S(f_1)(h),z)_Z=\sum_p (S(f_1)(h),(z,z_p)_Z z_p)_Z= \sum_p   (R_{\check D_p^\star}f_1(h),w_p)_W (z_p,z)_Z  $. \end{center}  As in \cite[Proof of Lemma 2]{OV2} we conclude that for any $f \in H^2(G,\tau)$    \begin{equation} S(f)(h)= \sum_{1\leq p \leq  \dim  Z} ( R_{\check D_p^\star}f(h),w_p)_W \,\, z_p. \end{equation} Since $D_p \in \mathcal U(\mathfrak h_0)$   such a expression of $S(f)$ is an representation  in terms of normal derivatives. \end{proof}

\subsubsection{Converse to Proposition~\ref{prop:symmenormal} }\label{sub:convprop} We want to show: If every element
 in $Hom_H(H^2(G,\tau), H^2(H,\sigma)) $ has a expression as differential operator by means of "normal derivatives",  then, the equality  $\mathcal L_\lambda[Z]= H^2(G,\tau)[H^2(H,\sigma)][Z]=\mathcal U(\mathfrak h_0)W[Z]$ holds.

In fact, the hypothesis  $ S(f)(h)= \sum_{1\leq p \leq  \dim  Z} ( R_{\check D_p^\star}f(h),w_p)_W \,\, z_p$, $D_p \in \mathcal U(\mathfrak h_0)$,   yields $K_S(\cdot,e)^\star z=L_{D_z} K_\lambda (\cdot,e)^\star w_z$, $D_z \in \mathcal U(\mathfrak h_0)$, $w_z \in W$. The fact that $(\sigma, Z)$ has multiplicity one in $H^2(H,\sigma)$  gives \\ \phantom{xx} $\dim  Hom_H(H^2(G,\tau), H^2(H,\sigma))=\dim  Hom_L(Z, H^2(G,\tau)[H^2(H,\sigma)][Z])$. \\ Hence,    the functions\\ \phantom{xxxxxx} $\{K_S(\cdot,e)^\star z, z \in Z, S \in Hom_H(H^2(G,\tau), H^2(H,\sigma))\}$ \\ span $H^2(G,\tau)[H^2(H,\sigma)][Z]$.   Therefore, $H^2(G,\tau)[H^2(H,\sigma)][Z]$ is contained in $\mathcal U(\mathfrak h_0)W[Z]=L_{\mathcal U(\mathfrak h_0)}H^2(G,\tau)[W][Z]$. Owing to Theorem~\ref{prop:Rwelldefgc},  both spaces have the same dimension, whence,  the equality holds.

\medskip
The pairs so that Proposition~\ref{prop:symmenormal}\footnote{In work in progress we have shown the Proposition holds for $(\mathfrak{sp}(m,1), \mathfrak{sp}(m-1,1)+\mathfrak{sp}(1))$ and a quaternionic representation.} holds for scalar holomorphic Discrete Series are
$ (\mathfrak{su}(m,n), \mathfrak{su} (m,l) +\mathfrak{su} ( n-l)+\mathfrak u(1)),$ $(\mathfrak{so}(2m,2), \mathfrak u(m,1)),$  $(\mathfrak{so}^\star (2n), \mathfrak u(1,n-1)),$ $(\mathfrak{so}^\star(2n),  \mathfrak{so}(2) +\mathfrak{so}^\star(2n-2)),$ $(\mathfrak e_{6(-14)}, \mathfrak{so}(2,8)+\mathfrak{so}(2)).$ See \cite[(4.6)]{Vaint}.

\subsubsection{Comments on  the interplay  among   the subspaces, $\mathcal L_\lambda$, $\mathcal U(\mathfrak h_0)W$, $H^2(G,\tau)_{K-fin}$ and symmetry breaking operators} It readily follows that the subspace $\mathcal L_\lambda[Z]=V_\lambda^G[H^2(H, \sigma)][Z]$ is equal to the closure of  the linear span of

  $\mathcal K_{Sy}(G,H):=\{K_{S^\star}(e,\cdot)z=K_{S}(\cdot,e)^\star z, z\in Z, S \in Hom_H(V_\lambda^G, V_\mu^H) \}$.

 \smallskip
 (1)\,  $ H^2(G,\tau)_{K-fin} \cap \mathcal L_\lambda [Z] $
 is equal to the linear span of  the elements in $\mathcal K_{Sy}(G,H) $    so that  the corresponding symmetry breaking operator  is  represented   by   a differential   operator. See \cite[Lemma 4.2]{OV2}.
 \smallskip

(2)  \, $ \mathcal U(\mathfrak h_0)W \cap  \mathcal L_\lambda [Z]$
 is equal to the linear span corresponding to the elements   $ K_S $ in $\mathcal K_{Sy}(G,H) $ so  that $S$ is  represented   by   normal   derivative    differential,  operator.   This is shown in Proposition~\ref{prop:symmenormal} and its converse.

 \smallskip
(3)    The set of  symmetry breaking operators represented   by  a differential   operator  is not the null space if and only if $res_H(\pi_\lambda)$ is $H$-discretely decomposable. See \cite[Theorem 4.3]{OV2} and the proof of Proposition~\ref{prop:symmenormal}.

\smallskip

(4) We believe that from Nakahama's thesis, it is possible to construct examples of   $V_\lambda^G[H^2(H, \sigma)][Z]\cap \mathcal U(\mathfrak h_0)W[Z]\not= \{0\}$, so that the equality $V_\lambda^G[H^2(H, \sigma)][Z]= \mathcal U(\mathfrak h_0)W[Z]$ does not hold! That is, there are symmetry breaking operators  represented by  plain differential operators and some of the operators are not represented by normal derivative   operators.

\subsubsection{A functional equation for symmetry breaking operators}   Notation is as in Theorem~\ref{prop:Rwelldefgc}.  We assume $(G,H)$ is a symmetric pair and $res_H(\pi_\lambda)$ is admissible. The objects involved in  the equation are:  $H_0=G^{\sigma \theta}$, $Z=V_{\mu+\rho_n^H}^L$ the lowest $L$-type for $V_\mu^H$, $\mathcal L_\lambda=\sum_{\mu} H^2(G,\tau)[V_\mu^H][V_{\mu+\rho_n^H}^L]$, $\mathcal U(\mathfrak h_0)W =L_{\mathcal U(\mathfrak h_0) }H^2(G,\tau)[W]$,     $L$-isomorphism $D: \mathcal L_\lambda [Z] \rightarrow \mathcal U(\mathfrak h_0)W[Z]$, a $H$-equivariant continuous linear map $S: H^2(G,\tau)\rightarrow H^2(H,\sigma)$, the kernel  $K_S :G\times H\rightarrow Hom_\mathbb C(W,Z)$ corresponding to $S$, \ref{eq:kten} implies $K_S(\cdot,e)^\star z \in \mathcal L_\lambda[Z]$, finally, we recall $K_\lambda :G\times G\rightarrow Hom_\mathbb C(W,W)$ the kernel associated to the orthogonal projector onto $ H^2(G,\tau)$. Then,
\begin{prop}\label{prop:funcequasymmt} For $z\in Z$, $y\in G$ we have \begin{equation*} D(K_S(e,\cdot)^\star (z))(y)=\frac{1}{c}\int_{H_0}K_\lambda(h_0,y) D(K_S(e,\cdot)^\star (z))(h_0) dh_0. \end{equation*}
\end{prop}
Here, $c=d(\pi_\lambda) \dim  W/ d(\pi_{\eta_0}^{H_0}).$
 When, $D$ is the identity map, the functional equation turns into
 \begin{equation*}  K_S(x,h)    =\frac{1}{c}\int_{H_0}  K_S(  h_0,e) K_\lambda(x,hh_0) dh_0 \end{equation*}

The   functional equation follows from Proposition~\ref{prop:ktfromkto} applied to $T:=S^\star$. The second equation follows after we compute the adjoint of the first equation.

We note, that as in the case of holographic operators, a symmetry breaking operator can be recovered from its restriction to $H_0$.

We also note that \cite{Na} has shown a different functional equation for $K_S$ for scalar holomorphic Discrete Series and holomorphic embedding $H/L \rightarrow G/K$.

 \section{Tables}

 For an arbitrary symmetric pair $(G,H),$ whenever $\pi_\lambda^G$ is an admissible representation of $H,$  we define,
  $$ K_1= \left\{ \begin{array}{ll} Z_K & \mbox{\, if $\Psi_\lambda$ holomorphic } \\ K_1(\Psi_\lambda) & \phantom{x} \mbox{otherwise} \end{array} \right. $$
In the next tables we present the 5-tuple satisfying: $(G, H)$ is a symmetric pair, $H_0$ is the associated group to $H,$ $\Psi_\lambda$ is a  system of positive roots such that $\pi_\lambda^G$ is an admissible representation of $H,$ and $K_1 =Z_1( \Psi_\lambda) K_1(\Psi_\lambda).$  Actually, instead of writing  Lie groups we write their respective Lie algebras. Each  table is in part a reproduction of tables in \cite{KO} \cite{Vaint}. The tables  can also be computed by means of the techniques presented in \cite{DV}. Note that each table is "symmetric" when we replace $H$ by $H_0.$ As usual, $\alpha_{max}$ denotes the highest root in $\Psi_\lambda.$    Unexplained notation is as in \cite{Vaint}.


\rotatebox{90}{
\begin{tabular}{|c| c | c| c |c |} \hline  $G$  &  $H$  & $H_0$  &  $\Psi_\lambda$  &  $ K_1$  \\
\hline  $\mathfrak{su}(m,n)$  &  $\mathfrak{su}(m,k)\oplus \mathfrak{su}(n-k)\oplus \mathfrak{u}(1)$  &  $\mathfrak{su}(m,n-k)\oplus \mathfrak{su}(k)\oplus \mathfrak{u}(1)$   &  $\Psi_a$  &  $ \mathfrak{su}(m)$  \\
\hline  $\mathfrak{su}(m,n)$  &  $\mathfrak{su}(k,n)\oplus \mathfrak{su}(m-k)\oplus \mathfrak{u}(1)$  &  $\mathfrak{su}(m-k,n)\oplus \mathfrak{su}(k)\oplus \mathfrak{u}(1)$   &  $\tilde{\Psi_b}$  &  $ \mathfrak{su}(n)$  \\
\hline   $\mathfrak{so}(2m,2n), m>2$   &   $\mathfrak{so}(2m,2k)\oplus \mathfrak{so}(2n-2k)$ & $\mathfrak{so}(2m,2n-2k)\oplus \mathfrak{so}(2k)$   &   $\Psi_{\pm}$  & $ \mathfrak{so}(2m)$   \\
\hline   $\mathfrak{so}(4,2n) $   &   $\mathfrak{so}(4,2k)\oplus \mathfrak{so}(2n-2k)$ & $\mathfrak{so}(4,2n-2k)\oplus \mathfrak{so}(2k)$   &   $\Psi_{\pm}$  & $ \mathfrak{su}_2(\alpha_{max})$   \\
\hline   $\mathfrak{so}(2m,2n+1), m>2$   &   $\mathfrak{so}(2m,k)\oplus \mathfrak{so}(2n+1-k)$ &   $\mathfrak{so}(2m,2n+1-k)\oplus \mathfrak{so}(k)$   &   $\Psi_{\pm}$   &   $ \mathfrak{so}(2m)$   \\
\hline   $\mathfrak{so}(4,2n+1) $   &   $\mathfrak{so}(4,k)\oplus \mathfrak{so}(2n+1-k)$ &   $\mathfrak{so}(4,2n+1-k)\oplus \mathfrak{so}(k)$   &   $\Psi_{\pm}$   &   $ \mathfrak{su}_2(\alpha_{max})$ \\ \hline   $\mathfrak{so}(4,2n),  n>2$   &   $\mathfrak{u}(2,n)_1$   &   $w\mathfrak{u}(2,n)_1$   &   $\Psi_{1\, -1}$   &   $ \mathfrak{su}_2(\alpha_{max})$   \\
\hline   $\mathfrak{so}(4,2n),  n>2$   &   $\mathfrak{u}(2,n)_2$   &   $w\mathfrak{u}(2,n)_2$   &   $\Psi_{1 \, 1}$   &   $ \mathfrak{su}_2(\alpha_{max})$   \\
\hline   $\mathfrak{so}(4,4)$   &   $\mathfrak{u}(2,2)_{1\,1}$   &   $w\mathfrak{u}(2,2)_{11}$   &   $\Psi_{1\, -1},  \, w_{\epsilon,\delta}\Psi_{1 \, -1}$   &   $ \mathfrak{su}_2(\alpha_{max})$ \\
\hline
$\mathfrak{so}(4,4)$   &   $\mathfrak{u}(2,2)_{12}$   &   $w\mathfrak{u}(2,2)_{12}$   &   $\Psi_{1\, -1},  \, w_{\epsilon,\delta}\Psi_{1\, 1}$   &    $ \mathfrak{su}_2(\alpha_{max})$ \\
\hline   $\mathfrak{so}(4,4)$   &   $\mathfrak{u}(2,2)_{21}$   &   $w\mathfrak{u}(2,2)_{21}$   &   $\Psi_{1\, 1}, \,  w_{\epsilon,\delta}\Psi_{1\, -1}$   &   $ \mathfrak{su}_2(\alpha_{max})$ \\
\hline   $\mathfrak{so}(4,4)$   &   $\mathfrak{u}(2,2)_{22}$   &   $w\mathfrak{u}(2,2)_{22}$   &   $\Psi_{1\,1}, \, w_{\epsilon,\delta}\Psi_{1\,1}$   &   $ \mathfrak{su}_2(\alpha_{max})$ \\
\hline
  $\mathfrak{sp}(m,n)$   &   $\mathfrak{sp}(m,k)\oplus \mathfrak{sp}(n-k)$   &   $\mathfrak{sp}(m,n-k)\oplus \mathfrak{sp}(k)$   &   $\Psi_+$   &   $ \mathfrak{sp}(m)$   \\
\hline   $\mathfrak f_{4(4)}$   &   $\mathfrak{sp}(1,2)\oplus \mathfrak{su}(2)$ & $\mathfrak{so}(5,4)$   &   $\Psi_{BS}$   &   $ \mathfrak{su}_2(\alpha_{max})$   \\
\hline
\phantom{ } $\mathfrak{e}_{6(2)}$   &   $\mathfrak{so}(6,4)\oplus \mathfrak{so}(2)$ & $\mathfrak{su}(4,2)\oplus \mathfrak{su}(2)$   &   $\Psi_{BS}$   &   $ \mathfrak{su}_2(\alpha_{max})$   \\
\hline
  $\mathfrak{e}_{7(-5)}$   &   $ \mathfrak{so}(8,4)\oplus \mathfrak{su}(2)$ &   $\mathfrak{so}(8,4)\oplus \mathfrak{su}(2)$   &   $\Psi_{BS}$ & $ \mathfrak{su}_2(\alpha_{max})$   \\
\hline
  $\mathfrak{e}_{7(-5)}$   &   $\mathfrak{su}(6,2)$   &   $\mathfrak{e}_{6(2)}\oplus \mathfrak{so}(2)$   &   $\Psi_{BS}$   &   $ \mathfrak{su}_2(\alpha_{max})$   \\
\hline
  $\mathfrak{e}_{8(-24)}$   &   $\mathfrak{so}(12,4)$   &   $\mathfrak{e}_{7(-5)}\oplus \mathfrak{su}(2)$   &   $\Psi_{BS}$   &   $ \mathfrak{su}_2(\alpha_{max})$   \\
\hline
\end{tabular} } \\
Table 1. Case $U=T, \Psi_\lambda$ non holomorphic.

\rotatebox{90}{  \begin{tabular}{|c| c | c| c |c |} \hline  $G$    &  $H$  &  $H_0$  & $\Psi_\lambda$  &  $ K_1$  \\
\hline  $\mathfrak {su}(2,2n), \, n> 2$  & $\mathfrak {sp}(1,n)$  &  $\mathfrak {sp}(1,n)$  &  $\Psi_1$  &  $ \mathfrak{su}_2(\alpha_{max})$  \\
\hline  $\mathfrak{su}(2,2)$  & \phantom{xx} $\mathfrak{sp}(1,1)$ &  $\mathfrak{sp}(1,1)$ &  $\Psi_1$  &  $ \mathfrak{su}_2(\alpha_{max})$  \\
\hline
   $\mathfrak{su}(2,2)$ &  $\mathfrak{sp}(1,1)$  &  $\mathfrak{sp}(1,1)$  &  $\tilde{\Psi}_1$  & $ \mathfrak{su}_2(\alpha_{max})$  \\
\hline
  $\mathfrak{so}(2m,2n), m>2$  &  $ \mathfrak{so}(2m,2k+1) + \mathfrak{so}(2n-2k-1)$  &  $\mathfrak{so}(2m,2n-2k-1)+ \mathfrak{so}(2k+1)$  &  $\Psi_{\pm}$  &  $ \mathfrak{so}(2m)$ \\
  \hline
  $\mathfrak{so}(4,2n), $  &  $ \mathfrak{so}(4,2k+1) + \mathfrak{so}(2n-2k-1)$  &  $\mathfrak{so}(4,2n-2k-1)+ \mathfrak{so}(2k+1)$  &  $\Psi_{\pm}$  &  $ \mathfrak{su}_2(\alpha_{max})$ \\
  \hline
  $\mathfrak{so}(2m,2), m>2 $  &  $ \mathfrak{so}(2m,1) $  &  $\mathfrak{so}(2m,1)$  &  $\Psi_{\pm}$  &  $ \mathfrak{so}(2m)$ \\
  \hline
  $\mathfrak{so}(4,2),  $  &  $ \mathfrak{so}(4,1) $  &  $\mathfrak{so}(4,1)$  &  $\Psi_{\pm}$  &  $ \mathfrak{su}_2(\alpha_{max})$ \\
\hline
 $\mathfrak{e}_{6(2)}$  &  $\mathfrak{f}_{4(4)}$  &  $\mathfrak{sp}(3,1)$  &  $\Psi_{BS}$  &  $ \mathfrak{su}_2(\alpha_{max})$  \\
\hline
\end{tabular} } \\

Table 2, Case $U\not=T, \Psi_\lambda$ non holomorphic.

\begin{center}
\begin{tabular}{|c| c | c|}
\hline   $G$  &   $H \, $ (a)  &  $H_0 \,   $  (b)   \\
 \hline   $\mathfrak{su}(m,n), m\not= n$   &     $\mathfrak{su}(k,l)+\mathfrak{su}(m-k,n-l)+ \mathfrak u(1)$ & $\mathfrak{su}(k,n-l)+\mathfrak{su}(m-k,l)+ \mathfrak u(1)$    \\
 \hline   $\mathfrak{su}(n,n)$  & $\mathfrak{su}(k,l)+ \mathfrak{su}(n-k,n-l)+\mathfrak u(1)$ &$\mathfrak{su}(k,n-l)+ \mathfrak{su}(n-k,l)+\mathfrak u(1)$       \\
 \hline   $\mathfrak{so}(2,2n)$  & $\mathfrak{so}(2,2k)+ \mathfrak{so}(2n-2k)$   & $\mathfrak{so}(2,2n-2k)+ \mathfrak{so}(2k)$      \\
\hline   $\mathfrak{so}(2,2
n)$  & $\mathfrak{u}(1,n)$ & $ \mathfrak{u}(1,n)$        \\
\hline   $\mathfrak{so}(2,2n+1)$  & $\mathfrak{so}(2,k)+ \mathfrak{so}(2n+1-k)$ & $\mathfrak{so}(2,2n+1-k)+ \mathfrak{so}(k)$       \\
\hline   $\mathfrak{so}^\star (2n)$  & $\mathfrak{u}(m,n-m)$ & $\mathfrak{so}^\star(2m)+ \mathfrak{so}^\star(2n-2m)$       \\
\hline   $\mathfrak{sp}(n, \mathbb R)$  & $\mathfrak{u}(m,n-m)$ & $\mathfrak{sp}(m, \mathbb R)+ \mathfrak{sp}(n-m, \mathbb R)$       \\
\hline $\mathfrak e_{6(-14)}$ & $\mathfrak{so}(2,8)+\mathfrak{so}(2) $ & $\mathfrak{so}(2,8)+\mathfrak{so}(2)$ \\
\hline $\mathfrak e_{6(-14)}$ & $\mathfrak{su}(2,4)+\mathfrak{su}(2) $ & $\mathfrak{su}(2,4)+\mathfrak{su}(2)$ \\
\hline $\mathfrak e_{6(-14)}$ & $\mathfrak{so}^\star(10)+\mathfrak{so}(2) $ & $\mathfrak{su}(5,1)+\mathfrak{sl}(2, \mathbb R)$ \\
\hline $\mathfrak e_{7(-25)}$ & $\mathfrak{so}^\star(12)+\mathfrak{su}(2) $ & $\mathfrak{su}(6,2)$ \\
\hline $\mathfrak e_{7(-25)}$ & $\mathfrak{so}(2,10)+\mathfrak{sl}(2, \mathbb R) $ & $\mathfrak e_{6(-14)} + \mathfrak{so}(2)$ \\
\hline   $\mathfrak{su}(n,n)$  & $\mathfrak{so}^\star(2n)$ &$\mathfrak{sp}(n,\mathbb R)$       \\
\hline   $\mathfrak{so}(2,2n)$  & $\mathfrak{so}(2,2k+1)+ \mathfrak{so}(2n-2k-1)$   & $\mathfrak{so}(2,2n-2k-1)+ \mathfrak{so}(2k+1)$      \\
\hline
\end{tabular} \\
Table 3, $\pi_\lambda^G $ holomorphic Discrete Series. \\
The last two lines show the unique holomorphic pairs so that $U \not= T.$
\end{center}

\section{Partial list of symbols and definitions}

\noindent
- $(\tau ,W),$ $(\sigma, Z)$, $L^2(G \times_\tau W), L^2(H\times_\sigma Z)$ (cf. Section~\ref{sec:prelim}).\\
-$\,H^2(G,\tau)=V_\lambda=V_\lambda^G $, $H^2(H,\sigma)=V_\mu^H, \pi_\mu^H, \pi_\nu^K.  $ (cf. Section~\ref{sec:prelim}).\\
-$\pi_\lambda =\pi_\lambda^G$, $d_\lambda=d(\pi_\lambda)$ formal degree of $\pi_\lambda,$   $P_\lambda, P_\mu,  K_\lambda,  K_\mu,  $  (cf. Section~\ref{sec:prelim}). \\
-$P_X$ orthogonal projector onto subspace $X$.\\
-$\Phi(x)=P_W \pi (x)P_W$ spherical function attached to the lowest $K$-type $W$ of $\pi_\lambda$.\\
-$K_\lambda(y,x)=d(\pi_\lambda)\Phi(x^{-1}y)$.\\
-$M_{K-fin} (resp. M^\infty) $ $K-$finite vectors in $M$ (resp. smooth vectors in $M$).\\
-$dg,dh$ Haar measures on $G$, $H$.\\
-A unitary representation is {\it square integrable}, equivalently a {\it Discrete Series} representation,  (resp. {\it integrable}) if some nonzero  matrix coefficient is square integrable (resp. integrable) with respect to Haar measure on the group in question. \\
-$\Theta_{\pi_\mu^H}(...)$  Harish-Chandra character of the representation $\pi_\mu^H.$\\
-For a module $M$ and a simple submodule $N$, $M[N]$ denotes the {\it isotypic component} of $N$ in $M$. That is, $M[N]$ is the sum of all irreducible submodules isomorphic to $N.$ If topology is involved, we define $M[N]$ to be the closure of $M[N].$ \\
-$  M_{H-disc}$ is the closure of the linear subspace spanned by the totality of $H-$irreducible submodules. $ M_{disc}:= M_{G-disc}$\\
-A representation $M$ is $H-${\it discretely decomposable} if $ M_{H-disc} =M.$\\
-A representation is $H-${\it admissible} if it is $H-$discretely decomposable and each isotypic component is equal to a finite sum of $H-$irreducible representations.\\
-$\mathcal U(\mathfrak g) $ (resp. $\mathfrak z(\mathcal U(\mathfrak g))=\mathfrak z_\mathfrak g$) universal enveloping algebra of the Lie algebra $\mathfrak g$(resp. center of universal enveloping algebra).\\
-$\mathrm{Cl}(X)= $closure of the set $X$.\\
-$I_X$ identity function on set $X$.\\
-$\mathbb T$ one dimensional torus.\\
-$Z_S$ identity connected component of the center of the group $S$.\\
$S^{(r)}(V)$ the $r^{th}$-symmetric power of the vector space $V$.

\end{document}